\documentclass{amsart}
\usepackage[latin1]{inputenc}
\usepackage[active]{srcltx}
\usepackage{fullpage}
\usepackage{amsfonts,enumerate}
\usepackage[mathscr]{euscript}
\usepackage{epsfig}
\usepackage{latexsym}
\usepackage[all]{xy}
\usepackage{amssymb}
\usepackage{amsthm}
\usepackage{amsmath}
\usepackage{amsxtra}
\usepackage{xcolor}
\usepackage[colorlinks]{hyperref}
\hypersetup{citecolor=blue}

.tfm
.tfm

\theoremstyle{plain}
\newtheorem{prop}{Proposition}[section]
\newtheorem{thm}[prop]{Theorem}
\newtheorem{coro}[prop]{Corollary}

\newtheorem{lemma}[prop]{Lemma}

\theoremstyle{definition}
\newtheorem{defi}[prop]{Definition}

\theoremstyle{remark}

\newtheorem{remark}{Remark}

\numberwithin{table}{section}

\DeclareMathOperator{\cond}{cond}
\DeclareMathOperator{\Frob}{Frob}
\DeclareMathOperator{\Id}{Id}
\DeclareMathOperator{\lcm}{lcm}

\DeclareMathOperator{\Cl}{Cl}
\DeclareMathOperator{\Gal}{Gal}

\DeclareMathOperator{\Ind}{Ind}

\newcommand{\F}{\mathbb F}

\newcommand{\Om}{{\mathscr{O}}}

\newcommand{\Disc}{\Delta}

\newcommand{\GL}{{\rm GL}}
\newcommand{\SL}{{\rm SL}}

\newcommand{\E}{E_{(a,b,c)}}
\newcommand{\Et}{\widetilde{E}_{(a,b,c)}}

\newcommand{\cyclo}{\chi_{\text{cyc}}}

\makeatletter
\newcommand{\verbatimfont}[1]{\renewcommand{\verbatim@font}{\ttfamily#1}}
\makeatother

\def\ZZ{\mathbb Z}

\def\RR{\mathbb R}
\def\FF{\mathcal F}
\def\QQ{\mathbb Q}
\def\II{\mathbb I}

\def\AA{\mathbb A}
\def\CC{\mathbb C}

\def\<#1>{{\left\langle{#1}\right\rangle}}

\def\Z{{\mathbb Z}}             
\def\Q{{\mathbb Q}}             

\def\id#1{{\mathfrak{#1}}}      

\DeclareMathOperator{\norm}{{\mathscr N}}

\DeclareMathOperator{\trace}{{\mathrm{Tr}}}

\newcommand{\lmfdbec}[3]{\href{http://www.lmfdb.org/EllipticCurve/Q/#1/#2/#3}{{\text{\rm#1-#2#3}}}}
\newcommand{\lmfdbecnf}[4]{\href{http://www.lmfdb.org/EllipticCurve/#1/#2/#3/#4}{{\text{\rm#1-#2-#3#4}}}}

\newcommand{\lmfdbbmf}[2]{\href{http://www.lmfdb.org/ModularForm/GL2/ImaginaryQuadratic/#1/#2}{{\text{\rm#1-#2}}}}

\let\kro\dkro

\begin{document}

\title{$\Q$-curves, Hecke characters and some Diophantine equations.}

\author{Ariel Pacetti}
\address{Center for Research and Development in Mathematics and Applications (CIDMA),
	Department of Mathematics, University of Aveiro, 3810-193 Aveiro, Portugal}
\email{apacetti@ua.pt}
\thanks{AP was partially supported by FonCyT BID-PICT 2018-02073 and by
the Portuguese Foundation for Science and Technology (FCT) within
project UIDB/04106/2020 (CIDMA). LVT was supported by a CONICET grant.}

\author{Lucas Villagra Torcomian}
\address{FAMAF-CIEM, Universidad Nacional de
  C\'ordoba. C.P:5000, C\'ordoba, Argentina.}
\email{lucas.villagra@unc.edu.ar}

\keywords{$\Q$-curves, Fermat equations}
\subjclass[2010]{11D41,11F80}

\begin{abstract}
  In this article we study the equations $x^4+dy^2=z^p$ and
  $x^2+dy^6=z^p$ for positive square-free values of $d$. A Frey curve over
  $\Q(\sqrt{-d})$ is attached to each primitive solution, which
  happens to be a $\Q$-curve.  Our main result is the construction of
  a Hecke character $\chi$ satisfying that the Frey elliptic curve
  representation twisted by $\chi$ extends to $\Gal_\Q$, therefore (by
  Serre's conjectures) corresponds to a newform in
  $S_2(n,\varepsilon)$ for explicit values of $n$ and
  $\varepsilon$. Following some well known results and elimination
  techniques (together with some improvements) it provides a
  systematic procedure to study solutions of the above equations and
  allows us to prove non-existence of non-trivial primitive solutions
  for large values of $p$ of both equations for new values of $d$.
\end{abstract}

\maketitle


\section*{Introduction}
Since Wiles' proof of Fermat's last theorem, there has been an
increasing interest in solving different Diophantine equations. An
open challenging problem is to understand solutions of a generalized
Fermat type equation of the form
\begin{equation}
\label{eq:generalizedfermat}
Ax^p+By^q=Cz^r.
\end{equation}
In \cite{MR1348707} it was proven that for each triple of exponents
$(p,q,r)$ satisfying $\frac{1}{p}+\frac{1}{q}+\frac{1}{r}<1$, the set
of integral primitive solutions is finite (i.e. the surface has
finitely many integral points). Recall that a solution $(a,b,c)$ to
(\ref{eq:generalizedfermat}) is called \emph{primitive} if the numbers
$\{aA,bB,cC\}$ are pairwise coprime. As explained in \cite{MR1468926},
there are many instances where equation~(\ref{eq:generalizedfermat})
posses infinitely many non-primitive solutions, one of the reason to
restrict to primitive ones.

Nowadays there exists a sort of ``guideline'' to study solutions of
equation~(\ref{eq:generalizedfermat}). A short version of the
procedure (more details will be given in
Section~\ref{section:generalstrategy}) is the following: firstly
attach to a non-trivial primitive solution an odd, two-dimensional
Galois representations (with values in $\overline{\F_p}$ or in
$\overline{\Q_p}$) with very special properties, for example that its
conductor is only divisible by the primes dividing $ABC$ (or its
reduction does). The next step is to prove modularity of the
representation (which in the case of rational representations follows
mostly by Serre's conjectures), i.e. prove that such representation matches
that of a newforms of ``known'' level and Nebentypus, independent of
the solution (in many circumstances this goal is achieved using some
result like Ribet's lowering the level one). At last, one
computes the particular spaces of modular forms and use different
elimination techniques (some will be explained in
Section~\ref{section:methostodiscard}) aiming to prove that all
computed newforms are not related to solutions, hence solutions cannot
exist.

The present article started as an attempt to study particular cases of equation~(\ref{eq:generalizedfermat}), namely equation
\begin{equation}
  \label{eq:24p}
x^4+dy^2=z^p,
\end{equation}
and equation
\begin{equation}
  \label{eq:ben-chen}
  x^2+dy^6=z^p,
\end{equation}
for positive square-free values of $d$.  For such equations, a
solution is called \emph{trivial} if one if its coordinate is zero,
for example the solutions $(\pm 1,0,1)$ are trivial solutions of both
(\ref{eq:24p}) and (\ref{eq:ben-chen}) for all values of $p$.  The
first equation was studied in \cite{MR2075481} for $d=1$ and in
\cite{MR2561200} for $d=2,3$.  The articles \cite{MR2075481} and
\cite{MR2561200} proved what might be called ``asymptotic results'',
namely the existence of a constant $N_d$ such that all non-trivial
primitive solutions of equation~(\ref{eq:24p}) have exponent
$p \le N_d$.  Explicitly, the constants are $N_d = 211, 349, 131$ for
$d=1,2,3$ respectively. The main contribution of \cite{MR2646760} was
to extend the result for small values of $p$, proving non-existence of
non-trivial primitive solutions for exponents $n \ge 4$ (not
necessarily prime) when $d=1$ and the existence of a unique
non-trivial primitive solution for exponents $n \ge 4$ when $d=2$.

Equation~(\ref{eq:ben-chen}) was studied in \cite{MR2966716} for $d=1$
and in \cite{Angelos} for $d=3$, where non-existence of primitive
non-trivial solutions was proved for all exponents $n \ge 3$ when
$d=1$, while there exist a unique primitive solution for exponents $n \ge 4$ when
$d=3$. The crucial connection between equations (\ref{eq:24p}) and
(\ref{eq:ben-chen}) is that in both cases, to a putative primitive
non-trivial solution $(a,b,c)$ one attaches an elliptic curve over the
quadratic extension $K=\Q(\sqrt{-d})$ that has the property of being a
$\Q$-curve (i.e. an elliptic curve $E$ whose Galois conjugates are
isogenous to $E$), fulfilling the first step of the general strategy
shortly described before.  In the case of equation~(\ref{eq:24p}), the
representation comes from the elliptic curve (proposed in
\cite{MR2561200}) given by the equation
\begin{equation}
\label{eq:freycurve}
  \E: y^2=x^3+4ax^2+2(a^2+\sqrt{-d}b)x.
\end{equation}
In the case of equation~(\ref{eq:ben-chen}), in \cite{MR2966716} and
\cite{Angelos} the authors attach to a solution $(a,b,c)$ a $\Q$-curve
with a $3$-torsion point. Generalizing their ideas (and using the
description of curves with a $3$-torsion point given by Kubert
(\cite{Kubert}) we attach to a solution $(a,b,c)$ of
(\ref{eq:ben-chen}) the $\Q$-curve
\begin{equation}
\label{eq:FC-BC}
  \Et :y^2+6b\sqrt{-d}xy-4d(a+b^3\sqrt{-d})y = x^3,
\end{equation}
where $(0,0)$ is its natural rational point of order $3$. Part of our
contribution to study solutions of (\ref{eq:ben-chen}) lies in the
fact that such an elliptic curve has a Galois representation
fulfilling the requirements of the general strategy (as will be proved
in Section~\ref{section:properties}).

A key property of $\Q$-curves is that their Galois representations can
be ``extended'' to the whole Galois group, providing the second result
of the general strategy. More concretely, a result of Ribet
(\cite{MR2058653}) implies that if $E/K$ is a $\Q$-curve then there
exists a character $\chi$ such that the twisted Galois representation
$\rho_{E,p} \otimes \chi$ extends to the whole Galois group
$\Gal_\Q$. Ribet's result is not explicit, as it depends on
trivializing a cocycle naturally attached to the $\Q$-curve $E$.  In
the aforementioned articles, the way the cocycle was trivialized was using 
an algorithm due to Quer (\cite{MR1770611}) which gives an ad-hoc
element (via Hilbert's 90 theorem) after a tedious search for it. The
disadvantage of such an approach is that a priori there is no control
on the ramification of the character nor a clear description of the
character itself.

One of the main contributions of the present article is to provide a
(computable) alternative to Ribet's approach.  Namely, we give an
explicit description of a Hecke character $\chi$ such that the Galois
representation $\rho_{\E,p}\otimes \chi$ extends to the whole Galois
group $\Gal_\Q$ (and a respective result for the representation
$\rho_{\Et,p}$). To explain our approach, let us introduce some
notation that will be used (and recalled) during the article. If $t$
is an integer, let $\psi_{t}$ denote the quadratic character
corresponding to the quadratic extension $\Q(\sqrt{t})/\Q$. Let
$L/\Q$ be a Galois extension, let
$\rho:\Gal_L \to \GL_2(\overline{\Q_p})$ be a Galois representation
and let $\tau \in \Gal_\Q$. By $^\tau\rho$ we denote the Galois
representation of $\Gal_L$ given by
$^\tau\rho(\sigma) = \rho(\tau \sigma \tau^{-1})$.

A key property satisfied by the curve $\E$ is that if $\tau$ denotes
any element of $\Gal_\Q$ not trivial on $K$, then
$^\tau\rho_{\E,p} \simeq \rho_{E,p} \otimes \psi_{-2}$ (see
Proposition~\ref{prop:qcurveE}). Analogously, under the same
hypothesis, $^\tau\rho_{\Et,p} \simeq \rho_{\Et,p} \otimes \psi_{-3}$
(see Proposition~\ref{prop:qcurveE2}). Recall that a representation of
$\Gal_K$ extends to $\Gal_\Q$ if and only if $^\tau\rho \simeq \rho$,
for $\tau$ as before. Although this result is well known to experts, a proof
of it will be provided in Theorems \ref{thm:levelandnebentypus} and
\ref{thm:levelandnebentypus2}, as we need control on the extension
conductor. Note that is we can construct a finite
order Hecke character $\chi_t: \Gal_K \to \CC^\times$ satisfying
\[
  ^\tau \chi_t = \chi_t \cdot \psi_{-t},
\]
for $t=2, 3$, then the twisted representation
$\rho_{\E,p}\otimes \chi_2$ (respectively $\rho_{\Et,p}\otimes\chi_3$)
does extend to $\Gal_\Q$. The main contribution of the present article
(as developed in Section~\ref{section:character}) is to give a general
strategy to construct the Hecke character $\chi_t$. Furthermore, we
give a concrete explicit construction of it in the cases $t=2$ and $t$
a prime number which is congruent to $3$ modulo $4$. This result might
be of independent interest, as it could be applied to other
Diophantine problems involving $\Q$-curves.

While constructing the Hecke character, we restrict to positive values
of $d$, since one of the steps of the construction involves verifying
a compatibility condition at units of the base field. This technical
problem is in general hard to verify for real quadratic fields due to
the existence of fundamental units. Nevertheless, we are able to prove
that our construction works when the fundamental unit of
$\Q(\sqrt{d})$ (for $d$ positive) has norm $-1$. The case of general
real quadratic fields (and applications to Diophantine problems) can
be found in the sequel \cite{2103.06965}.

By construction, the Hecke character $\chi_t$ is ramified only at
primes ramifying in $K/\Q$ and at $t$. Going back to the problem of
extending the representation $\rho_{\E,p}$ (respectively
$\rho_{\Et,p}$), this fact allows us to prove that the extended
representation has residual image of conductor divisible only by
primes dividing $6d$. In particular, we construct a rational
representation (whose modularity follows from Serre's conjectures)
with small conductor, fulfilling the first two steps of the general
strategy. Moreover, in Corollaries~\ref{coro:loweringlevelE} and
\ref{coro:loweringlevelEt} we give an explicit description of the
conductor and Nebentypus of such residual representations.

To apply the last steps of the general strategy, one needs a result on
the image of the residual Galois representation (needed in Ribet's
lowering the level result) and an elimination procedure. For our
purpose, the main result regarding residual images of Galois
representations coming from $\Q$-curves is due to Ellenberg
(\cite[Theorem 3.14]{MR2075481}), which states that if a $\Q$-curve $E$ over a
quadratic field $K$ happens to have a prime of multiplicative
reduction larger than $3$, then for all large enough primes $p$ (say
$p >N_K$), the residual image is ``large'' (i.e. its projective image
contains $\SL_2(\F_p)$). Moreover, Elleberg's article contains bounds
that allows to compute explicitly the constant $N_K$, which depends
only on the base field $K$ and the degree of the isogenies between $E$
and its Galois conjugates.  We included some improvements from the
literature to Ellenberg's original result, and we wrote a
\verb*|PARI/GP| script to compute the bound needed for the
examples. 

After explaining how Ellenberg's result is important to eliminate
newforms with complex multiplication, we recall a strategy due to
Mazur that we learned from \cite{MR3098134}. It is the main tool used
in the present article to discard newforms as not coming from
solutions of our equation. All the previous techniques/constructions
are used to study solutions of equations (\ref{eq:24p}) and
(\ref{eq:ben-chen}) for the values $d=1, 2, 3, 5, 6$ and $7$ not
studied before. The choice of the values for $d$ is arbitrary, since in
principle our method could be applied to other values of $d$, within the 
computational limitations of the algorithms used to construct newforms.

We succeed to prove non-existence
of solutions in all cases but $d=5$ and $d=7$ for equation
(\ref{eq:ben-chen}) (where the existence of newforms satisfying all
the required properties makes the classical approach to fail). In
\cite{2109.07291} more advanced elimination techniques are used to get
partial results for these two cases as well.

The article is organized as follows: in
Section~\ref{section:generalstrategy} we give details of a general
strategy due to Darmon for studying solutions of
equation~(\ref{eq:generalizedfermat}). In
Section~\ref{section:properties} we study the main properties of the
curves (\ref{eq:freycurve}) and (\ref{eq:FC-BC}). In particular, we
prove that they are $\Q$-curves, compute their complete field of
definition, and their local type at primes of bad reduction (getting
in particular a formula for their
conductor). Section~\ref{section:character} is devoted to the
construction of the Hecke character $\chi_t$. It starts with a general
idea of its construction and main properties, while the proper
construction and proofs are given in Subsection~\ref{section:t=2} for
$t=2$ and in Subsection~\ref{section:t=3} for $t$ any prime number
congruent to $3$ modulo $4$. The proofs are a little tedious, so we
suggest the reader to avoid reading them on a first read of the
article.  The way to define the character $\chi_t$ is to first give
its local ramification (i.e. the value of the restriction of the local
components to the local ring of integers) and then extend it to the
whole id\`ele group. Although it is not clear how the local
definitions were obtained, they are the result of many computational
experiments done with a script wrote in \verb*|PARI/GP| (\cite{PARI2})
for that purpose, and a meticulous task of finding patterns on the
outputs. Such experiments allowed us to give a very clean statement;
once the character definition is given, one is only left to check that it
satisfies the expected properties.

Section~\ref{section:extension} contains a proof of the extension
result. As mentioned before, although it is well known to experts, we
included its proof for two reasons: on the one hand, we could not find
a clear reference to its proof, but also because we need the explicit
result on the conductor of the extended representation. The section is
split into two parts: one for the representation $\rho_{\E,p}$ and
another one for the representation $\rho_{\Et,p}$. In both cases, a formula for the conductor of the residual extended representation is given.

Section~\ref{section:levellowering} is devoted to study conditions to
assure absolutely irreducible image of the residual representation. As
explained before, the main result is due to Ellenberg. The same section
contains a strategy to be used when the curve $E$ does not have the
prime of multiplicative reduction needed to apply Ellenberg's result.
Section~\ref{section:methostodiscard} contains a description of
Mazur's method, the main tool used to discard newforms. It also
contains some remarks on how to handle newforms with complex
multiplication, and the role of the trivial solutions in the
elimination procedure.

At last, Section~\ref{section:solvingequation1} contains results on
equation~(\ref{eq:24p}) for $d=5$ (Theorem~\ref{thm:eq1d5}), $d=6$
(Theorem~\ref{thm:eq1d6}) and $d=7$ (Theorem~\ref{thm:eq1d7}), while
Section~\ref{section:solvingequation2} contains solutions and attempts
to study primitive solutions of equation~(\ref{eq:ben-chen}) for $d=2, 5, 6$ and
$7$. We do get non-existence results for $d=2$
(Theorem~\ref{thm:eq2d2}) and for $d=6$ (Theorem~\ref{thm:eq2d6}) but
the elimination procedure fails in the cases $d=5, 7$ due to the
existence of newforms (without complex multiplication) that
systematically pass Mazur's test. In \cite{2109.07291} some
partials results are presented for both such equations.

The present article depends on different scripts that can be
downloaded from the web page
\url{http://sweet.ua.pt/apacetti/research.html}. There are mainly two
different scripts used. One of them (written in \verb*|PARI/GP| \cite{PARI2})
is the one used to give an explicit bound for Ellenberg's result (as
explained in the proof of Theorem~\ref{thm:ellenberg}), called
\textit{``Ellenberg.gp''}. The second ones are used to eliminate
newforms (mostly an implementation of Mazur's trick) in \verb*|Magma|
(\cite{MR1484478}). There is one script for each equation (as it
includes the definition of the character $\chi_t$); the scripts used
for equation~(\ref{eq:24p}) are \textit{``Eq1d5.mg''},
\textit{``Eq1d6.mg''} and \textit{``Eq1d7.mg''} which depend on the
file \textit{``Mazur26p.mg''}. The scripts used for
equation~(\ref{eq:ben-chen}) are 
\textit{``Eq2d5.mg''}, \textit{``Eq2d6.mg''} and
\textit{``Eq2d7.mg''}. The outputs of the scripts are available in the files \textit{``OutputsEq1.txt''} and \textit{``OutputsEq2.txt''}.

\subsection*{Acknowledgments} We would like to thank John Cremona for
many useful conversations regarding computing with Bianchi modular
forms, and for explaining us how to use his code to compute them.

\section{General strategy}
\label{section:generalstrategy}
Based on Frey-Hellegouarch's approach of relating a solution to
Fermat's equation with a ``special'' elliptic curve, Darmon in
\cite{MR1756104} proposed a generalization of the strategy to study a
general Fermat-type equation of the
form~(\ref{eq:generalizedfermat}). The strategy consists roughly
in the following three steps:
\begin{enumerate}[(1)]
\item \textbf{Construct a Galois representation:} to a putative
  solution $(a,b,c)$ of equation~(\ref{eq:generalizedfermat}), attach
  an odd residual Galois representation
  $\overline{\rho}_{(a,b,c)}:\Gal_K \to \GL_2(\F)$ that is unramified at
  primes not dividing $ABC$ and $\ell$, where $\F$ is a finite field
  of characteristic $\ell$ and $K$ is a finite extension of $\Q$
  depending on the exponents $(p,q,r)$.
\item \textbf{Prove Modularity of $\overline{\rho}_{(a,b,c)}$:} prove that
  $\overline{\rho}_{(a,b,c)}$ matches the reduction of a Galois
  representation attached to an automorphic representation of
  $\GL_2(\AA_K)$ whose level is only divisible by primes dividing the
  Artin conductor of $\overline{\rho}_{(a,b,c)}$.

\item \textbf{Reach a Contradiction:} compute the space of
  automorphic representations of the last item, and prove that none is
  related to a possible solution.
\end{enumerate}
The way the representation $\overline{\rho}_{(a,b,c)}$ is constructed
in \cite{MR1756104} (that covers most possible exponents $(p,q,r)$ but
not all) is via the specialization of a parametric family
$\overline{\rho}(t)$ at the point $t=\frac{Aa^p}{Cc^r}$. Furthermore,
in the aforementioned article, Darmon gives an equation of a curve
$C(a,b,c)$ whose natural residual Galois representation (via the
action of the Galois group $\Gal_\Q$ on the $p$-torsion points of its
Jacobian) gives raise to $\overline{\rho}_{(a,b,c)}$.  In our case of
interest, as mentioned in the introduction, to study solutions
of~(\ref{eq:24p}), in \cite{MR2561200} the authors attach to a
solution the elliptic curve
\begin{equation*}
\E: y^2=x^3+4ax^2+2(a^2+\sqrt{-d}b)x,  
\end{equation*}
defined over the field $K = \Q(\sqrt{-d})$. For studying the
equation~(\ref{eq:ben-chen}), we propose as a natural object
attached to a solution $(a,b,c)$ the elliptic curve
\begin{equation*}
\Et:y^2+6b\sqrt{-d}xy-4d(a+b^3\sqrt{-d})y = x^3,
\end{equation*}
defined again over the quadratic field $K=\Q(\sqrt{-d})$ (generalizing the construction in \cite{MR2966716}).  Note that
if $d>0$, then both curves are defined over an imaginary quadratic
field. For that purpose, let us explain in more detail what the second
step ``prove modularity'' means. There should exist an automorphic
representation $\Pi$ of $\GL_2(\AA_K)$ with the following properties:
\begin{itemize}
\item Let $L$ denote the coefficient field of the automorphic
form $\Pi$ and let $\Om$ denote its ring of integers. Then for each
prime ideal $\lambda$ of $\Om$, there exists a Galois representation
$\rho_{\Pi,\lambda}:\Gal_K \to \GL_2(\Om_{\lambda})$ such that
$L(\Pi,s) = L(\rho_{\Pi,\lambda},s)$.
\item There exists a prime $\lambda$ dividing $\ell$ such that the residual representation $\overline{\rho_{\Pi,\lambda}}$ is isomorphic to $\overline{\rho}_{(a,b,c)}$.
\end{itemize}
The curves $C(a,b,c)$ given in \cite{MR1756104} are all defined over
$\Q$ and attain their full endomorphism ring over a totally real
field, hence the relation between automorphic representations, Hilbert
modular forms and Galois representations is well known in this
case. Although our curves are defined over imaginary quadratic fields,
they are what is called a ``$\Q$-curve'' (as will be explained in
Section~\ref{section:properties}), which implies that their Galois
representations are related to classical weight $2$ modular forms.
The existence of a Bianchi modular whose $L$-series matches the one of
our elliptic curve then follows from Langlands results on cyclic base
change. While solving equation~(\ref{eq:ben-chen}) for $d=2$ this fact
will be extremely useful, as computing Bianchi modular forms is more
effective in this case (all other computations will involve classical
modular forms).

\section{$\Q$-curves and properties of $\E$ and $\Et$} \label{section:properties}

\begin{defi}
  Let $L$ be a number field and $E/L$ an elliptic curve. The curve $E$
  is called a \emph{$\Q$-curve} if for all $\sigma \in \Gal_\Q$, the curve
  $\sigma(E)$ is isogenous to $E$.
\end{defi}

The isogeny between $E$ and $\sigma(E)$ needs not be defined over
$L$. The minimum field where the curve and all the isogenies are
defined is usually called the field of total definition. Let
$\tau \in \Gal_\Q$ be any element whose restriction to $\Q(\sqrt{-d})$
is not the identity and let $\psi_{-2}$ denote the quadratic character
corresponding to the quadratic extension $\Q(\sqrt{-2})/\Q$.

\begin{prop}
  The elliptic curve $\E$ is a $\Q$-curve which is totally defined
  over the field $\Q(\sqrt{-d},\sqrt{-2})$. Furthermore, $\tau(E)$ is
  isogenous to the quadratic twist $E \otimes \psi_{-2}$.
  \label{prop:qcurveE}
\end{prop}
\begin{proof}
  This result is explained in \cite{MR2561200}. Let $\tau$ be a
  non-trivial generator of $\Gal(\Q(\sqrt{-d})/\Q)$. The point $(0,0)$
  has order two, and (as explained in \cite[Example 4.5]{MR2514094})
  the quotient has equation
\[
y^2=x^3-8ax^2+8(a^2-\sqrt{-d}b)x.
\]
An easy change of variables proves that it equals the quadratic twist
by $-2$ of $\tau(E)$. In particular, the curve and the isogenies are all defined over the field $\Q(\sqrt{-d},\sqrt{-2})$ as claimed.
\end{proof}

Let $\psi_{-3}$ denote the quadratic character corresponding to the
quadratic extension $\Q(\sqrt{-3})/\Q$.

\begin{prop}
  The elliptic curve $\Et$ is a $\Q$-curve which is totally defined
  over the field $\Q(\sqrt{-d},\sqrt{-3})$. Furthermore, $\tau(\Et)$
  is isogenous to the quadratic twist $\Et \otimes \psi_{-3}$.
  \label{prop:qcurveE2}
\end{prop}

\begin{proof}
  As explained in \cite[Table 1]{Kubert}, all elliptic
  curves having a $3$-torsion point have a minimal model of the form:
\[
E:y^2+a_1xy+a_3y = x^3,
\]
where $P=(0,0)$ is a point of order $3$. Its $3$-isogenous curve (obtained
as the quotient of the curve by the order $3$ group generated by $P$)
has equation
\[
y^2+a_1xy+a_3y = x^3-5a_1a_3x-a_1^3a_3-7a_3^2.
\]
Our curve $\Et$ corresponds to the values
$a_1 = 6b\sqrt{-d}$, $a_3 = -4d(a+b^3\sqrt{-d})$. The previous formula
implies that the quotient $\Et$ by
$\langle(0,0)\rangle$  has equation
\begin{multline}
\label{eq:3quot}
y^2+6b\sqrt{-d}xy-4d(a+b^3\sqrt{-d})y = x^3+\\
(-120b^4d^2+120abd\sqrt{-d})x + 976b^6d^3-1088ab^3d^2\sqrt{-d}-112a^2d^2. 
\end{multline}
On the other hand, if $\tau$ generates the Galois group
$\Gal(\Q(\sqrt{-d})/\Q)$, clearly
$\tau(\Et) = \widetilde{E}_{(a,-b,c)}$. The
quadratic twist of $\widetilde{E}_{(a,-b,c)}$ by $\sqrt{-3}$ (which
can be computed for example using \cite{PARI2}) corresponds to the
equation
\begin{multline}
\label{eq:twist-3}
y^2+6b\sqrt{-d}xy+12d(-a+b^3\sqrt{-d})y = x^3+36b^2dx^2+\\(144abd\sqrt{-d}+144b^4d^2)x+288ab^3d^2\sqrt{-d}+144b^6d^3-144a^2d^2.
\end{multline}
Via the usual change of variables (making $a_1=a_3=a_2=0$) it is easy
to check that both (\ref{eq:twist-3}) and (\ref{eq:3quot}) translate
to the curve
\[
y^2=x^3 + (108abd\sqrt{-d} - 135b^4d^2)x   - 756ab^3d^2\sqrt{-d} +594b^6d^3 - 108a^2d^2.
\]
In particular, $\tau(\Et)$ is isogenous to the quadratic twist of
$\Et$ by $\sqrt{-3}$. Hence $\Et$ is a $\Q$-curve and its complete field
of definition equals $\Q(\sqrt{-d},\sqrt{-3})$ as claimed.
\end{proof}

Let $K =\Q(\sqrt{-d})$ and $\Om_K$ its ring of integers. The following trivial result will be useful later. 

\begin{lemma}
  Let $(a,b,c)$ be a primitive solution of either equation
  (\ref{eq:24p}) or equation~(\ref{eq:ben-chen}) with $p>3$. Then
  \begin{itemize}
  \item $\gcd(ac,d)=1$.
		
  \item If $d \equiv 1, 3, 5 \pmod 8$ only one of $\{a,b\}$ is even    and the other one is odd.
  \item If $2$ does not split in $K$ (i.e. $d\not\equiv7\pmod8$), then $c$ is odd.
  \item If $3$ does not split in $K$ (i.e. $d\not\equiv2\pmod3$), then $c$ is not divisible by $3$.
  \end{itemize}
\label{lemma:odd}
\end{lemma}

\subsection{Properties of $\E$}
Recall the defining equation:
\[
  \E: y^2=x^3+4ax^2+2(a^2+b\sqrt{-d})x.
\]
Its discriminant equals $\Disc(\E) = 512(a^2+b\sqrt{-d})c^p$ and its
$j$-invariant equals
$j(\E)=\frac{64(5a^2-3b\sqrt{-d})^3}{c^p(a^2+b\sqrt{-d})}$.

\begin{lemma}
  \label{lemma:notramification}
  Suppose that $q$ is an odd rational prime ramified at $K/\Q$ and let
  $\id{q}$ denote the (unique) prime ideal in $\Om_K$ dividing $q$. Then
  $\id{q}\nmid \Disc(\E)$.
\end{lemma}
\begin{proof}
  Since $q$ is ramified, $\id{q} \mid \sqrt{-d}$, and since $(a,b,c)$
  is a primitive solution, $\id{q}\nmid a$. Then
  $\id{q} \nmid c^p(a^2+\sqrt{-d}b)$.
\end{proof}

\begin{lemma}
  Let $\id{q}$ be an odd prime ideal of $\Om_K$ such that $\id{q}\mid\Disc(\E)$. Then
  $\E$ has multiplicative reduction at $\id{q}$.
	\label{lemma:multred}
\end{lemma}
\begin{proof}
  By Lemma~\ref{lemma:notramification} we know that primes dividing
  $\Disc(\E)$ are not ramified in $K/\Q$; in particular, if
  $\id{q}$ is an odd prime dividing $\Disc(E_{(a,b,c)})$,
  $\id{q}\nmid 4a$, hence clearly the reduction of
  (\ref{eq:freycurve}) modulo $\id{q}$ is multiplicative.
\end{proof}

\begin{lemma}
Let $\id{q}$ be an odd prime ideal of $\Om_K$. Then $v_{\id{q}}(\Disc(\E)) \equiv 0 \pmod p$.
\label{lemma:loweringthelevel}
\end{lemma}
\begin{proof}
  From the discriminant formula, clearly
  $v_{\id{q}}(\Disc(\E)) = pv_{\id{q}}(c) +
  v_{\id{q}}(a^2+b\sqrt{-d})$.  If
  $\id{q} \mid \gcd(a^2+\sqrt{-d}b,a^2-\sqrt{-d}b)$ then
  $\id{q} \mid 2$ (because $(a,b,c)$ is primitive), hence the equality
  $c^p = a^4+db^2=(a^2+\sqrt{-d}b)(a^2-\sqrt{-d}b)$ implies that
\[
  v_{\id{q}}(a^2+\sqrt{-d}b) =\begin{cases}
    0 & \text{if } \id{q} \nmid (a^2+\sqrt{-d}b),\\
    v_{\id{q}}(c^p) & \text{ otherwise}.
\end{cases}
\]
\end{proof}

The last two results are the ones needed for removing the primes dividing $c$ from the conductor of the residual representation (via a theorem of Ribet).
Let $N(\E)$ denote the conductor of $\E$. Assume that $p\ge 11$ to avoid
extra computations when $2$ splits in $K$.
\begin{lemma}
	\label{lemma:2valuation}
	Let $\id{p}_2$ be a prime ideal of $\Om_K$ dividing $2$.
	\begin{enumerate}
		\item If $2$ is inert in $K$ then $\E$ has reduction type {\rm III} at 2, with $v_{2}(N(\E))=8$.
		\item If $2$ ramifies in $K$ then $\E$ has reduction type ${\rm I}_2^*$ or ${\rm I}_4^*$ at $\id{p}_2$, with 
		$v_{\id{p}_2}(N(\E)) \in \{10,12\}$.
		\item If $(2)=\id{p}_2\overline{\id{p}}_2$ then either $\E$ has reduction
		type {\rm III} at both primes, with
		$v_{\id{p}_2}(N(\E))=v_{\overline{\id{p}}_2}(N(E_{(a,b,c)}))=8$, or $\E$
		is a twist of a curve with multiplicative reduction, so the primes
		can be chosen so that $v_{\id{p}_2}(N(\E))=6$ and
		$v_{\overline{\id{p}}_2}(N(\E))\in\{1,4\}$.
	\end{enumerate}
\end{lemma}

\begin{proof} The proof is a straight application of Tate's algorithm
	(\cite{MR0393039}). The invariants of $\E$ are: $a_6=0$, $b_2=16a$,
	$b_6=0$ and $b_8=-4(a^2+\sqrt{-d}b)^2$. 
	\begin{enumerate}
		\item The hypothesis implies that $d\equiv3\pmod8$ and $2$ is prime
		in $\Om_K$.  Notice that, by Lemma~\ref{lemma:odd},
		$2\nmid a^2+\sqrt{-d}b$ hence $v_2(\Delta(E_{(a,b,c)}))=9$. Since:
		$2\mid b_2$, $2^2\mid a_6$, $2^3\nmid b_8$, the curve has
		reduction type \r{III} and
		$v_2(N(\E))=v_2(\Delta(\E))-1=8.$
		
		\item Let $\id{p}_2$ be the unique prime in $\Om_K$ dividing $2$ and
		let $\pi$ be a local uniformizer. By Lemma~\ref{lemma:odd},
		$\id{p}_2 \nmid(a^2+\sqrt{-d}b)$, hence
		$v_{\id{p}_2}(\Delta(\E))=18$. To easy notation, consider
		the curve
		\begin{equation}
			\label{eq:eccequation}
			y^2=x^3+4\alpha x^2+2\beta x,      
		\end{equation}
		where $\id{p}_2\nmid \beta$. Clearly $\id{p}_2 \mid b_2$,
		$\id{p}_2^2 \mid a_6$, $\id{p}_2^3 \mid b_8$ and $\id{p}_2^3\mid b_6$.
		Following Tate's notation, let $a_{n,m}=\frac{a_n}{\pi^m}$. The
		polynomial $P=x^3+a_{2,1}x^2+a_{4,2}x+a_{6,3}$ has a double root
		at $x=1$, hence we make the translation $x \to x+\pi$ in
		(\ref{eq:eccequation}), to get the new equation
		\begin{equation}
			\label{eq:secondequation}
			y^2=x^3+(4\alpha+3\pi)x^2+(8\pi \alpha+3\pi^2+2\beta)x +
			4\pi^2\alpha + \pi^3+2\beta\pi.
		\end{equation}
		Write $\tilde{a}_i$ for the new coefficients. If either $d$ is
		even (so we can take $\pi=\sqrt{-d}$) and $b$ is odd, or $d$ is
		odd (so $\pi = 1+\sqrt{-d}$) and $b$ is even (hence $a$ is odd),
		$v_{\id{p}_2}(\pi^2+2\beta) = 3$ hence $v_{\id{p}_2}(\tilde{a_6})=4$
		and the polynomial $Y^2+\tilde{a}_{3,2}Y -\tilde{a}_{6,4}$ has a
		double non-zero root. Although we need to make a translation (to
		take the double root to zero), such procedure will not change
		$\tilde{a}_{4,3}$, which has valuation $3$, so the type equals
		I$_2^*$ and
		$v_{\id{p}_2}(N(\E))=v_{\id{p}_2}(\Disc(\E))-6=12$.
		
		Suppose that $d$ is even and $b$ is even. If
		$\frac{d}{2}\equiv 1 \pmod 4$ then
		$v_{\id{p}_2}(\tilde{a}_{6}) \ge 6$ and
		$v_{\id{p}_2}(\tilde{a}_{4}) =4$, so we do not need to make any
		translation and the type is I$_4^*$ and
		$v_{\id{p}_2}(N(\E))=18-8=10$. If $\frac{d}{2}\equiv 3 \pmod 4$ then
		$v_{\id{p}_2}(\tilde{a}_{6})=4$ and $v_{\id{p}_2}(\tilde{a}_{4})\ge 5$
		hence the polynomial $Y^2+\tilde{a}_{3,2}Y-\tilde{a}_{6,4}$ has a
		non-zero double root, and after sending the root to $0$, we get
		that the new $a_4$ has valuation $4$, so the same computation
		works.
		
		At last, if $d \equiv 1 \pmod 4$ and $b$ is odd,
		$v_{\id{p}_2}(\tilde{a}_{6}) \ge 5$ and
		$v_{\id{p}_2}(\tilde{a}_{4}) =4$, hence again the type is I$_4^*$
		and $v_{\id{p}_2}(N(\E))=18-8=10$.
		
		\item Let $\id{p}_2$ be a prime dividing $2$. Consider the different
		cases:
		\begin{itemize}
			\item If either $a$ or $b$ is even (hence the other one is odd)
			then $v_{\id{p}_2}(a^2+\sqrt{-d}b)=0$ and
			$v_{\id{p}_2}(\Delta(\E))=9$ (for both primes). Clearly
			$v_{\id{p}_2}(b_2)\ge 4$ and $v_{\id{p}_2}(b_8)=2$ hence the
			reduction type is $\rm{III}$ and
			$v_{\id{p}_2}(N(\E))=9-1=8$ (at both primes).
			
			\item If both $a,b$ are odd, we can assume that
			$v_{\id{p}_2}(a^2+\sqrt{-d}b) >1$ while
			$v_{\overline{\id{p}}_2}(a^2+\sqrt{-d}b)=1$ (since
			$\frac{a^2+\sqrt{-d}b}{2}$ is an integer, and
			$v_{\overline{\id{p}}_2}(a^2+\sqrt{-d}b) = v_{\id{p}_2}(a^2-\sqrt{-d}b) =
			v_{\id{p}_2}(a^2+\sqrt{-d}b-2\sqrt{-d}b)$). Furthermore, our
			assumption $p\ge 11$ implies that
			$v_{\id{p}_2}(a^2+\sqrt{-d}b)\ge 11$, so
			$v_{\id{p}_2}(j(E_{(a,b,c)}))<0$. In particular $\E$ has
			potentially multiplicative reduction. The equation is not
			minimal at $\id{p}_2$; under a change of variables, it equals
			\[
			y^2 = x^3+ax^2+\frac{(a^2+\sqrt{-d}b)}{2^5}x,
			\]
			which already has multiplicative reduction. Hence its conductor
			equals $\id{p}_2$ or $\id{p}_2^4$. To compute the type at
			$\overline{\id{p}}_2$, the hypothesis also implies that
			$v_{\overline{\id{p}}_2}(j(\E))<0$ so the curve has
			potentially multiplicative reduction, but it equals a quadratic
			twist (by the character of conductor $8$) of a curve with
			multiplicative reduction, hence its conductor equals
			$\overline{\id{p}}_2^6$.
		\end{itemize}
	\end{enumerate}
\end{proof}

The following technical Lemma is needed only to compute the conductor
of the extended Galois representation when $d \equiv 1 \pmod 8$, so we
recommend the reader to skip it on a first reading.
\begin{lemma}
	Suppose that $d \equiv 1 \pmod 8$ and $b$ is odd. Let $\id{p}_2$ the prime dividing $2$,  $\pi = 1+\sqrt{-d}$ and $\epsilon$ be any unit. Then, at $\id{p}_2$,  the elliptic
	curve $\E$ twisted by $\epsilon \pi$ with equation
	\[
	\epsilon\pi y^2 = x^3 +4ax^2+ 2(a^2+\sqrt{-d}b)x,
	\]
	has conductor exponent $8$ if $b \equiv 1 \pmod 4$ and $6$ if
	$b \equiv 3 \pmod 4$.
	\label{lemma:d=1reduction}
\end{lemma}

\begin{proof}
	The proof also follows Tate's algorithm. The effect of twisting our
	elliptic curve in equation (\ref{eq:secondequation}) is that the
	coefficient $a_i$ becomes $a_i(\pi\varepsilon)^{-i}$, hence the
	discriminant valuation decreases by $6$ and the coefficients $a_i$
	decrease their valuation by $i$. Then if $b \equiv 1 \pmod 4$, the
	valuations are $v_{\id{p}_2}(\tilde{a}_2)=0$,
	$v_{\id{p}_2}(\tilde{a}_4)\ge 4$ while $v_{\id{p}_2}(\tilde{a}_6)=2$
	(because $\frac{\pi^2}{2} + b\sqrt{-d}$ is divisible by $\pi$ but
	not by $\pi^3$). Following Step 6 of Tate's algorithm, we make the
	change of variables $y \to y+x+\pi$ and the new even coefficients
	become $\tilde{a}_2-1$, $\tilde{a}_4-2\pi$ and
	$\tilde{a}_6 - \pi^2$, with valuations $2$, $3$ and $3$
	respectively, so the reduction type is I$_0^*$ and the conductor
	equals $18-6-4=8$.
	
	On the other hand, if $b \equiv 3 \pmod 4$ then
	$v_{\id{p}_2}(\tilde{a}_2)=0$, $v_{\id{p}_2}(\tilde{a}_4)=2$
	(because $2 \mid \frac{3\pi^2}{2}+b\sqrt{-d}$ but $\pi^3$ does not)
	while $v_{\id{p}_2}(\tilde{a}_6)\ge 4$. Following Tate's algorithm,
	apply the change of variables $y \to y+x$ to get coefficients
	$\alpha_1=2$, $\alpha_2=\tilde{a}_2-1$, $\alpha_3=0$,
	$\alpha_4=\tilde{a}_4$ and $\alpha_6=\tilde{a}_6$ with
	$v_{\id{p}_2}(\alpha_2)=2$, $v_{\id{p}_2}(\alpha_4)=2$ and
	$v_{\id{p}_2}(\alpha_6)\ge 4$. Since the reduction of the polynomial
	$t^3+\frac{\alpha_2}{\pi}t^2+\frac{\alpha_4}{\pi^2}t +
	\frac{\alpha_6}{\pi^3}$ has a double root and one simple one, the
	reduction type is I$_n^*$. Making the change of variables
	$x \to x+\pi$ the new $a_4$ equals $3\pi^2+\frac{2}{\pi^2}\beta+3$,
	which has valuation $3$ at $\id{p}_2$, so the type is I$_2^*$ and the
	conductor valuation equals $18-6-6=6$ as claimed.
\end{proof}

All primitive trivial solutions  correspond to very special curves.

%

\begin{lemma}
  The trivial solution $(0,0,0)$ gives rise to a singular curve. All
  other trivial primitive solutions of~(\ref{eq:24p}) are the
  following:
  \begin{itemize}
  \item The solution $(1,0,1)$, corresponding to the curve with lmfdb
    label \lmfdbec{256}{a}{2} and complex multiplication by
    $\Z[\sqrt{-2}]$.
		
  \item The solution $(-1,0,1)$, corresponding to the curve with lmfdb
    label \lmfdbec{256}{d}{2} and complex multiplication by
    $\Z[\sqrt{-2}]$.
		
  \item The solution $(0,\pm 1,1)$ (when $d=1$), corresponding to the
    curve with lmfdb label \lmfdbec{256}{c}{2} and complex
    multiplication by $\Z[\sqrt{-1}]$.
		
	\end{itemize}
\label{lemma:CM1}
\end{lemma}
\begin{proof} The solution $(0,0,0)$ clearly gives a singular
  curve. Any other trivial solution must have $b=0$ or $a=0$. In the
  first case, the primitive hypothesis implies that the solution
  equals $(\pm 1,0,1)$. If $a=0$, and $q$ is a prime dividing $d$, it
  must divide $c$ but since the solution is primitive, $q$ cannot
  divide $b$. This implies that no such prime can exist, hence and the
  later is a solution precisely $d=1$, with trivial solution $(0,\pm 1,1)$.
\end{proof}

\begin{remark}
  The conductor of the curve $E_{(\pm 1,0,1)}$ over $K$ has valuation
  $8$ at primes dividing $2$ when $2$ is unramified in $K/\Q$,
  valuation $10$ when $2 \mid d$ and valuation $12$ when $2$ ramifies in
  $K/\Q$ but $2 \nmid d$ (equivalently when $d \equiv 1 \pmod 4$).
  \label{rem:conductorCM}
\end{remark}


\subsection{Properties of $\Et$}
Recall once again the defining equation
\[
  \Et:y^2+6b\sqrt{-d}xy-4d(a+b^3\sqrt{-d})y = x^3.
 \]
 Its discriminant equals
 $\Delta(\Et)=-2^ 8 3^ 3 d^ 4c^p(a+b^ 3\sqrt{-d})^ 2$; note that the
 discriminant is divisible by $d$ (unlike the previous case). Its
 $j$-invariant is
 $j(\Et)=\frac{2^43^3b^3\sqrt{-d}(4a-5b^3\sqrt{-d})^3}{c^p(a+b^3\sqrt{-d})^2}$.

\begin{lemma}
  Let $\id{q}$ be a prime ideal of $\Om_K$ such that
  $\id{q} \mid \Delta(\Et)$ and $\id{q} \nmid 6d$. Then
  $\Et$ has multiplicative reduction at $\id{q}$.
\label{lemma:multredEt}
\end{lemma}
\begin{proof}
  Mimics the proof of Lemma~\ref{lemma:multred}.
\end{proof}

\begin{lemma}
  Let $\id{q}$ be a prime of $\Om_K$ such that $\id{q}\nmid6d$. Then
  $v_{\id{q}}(\Disc(\widetilde{E}_{(a,b,c)})) \equiv 0 \pmod p$.
  \label{lemma:bajadaniveleq2}
\end{lemma}
\begin{proof}
Mimics the proof of Lemma~\ref{lemma:loweringthelevel}.
\end{proof}

One important difference between equation~(\ref{eq:24p}) and
equation~(\ref{eq:ben-chen}) is that we will not be able to remove
(via a lowering the level result) the ramified odd primes from the
conductor of the residual representation.

\begin{lemma}
  Suppose that $q$ is an odd rational prime ramified at $K/\Q$ and let
  $\id{q}$ denote the (unique) prime ideal in $\Om_K$ dividing $q$. Then
  $v_{\id{q}}(\Disc(\Et))=8+3v_{\id{q}}(3)$.
  \label{lemma:discvaluation}
\end{lemma}
\begin{proof}
  Since $q$ is ramified, $\id{q} \mid \sqrt{-d}$, and since $(a,b,c)$
  is a primitive solution, $\id{q}\nmid a$. Then, using that
  $\id{q} \nmid c^p(a+b^3\sqrt{-d})$ and that $v_{\id{q}}(d)=2$ the
  result follows.
\end{proof}

\begin{remark}
  The curve $\Et$ has bad additive reduction at all odd primes
  different from $3$ ramifying in $K/\Q$. However, over the extension
  $K(\sqrt[6]{-d})$ it attains good reduction (via the usual change of
  coordinates $(x,y) \to (\sqrt[3]{(-d)^2}x,dy)$). If $q \mid d$ is
  such an odd prime, let $\id{q} = \langle q, \sqrt{-d}\rangle$ denote
  the ideal in $K$ dividing it. If $q \equiv 1 \pmod 3$, the extension
  $K_{\id{q}}(\sqrt[6]{-d})/K_{\id{q}}$ is an abelian extension, hence
  the local type of the Weil-Deligne representation at $\id{q}$ is
  that of a principal series (given by an order $3$ character), while
  if $q \equiv 2 \pmod 3$ the curve attains good reduction over a
  non-abelian extension, hence its local type matches that of a
  supercuspidal representation (obtained inducing an order $3$
  character from the quadratic unramified extension
  $K_{\id{q}}(\zeta_3)/K_{\id{q}}$).
	\label{rem:twistlevel}
\end{remark}
Let $N(\Et)$ denote the conductor of $\Et$ and suppose that $p>3$.

\begin{lemma}
  Let $\id{p}_2$ be a prime ideal of $\Om_K$ dividing $2$. Then:
  \begin{enumerate}
		
  \item If $2$ is inert in $K$ then $\Et$ has reduction type
    $\rm{IV}^*$ at $2$, with
    $v_{2}(N(\Et))=2$.
		
  \item If $2$ is split in $K$ then $\Et$ has reduction type $\rm{IV}^*$ or $\rm{I}_n$ at $\id{p}_2$, with
    $v_{\id{p}_2}(N(\Et))=1, 2$ at both primes above
    $2$.
  \item If $2$ ramified in $K$ but $2 \nmid d$ then
    $\Et$ has  reduction type $\rm{IV}$ at $\id{p}_2$,
    with $v_{\id{p}_2}(N(\Et))=2$.
		
  \item If $2 \mid d$ then $\Et$ has good reduction at $\id{p}_2$.
	\end{enumerate}
	\label{lemma:case2conductor2}
\end{lemma}
\begin{proof} Consider each case separately:
  \begin{enumerate}
  \item If $2$ is inert then $2 \nmid c$, by Lemma~\ref{lemma:odd}. Clearly $2 \mid b_2$, $4 \mid a_6$ and $8 \mid b_8$, but since
    $2 \nmid (a+b^3\sqrt{-d})$, the polynomial
    $y^2+\frac{a_3}{4}y - a_6$ has distinct roots, so Step 8 of Tate's
    algorithm implies the reduction is of type $\rm{IV}^*$ and the
    conductor equals $v_2(\Disc(\Et))-6=2$.
		
  \item Suppose that $2$ splits and let $\id{p}_2$ be a prime dividing
    $2$. The primitive hypothesis implies that either one of $a$, $b$
    is even and the other is odd or both are odd. In the first case,
    $v_{\id{p}_2}(a_1)\ge 1$ and $v_{\id{p}_2}(a_3)=2$ hence we are again
    in Step 8 of Tate's algorithm (type $\rm{IV}^*$), therefore the
    conductor exponent is $2$. On the other hand, if both $a$ and $b$
    are odd, the model is not minimal, as $v_{\id{p}_2}(a_1) = 1$ and
    $v_{\id{p}}(a_3) \ge 3$; its minimal model has $\tilde{a_1}$ a
    unit (hence $\tilde{b_2}$ a unit) and the curve has type $\rm{I}_n$. In
    particular, its conductor exponent equals $1$.
		
  \item Suppose $2$ ramifies but $2 \nmid d$ and let $\pi$ be a local
    uniformizer. The hypothesis $(a,b,c)$ primitive implies that
    $v_{\pi}(a+b^3\sqrt{-d}) = 0$ (i.e. one of $a$ or $b$ is even but
    not both). The model is not minimal; the change of variables
    $y \to \pi^3y$, $x \to \pi^2 x$ gives a minimal model with
    valuations $v_{\pi}(\tilde{a_1})\ge 1$ and
    $v_{\pi}(\tilde{a_3})=1$. In particular, $v_{\pi}(\tilde{b_6})=2$
    so we are in Step 5 of Tate's algorithm which implies that the
    reduction has type $\rm{IV}$ and its conductor equals
    $v_{\pi}(\Delta(\Et))-2 = 2$.
  \item If $2 \mid d$ then $2 \nmid a$ (as the solution is primitive),
    so the change of variables $x \to 2^2x$, $y \to 2^3y$ gives a
    non-singular curve.
		
  \end{enumerate}
\end{proof}

At last, we need information on primes dividing $3$. 
\begin{lemma}
  Let $\id{p}_3$ be a prime ideal of $\Om_K$ dividing $3$.
  \begin{enumerate}
  \item If $3$ is inert in $K$ then $\Et$ has reduction type \rm{II} or \rm{III} at $3$, with
    $v_{3}(N(\Et))\in\{2,3\}$.
  \item If $3=\id{p}_3\overline{\id{p}}_3$ in $K$ then $\Et$ has reduction type \rm{II} or \rm{III} at $\id{p}_3$, with
    $v_{\id{p}_3}(N(\Et))=v_{\overline{\id{p}}_3}(N(\Et))\in\{2,3\}$, or the primes can be chosen so that
    $v_{\id{p}_3}(N(\Et))=2$ and
    $v_{\overline{\id{p}}_3}(N(\Et))=1$.
  \item If $3$ ramifies in $K$ then $\Et$ has reduction type $\rm{IV}^*$ at $\id{p}_3$, with
    $v_{\id{p}_3}(N(\Et))=8$.
  \end{enumerate}
  \label{lemma:case2conductor3}
\end{lemma}

\begin{proof}
  Let's consider the different cases:
  \begin{enumerate}
  \item If $3$ is inert, the primitive hypothesis implies that $c$ is
    not divisible by $3$ and $v_3(a_3)=0$ hence the singular point is
    not at the origin but it goes to the origin under the translation
    $(x,y) \to (x-a_3^6,y+a_3)$ (we are using that in the residue
    field raising to the eight power is the constant map). Let $a_1$
    and $a_3$ denote the corresponding coefficients of $\Et$ (to easy
    notation). Then the model becomes
    \begin{equation}
      \label{eq:3inerttypeIII}
      y^2+a_1xy + (3a_3-a_1a_3^6)y=x^3-3a_3^6x^2-a_1a_3x+(a_1a_3^7-a_3^{18}-2a_3^2).
    \end{equation}
		
    Let $\tilde{a_i}$ denote such coefficients. If $3 \mid b$ then
    $v_3(a_1)\ge 2$ so $v_3(\tilde{a_6})=1$ hence we are in Step 3 of
    Tate's algorithm, hence the curve has type $\rm{II}$ and the
    conductor exponent is $3$. If $3 \nmid b$, $v_3(a_1)=1$. If
    $9 \nmid a_1a_3^7-a_3^{18}-2a_3^2$ we are again in case $\rm{II}$
    (with exponent $3$).  Otherwise the following equality holds:
    \[
      \frac{a_1}{3} \equiv a_3^3\left(\frac{a_3^{16}+2}{3}\right)
      \pmod 3.
    \]
    The coefficient $\tilde{b_2}$ equals
    $-4a_1^2a_3^{18}+6a_1a_3^{13}+12a_3^{24}-3a_3^8$. Using the above
    equation a simple computation shows its valuation at $3$ equals
    $2$ hence the reduction type is $\rm{III}$ and the conductor
    exponent equals $2$.
		
  \item If $3$ splits in $K$, let $\id{p}_3$ be a prime dividing
    it. If $3\mid a$ then $3 \nmid b$ hence $v_{\id{p}_3}(a_1)=1$ and
    $v_{\id{p}_3}(a_3)=0$.  This situation matches the previous case
    and a similar computation proves that the type is $\rm{II}$ or
    $\rm{III}$ and the exponent valuation $3$ or $2$ at both $\id{p}_3$
    and $\overline{\id{p}}_3$.  If $3 \mid b$ then $3 \nmid a$, hence
    $v_{\id{p}_3}(a_1)\ge 2$ and $v_{\id{p}_3}(a_3)=0$; as in the
    previous case this corresponds to type $\rm{II}$ with conductor
    exponent $3$.
		
    Suppose then that $3 \nmid ab$. Then one of the primes (say
    $\id{p}_3$) divides $a+b^3\sqrt{-d}$ while the other does not.
    Since $c^p = (a+b^3\sqrt{-d})(a-b^3\sqrt{-d})$ the assumption
    $p\ge 5$ implies that (without loss of generality)
    $v_{\id{p}_3}(a+b^3\sqrt{-d})>3$ so $\id{p}_3$ divides the denominator
    of the $j$-invariant. Furthermore, the model is not minimal, and
    under the usual change of variables (sending
    $(x,y) \to (3^2x,3^3y)$ we get a curve with multiplicative
    reduction, hence the discriminant exponent equals $1$. At the
    prime $\overline{\id{p}}_3$ the curve is a quadratic twist (by the
    character of conductor $3$) of a curve with multiplicative
    reduction, hence the statement.
		
  \item If $3$ ramifies in $K$ then $3 \mid d$ and the primitive
    hypothesis implies that $3 \nmid a$. Let $\id{p}$ denote the prime
    ideal dividing $3$ in $K$. Then $v_{\id{p}_3}(a_1) \ge 2$ and
    $v_{\id{p}_3}(a_3) = 2$ hence we are in Step 8 of Tate's algorithm,
    the reduction type is $\rm{IV}^*$ and the conductor exponent
    equals $14-6=8$.
  \end{enumerate}
\end{proof}

\begin{remark}
	If $3$ is inert in $K/\Q$ and the curve has type $\rm{III}$
	reduction (the case of conductor valuation $2$), the change of
	variables $(x,y) \to (\sqrt[4]{3}x,\sqrt{3}y)$ in
	equation~(\ref{eq:3inerttypeIII}) gives a curve with good
	reduction. Since the fourth roots of unity are in $K_3$, the local
	type of the Weil-Deligne representation is that of a principal
	series (whose inertia is given by an order $4$ character).
	\label{rem:case3inert}
\end{remark}

\begin{lemma}
	Suppose that $q$ is a rational prime ramified at $K/\Q$ such that $q\nmid6$ and let	$\id{q}$ denote the (unique) prime ideal in $\Om_K$ dividing $q$. Then $\Et$ has reduction type $\rm{IV}^*$ at $\id{q}$ and
	$v_{\id{q}}(N(\tilde{E}_{(a,b,c)}))=2$.
	\label{lemma:conductorramifiedprimes}
\end{lemma}
\begin{proof}
  Since $(a,b,c)$ is a primitive solution, then $q \nmid a$ so $v_{\id{q}}(a_3)=2$ and
  $v_{\id{q}}(b_6)=4$. Also, $v_{\id{q}}(b_2)\ge 2$ which implies that
  we are in Step 8 of Tate's algorithm so the result follows from
  Lemma~\ref{lemma:discvaluation}.
\end{proof}

Once again, all trivial solutions correspond to special curves.
\begin{lemma}
  The trivial solution $(0,0,0)$ gives rise to a singular curve. All
  other trivial primitive solutions of~(\ref{eq:ben-chen}) are the
  following:
  \begin{itemize}
  \item The solution $(0,\pm1,1)$ (when $d=1$), corresponding to the
    elliptic curve with lmfdb label \lmfdbecnf{2.0.4.1}{324.1}{a}{3}
    and complex multiplication by $\Z[\sqrt{-1}]$.
	 
  \item The solution $(\pm1,0,1)$, corresponding to a cubic twists by
    $\sqrt[3]{d}$ of the elliptic curve with lmfdb
    label~\lmfdbec{108}{a}{2}. Such a twist has complex multiplication
    by $\Z\left[\frac{1+\sqrt{-3}}{2}\right]$.
\end{itemize}
\end{lemma}
 \begin{proof}
   The solution $(0,0,0)$ clearly gives a singular curve. The other
   trivial primitive solutions must be of the form $(\pm 1,0,1)$ or
   $(0,\pm 1,1)$. The later case can only occur when $d=1$ (since $d$
   is assumed to be square-free) giving the first equation. The
   solution $(\pm 1,0,1)$ corresponds to the elliptic curve
   \[
     \tilde{E}_{(\pm 1,0,1)}: y^2 \pm 4dy=x^3.
   \]
   When $d=1$, it corresponds to the elliptic curve with lmfdb label
   \lmfdbec{108}{a}{2}, which has complex multiplication by
   $\Z\left[\frac{1+\sqrt{-3}}{2}\right]$. The fact that the cubic
   roots of unity act on the curve allows to define cubic twists of
   the curve, and the general case is precisely the cubic twist by
   $\sqrt[3]{d}$ of the case $d=1$.
\end{proof}

\begin{remark}
  The trivial solution $(\pm 1,0,1)$ corresponds to an
  elliptic curve over $K$ with complex multiplication whose conductor has
  valuation:
  \begin{itemize}
  \item $2$ for all odd primes $\id{q}$ dividing $d$ but not $3$,
    
  \item $2$ if $\id{p}_2\mid 2$ and $2 \nmid d$,
    
  \item $0$ if $\id{p}_2\mid 2$ and $2 \mid d$,
    
  \item $3$ if $\id{p}_3\mid 3$ and $3 \nmid d$,
    
  \item $8$ if $\id{p}_3\mid 3$ and $3 \mid d$.
  \end{itemize}
\label{rem:conductorCM2}
\end{remark}

\section{Construction of the Hecke character}
\label{section:character}
  For a number field $L$, let $\II_L$ denote its id\`ele group and
  $\Cl(L)$ denote its class group.  Class field theory relates finite
  characters of $\Gal_L$ with finite characters of the id\`ele group
  $\II_L$. We will make constant use of this relation, and will denote
  by the same letter both incarnations of the same object (and hope
  there is no confusion on doing that).

  Let $\tau \in \Gal_\Q$ and $\chi: \II_L \to \CC^\times$ be a finite
  order Hecke character. Denote by $^\tau\chi$ the Hecke character
  given on an element $\alpha \in \II_L$ by
\[
^\tau\chi(\alpha) = \chi(\tau(\alpha)).
\]
Via class field  theory, the character $^\tau\chi$ corresponds to
the character on $\Gal_L$ given by 
$^\tau\chi(\sigma) = \chi(\tau \sigma \tau^{-1})$. In general, if
$\rho:\Gal_L \to \GL_n(\overline{\Q_p})$ is a Galois
representation, and $\tau \in \Gal_\Q$, we denote by $^\tau\rho$ the
Galois representation whose value at $\sigma \in \Gal_L$ equals
$^\tau\rho(\sigma) = \rho(\tau \sigma \tau^{-1})$.

For $K$ a quadratic extension of $\Q$, a representation
$\rho:\Gal_K \to \GL_2(\overline{\Q_p})$ extends to a $2$-dimensional
representation of $\Gal_\Q$ if and only if $^\tau\rho \simeq \rho$, where
$\tau \in \Gal_\Q$ is any element whose restriction to $\Gal_K$ is not
the identity (although this result is well known for experts, a proof
will be given in Theorem~\ref{thm:levelandnebentypus}).

Let $t$ be an integer, and let $\psi_t$ denote the character of
$\Gal_\Q$ corresponding to the quadratic extension
$\Q(\sqrt{t})/\Q$. Proposition~\ref{prop:qcurveE} implies that for any
prime $p$, $^\tau\rho_{\E,p}$ is isomorphic to
$\rho_{\E,p}\otimes \psi_{-2}$, while Proposition~\ref{prop:qcurveE2}
implies that for any prime $p$, $^\tau\rho_{\Et,p}$ is isomorphic to
$\rho_{\Et,p}\otimes \psi_{-3}$.  If we construct a Hecke character
$\chi_t:\Gal_K\to \overline{\Q}^\times$ satisfying that
$^\tau\chi_t = \chi_t \cdot \psi_{-t}$ (as characters of $\Gal_K$)
then the twisted representations $\rho_{\E,p} \otimes \chi_2$
(respectively $\rho_{\Et,p}\otimes \chi_3$) is isomorphic to
$^\tau(\rho_{\E,p} \otimes \chi_2)$ (respectively to
$^\tau(\rho_{\Et,p}\otimes \chi_3)$) and so as explained in the
introduction (see also Theorems \ref{thm:levelandnebentypus} and
\ref{thm:levelandnebentypus2}) it does extend to a two dimensional
representations of $\Gal_\Q$. Furthermore, an explicit description of
$\chi$ and of its conductor allows  to give a formula for the level
and Nebentypus of the extended rational representation. This leads to
the following problem.

\medskip

\noindent{\bf Problem 1:} Let $\psi_t$ be the quadratic character of
$\Gal_\Q$ corresponding to the extension $\Q(\sqrt{t})/\Q$. Find a
Hecke character $\chi_t$ of $\Gal_K$ such that
$^\tau\chi_t = \chi_t \cdot \psi_{-t}$.

\medskip

Our main result is to give a solution the previous problem for $t=2$ (corresponding to
equation~(\ref{eq:24p})) and for $t$ a prime number congruent to $3$
modulo $4$ (corresponding to equation~(\ref{eq:ben-chen})). To explain how our construction works, consider the natural short exact sequence
\begin{equation}
  \label{eq:ideleSES}
  \xymatrix{
    0 \ar[r] & L^\times \cdot(\prod_{\id{q}}\Om_{\id{q}}^\times \times (L \otimes \RR)^\times) \ar[r]  & \II_L \ar[r]^{\Id} & \Cl(L) \ar[r] & 0.
  }
\end{equation}
We start defining the Hecke character $\chi_t$ at elements of the
first term of the short exact sequence and then extend it to elements
of $\II_L$ that map via $\Id$ to representatives of the class
group. Let us make one important remark, as the picture might be a
little misleading: our character $\chi_t$ will not be trivial at units
(the first term of the sequence), hence it will not be a character of
the whole class group!  (it will be a character of a suitable ray
class group though).

Recall that Hecke characters are trivial at elements of
$L^\times$, hence we are only left to define our character on local units (which
determines the ramification behavior of the abelian extension cut out
by the kernel of the character $\chi_t$).
Note that
$(\prod_{\id{q}}\Om_{\id{q}}^\times \times (L \otimes \RR)^\times)
\cap L^\times = \Om_L^\times$, hence there is a compatibility
condition that always needs to be checked:

\medskip
\noindent {\bf Compatibility condition:} the product of the local
components evaluated at a unit equals $1$, i.e.
\begin{equation}
	\label{eq:existencehechechar}
	\prod_{v}\chi_{t,v}(\epsilon)=1  
      \end{equation}
for all $\epsilon \in \Om_L^\times$. Verifying this property is what makes construction of Hecke characters quite hard in general.
\medskip

In our case, the field $L$ is the imaginary quadratic field $K$, so
its set of units is well known. Our character $\chi_t$ will be ramified at
most at the odd primes ramifying in $K/\Q$ (with conductor exponent
one), at primes dividing $2$ and at primes dividing $t$.

Let $\norm:\II_K \to \II_\Q$ be the norm map. The way we verify the
compatibility condition of our defined character (and other properties
it satisfies) is via the construction of an auxiliary rational Hecke
character $\varepsilon_t$ (that will end up being the Nebentypus of
the extended representation) unramified outside $2td$ with the
following key properties relating $\varepsilon_t$ to $\chi_t$:
\begin{enumerate}
\item The local character $\chi_{t,\id{p}}$ satisfies that
  $^\tau\chi_{t,\id{p}}=\chi_{t,\id{p}} \cdot ((\psi_{-t})_p\circ
  \norm)$.
\item Let $p$ be an odd prime ramified in $K/\Q$, and let $\id{p}$ the unique
  prime of $\Om_K$ dividing it. Then under the field isomorphism
  $\Om_K/\id{p} \to \Z/p$, the local character $\chi_{t,\id{p}}$ equals
  $\varepsilon_{t,p} \delta_p$, where $\delta_p$ is the quadratic
  character of $(\Z/p)^\times$.
\item An extra condition at primes dividing $2$ so that a
  compatibility conditions on units holds (a condition similar to
  (\ref{eq:existencehechechar}) for the units of $K$).
\item For all $\sigma \in \Gal_K$,
  $\chi_t(\sigma)^2 = \varepsilon_t(\sigma)$ (or in terms of id\`eles,
  $\chi_t^2 = \varepsilon_t \circ \norm$).
\end{enumerate}

Let us explain the role of each property we imposed.  The first
condition is needed for $\chi_t$ to locally solve the problem. The
second condition will play a crucial role while proving the
compatibility condition (via quadratic reciprocity). It is a local
version of the last one.

\begin{lemma} Let $p$ be a rational prime and $\id{p}$ a prime of $\Om_K$
  dividing it. Then the second condition implies that the fourth condition
  holds locally, namely
  $\chi_{t,\id{p}}=\varepsilon_{t,p} \circ \norm$.
\label{lemma:charnorm}
\end{lemma}
\begin{proof}
  For primes $\id{p} \nmid 2td$, both characters are trivial, hence
  the statement trivially holds.  For odd primes $\id{p}$ (of norm $p$) such that
  $\id{p} \nmid t$ and $\id{p} \mid d$, recall that the
  restriction of $\varepsilon_{t,p}$ to $\Gal_{K_{\id{p}}}$ equals (as
  Hecke characters) $\varepsilon_{t,p} \circ \norm$, where
  $\norm:K_{\id{p}} \to \Q_p$ is the norm map. Since $p$ ramifies in
  $K/\Q$ the local norm map (modulo $\id{p}$) is given by $x \mapsto x^2$,
  so we get the equality
\[
\chi_{t,\id{p}}^2(x) = \varepsilon_{t,p}^2(x) = \varepsilon_{t,p}(x^2)=\varepsilon_{t,p} \circ \norm(x).
\]
\end{proof}
The third condition is needed to prove the compatibility
condition. The first three conditions will be enough to define the
character $\chi_t$ at elements of the first term of
(\ref{eq:ideleSES}). The last condition will be used to extend the
character to id\`eles that are representatives of the class group of
$K$.

The general strategy to prove existence of the characters
$\varepsilon_t$ and $\chi_t$ with the above properties, is to split
the set of odd primes $\{p : p \text{ ramifies in }K/\Q\}$ into four
sets. More concretely, for $t=2$ they will be divided depending on
their congruence modulo $8$, while for $t$ an odd prime, they will be
divided depending on whether $p$ is a square modulo $t$ or not and on
whether $p$ is a square modulo $4$ or not. Then for primes $p$ in each
set we give an explicit definition of the local character
$\varepsilon_{t,p}$ and $\chi_{t,\id{p}}$ satisfying the previous four
properties. The description (and proof) depends on whether $t=2$ or
$t$ is an odd prime congruent to $3$ modulo $4$, so each case will be
considered in a different section. To easy notation, in each section
the subscript $t$ will be removed.

\subsection{The case $t=2$} 
\label{section:t=2}
Let $K=\Q(\sqrt{-d})$ with $d$ a positive square-free integer. Split
the odd prime divisors of $d$ into four different sets, namely:
\[
Q_i =\{ p \text{ prime}\; : \; p \mid d, \quad p \equiv i \pmod 8\},
\]
for $i=1, 3, 5, 7$.

\noindent{\bf The character $\varepsilon$:} Define an even character
$\varepsilon:\II_\Q \to \CC^\times$ ramified at the set of primes of
$Q_3 \cup Q_5$ and sometimes in $\{2\}$, with local component
$\varepsilon_p$ given by:
\begin{itemize}
\item For primes $p \in Q_1 \cup Q_7$, the character
  $\varepsilon_p:\ZZ_p^\times \to \CC^\times$ is trivial.
	
\item For primes $p \in Q_3$, the character
  $\varepsilon_p = \delta_p$, the quadratic character defined by
  $\delta_p(n) = \kro{n}{p}$.
	
\item For $p \in Q_5$, let $\varepsilon_p$ be a character of order $4$
  and conductor
  $p$.
\item The character $\varepsilon_\infty$ (the archimidean component)
  is trivial.
\end{itemize}

Before defining the character at the prime $2$, let us introduce some
notation. Let $\psi_{-1}$, $\psi_{2}$, $\psi_{-2}$ be the characters
of $\ZZ$ corresponding to the quadratic extensions $\QQ(\sqrt{-1})$,
$\QQ(\sqrt{2})$ and $\QQ(\sqrt{-2})$ respectively and let
$\delta_{-1}$, $\delta_{2}$, $\delta_{-2}$ be their local component at
the prime $2$ (see Table~\ref{table:characters} for their values).
\begin{table}[h]
	\begin{tabular}{|l|c|c|c|c|}
		\hline
		Char & $1$ & $3$ & $5$ & $7$\\
		\hline\hline
		$\delta_{-1}$ & $1$ & $-1$ & $1$& $-1$\\
		\hline
		$\delta_{-2}$ & $1$ & $1$ & $-1$& $-1$\\
		\hline
		$\delta_{2}$ & $1$ & $-1$ & $-1$& $1$\\
		\hline
	\end{tabular}
	\caption{\label{table:characters}}
\end{table}
\begin{itemize}
\item Define $\varepsilon_2=\delta_{-1}^{\#Q_3+\#Q_5}$.
\end{itemize}

By construction, the character $\varepsilon$ satisfies the compatibility condition,
namely
\[
\prod_p\varepsilon_p(-1)\varepsilon_{\infty}(-1) = \prod_{p \in Q_3\cup Q_5}\varepsilon_p(-1) \varepsilon_2(-1) = (-1)^{\#Q_3+\#Q_5}\varepsilon_2(-1)=1.
\]

This gives a well defined Hecke
character $\varepsilon$ of $\II_\QQ$ corresponding to a totally real
field $L$ whose degree equals $1$ if $Q_3=Q_5=\emptyset$, $2$ if $Q_3\neq Q_5 = \emptyset$ and $4$
otherwise. By class field theory, $\varepsilon$ gets identified with a
character $\varepsilon:G_\Q \to \overline{\Q}$.  Let
$N_\varepsilon$ denote its conductor, given by $N_\varepsilon= 2^e\prod_{p \in Q_3 \cup Q_5}p$, where $e=0$ if $\#Q_3+\#Q_5$ even and $2$ otherwise.
\begin{remark}\label{remark}
  The possible values for $\#Q_3$, $\#Q_5$ and $\#Q_7$ depending on
  the congruence of $d$ modulo $8$ are given on
  Table~\ref{table:parity}. Note in particular that when $d$ is odd,
  the definition of $\varepsilon_2$ depends only on the value of
  $d \pmod 8$ and not on the parity of $\#Q_3$ and $\#Q_5$.
  \begin{table}[h]
    \begin{tabular}{|l|c|c|c||l|c|c|c|}
      \hline
      $d$ & $\#Q_3$ & $\#Q_5$ & $\#Q_7$ &$d$ & $\#Q_3$ & $\#Q_5$ & $\#Q_7$ \\
      \hline\hline
      $1$ & $0$ & $0$ & $0$ &$5$ & $0$ & $1$ & $0$\\
          & $1$ & $1$ & $1$ &  & $1$ & $0$ & $1$\\
      \hline
      $3$ & $0$ & $1$ & $1$ &$7$ & $0$& $0$ & $1$\\
          & $1$ & $0$ & $0$ &    & $1$ & $1$ & $0$\\
      \hline
      $2$ & $0$ & $0$ & $0$ & $6$ & $0$& $0$ & $1$\\
          & $0$ & $1$ & $0$ & & $0$ & $1$ & $1$\\
          & $1$ & $0$ & $1$ & & $1$ & $0$ & $0$\\
          & $1$ & $1$ & $1$ & & $1$ & $1$ & $0$ \\
      \hline
    \end{tabular}
    \caption{\label{table:parity}}
	\end{table}
\end{remark}

\begin{thm}
  There exists a Hecke character $\chi:\Gal_K \to \overline{\QQ}$ such
  that:
  \begin{enumerate}
  \item $\chi^2(\sigma) = \varepsilon(\sigma)$ for all $\sigma \in \Gal_K$,
		
  \item $\chi$ is unramified at primes not dividing
    $2\prod_{p \in Q_1 \cup Q_5 \cup Q_7}p$,
		
  \item If $\tau \in \Gal_\Q$ is not the identity on $K$, 
    $^\tau\chi = \chi \cdot \psi_{-2}$ as characters of $\Gal_K$.
  \end{enumerate}
  \label{thm:charexistence}
\end{thm}

\begin{proof} Recall that for each ramified prime $p$ in $K$ there is
  a natural group isomorphism
  $(\Om_{\id{p}}/\id{p})^\times \simeq (\ZZ/p)^\times$ (where
  $\Om_{\id{p}}$ denotes the completion of $\Om_K$ at $\id{p}$). In
  particular, we will denote by $\delta_p$ or $\varepsilon_p$ the same
  character of $(\Om_{\id{p}}/\id{p})^\times$. Following the strategy
  described above,
  define $\chi_\id{p}:\Om_{\id{p}}^\times \to \CC^\times$  by:
	
\begin{itemize}
\item If $\id{p}$ is an odd (i.e. $\id{p}\nmid 2$) unramified prime,
  $\chi_{\id{p}}$ is the trivial character. 		
\item If $\id{p}$ is an odd ramified prime,
  \begin{equation}
    \label{eq:ramprimedef}
    \chi_{\id{p}}= \varepsilon_p \delta_p.
    \end{equation}
  \item The archimidean component of $\chi$ is trivial.
\end{itemize}
Its local definition at places dividing $2$ is more involved. Suppose
that $2$ does not split in $K$, and let $\id{p}_2$ denote the unique
ideal dividing $2$. The character $\chi_{\id{p}_2}$ has conductor
dividing $2^3$; the group structure of $(\Om_{\id{p}_2}/2^n)^\times$
and its generators are given in
Table~\ref{table:quotientstructure}. The generators are ordered so
that the order of the generator $i$ matches the $i$-th factor of the
group structure, while the elements norms are modulo $8$.
\begin{table}[h]
  \begin{tabular}{|c|c|c|c|c|}
    \hline
    $d$ & $n$ & Structure & Generators & Norms \\
    \hline\hline
    $1$ & $3$ & $\Z/4 \times \Z/4 \times \Z/2$& $\{\sqrt{-d}, 1+2\sqrt{-d},5\}$ & $\{1,5,1\}$ \\
    \hline
    $3$ & $3$ & $\F_3 \times \Z/4 \times \Z/2 \times \Z/2$ &$\{\zeta_3,\sqrt{-d},3+2\sqrt{-d},-1\}$ & $\{1,3,5,1\}$\\
    \hline
    $5$ & $3$ & $\Z/4 \times \Z/4 \times \Z/2$ &$\{\sqrt{-d},1+2\sqrt{-d},-1\}$ & $\{5,5,1\}$\\
    \hline
    even & $2$& $\Z/4 \times \Z/2 $ & $\{1+\sqrt{-d},-1\}$ & $\{3,1\}$ \\
    \hline  
  \end{tabular}
  \caption{\label{table:quotientstructure}}
\end{table}
Define~$\chi_{\id{p}_2}$ on the set of generators as follows:
\begin{itemize}
  \item If $d \equiv 1 \pmod{16}$, $\chi_{\id{p}_2}(\sqrt{-d})=1$,
  $\chi_{\id{p}_2}(1+2\sqrt{-d})=1$, $\chi_{\id{p}_2}(5)=-1$. 

  \item If $d \equiv 9 \pmod{16}$, $\chi_{\id{p}_2}(\sqrt{-d})=-1$,
  $\chi_{\id{p}_2}(1+2\sqrt{-d})=1$, $\chi_{\id{p}_2}(5)=-1$. 

\item If $d \equiv 3 \pmod 8$, $\chi_{\id{p}_2}(\zeta_3)=1$,
    $\chi_{\id{p}_2}(\sqrt{-d})=i$, $\chi_{\id{p}_2}(3+2\sqrt{-d})=1$,
    $\chi_{\id{p}_2}(-1)=1$. 
    
  \item If $d \equiv 5 \pmod {16}$, $\chi_{\id{p}_2}(\sqrt{-d})=1$, $\chi_{\id{p}_2}(1+2\sqrt{-d})=1$, $\chi_{\id{p}_2}(-1)=-1$.
  \item If $d \equiv 13 \pmod {16}$, $\chi_{\id{p}_2}(\sqrt{-d})=-1$, $\chi_{\id{p}_2}(1+2\sqrt{-d})=1$, $\chi_{\id{p}_2}(-1)=-1$.
	
\item If $d \equiv 2 \pmod 8$ and $\#Q_3+\#Q_5$ is even,
  $\chi_{\id{p}_2}(1+\sqrt{-d})=1$,
  $\chi_{\id{p}_2}(-1)=1$.
		
\item If $d \equiv 2 \pmod 8$ and $\#Q_3+\#Q_5$ is odd,
  $\chi_{\id{p}_2}(1+\sqrt{-d})=i$,
  $\chi_{\id{p}_2}(-1)=-1$.
		
\item If $d \equiv 6 \pmod 8$ and $\#Q_3+\#Q_5$ is even,
  $\chi_{\id{p}_2}(1+\sqrt{-d})=1$,
  $\chi_{\id{p}_2}(-1)=-1$.
		
\item If $d \equiv 6 \pmod 8$ and $\#Q_3+\#Q_5$ is odd,
  $\chi_{\id{p}_2}(1+\sqrt{-d})=i$,
  $\chi_{\id{p}_2}(-1)=1$.
\end{itemize}
At last,
\begin{itemize} 
\item If $d \equiv 7 \pmod 8$, the prime $2$ splits as
  $2 = \id{p}_2 \overline{\id{p}_2}$. Let
  $\chi_{\id{p}_2}:=\delta_{-2}$ and $\chi_{\overline{\id{p}_2}}:=1$
  (trivial) or take $\chi_{\id{p}_2}:=\delta_{2}$ and
  $\chi_{\overline{\id{p}_2}}:=\delta_{-1}$. 
\end{itemize}
To make the proofs  consistent, we denote  $\chi_2 = \prod_{\id{p}_2\mid2}\chi_{\id{p}_2}$.

Define $\chi$ on
$K^\times \cdot (\prod_{\id{q}}\Om_{\id{q}}^\times \times \CC^\times)$
to be trivial at elements of $K^\times$ and as the product of the
local components at elements of the second factor. Let us start proving
that even when we do not know that that our character $\chi$
satisfies the compatibility condition, nor have extended it to the whole
id\`ele group, it satisfies the three properties of the Theorem at
elements of
$K^\times \cdot (\prod_{\id{q}}\Om_{\id{q}}^\times \times
\CC^\times)$.

\begin{enumerate}
\item We need to verify that the equality
\[
\chi^2 = \varepsilon \circ \norm
  \]
  holds for all elements of
  $K^\times \cdot (\prod_{\id{q}}\Om_{\id{q}}^\times \times
  \CC^\times)$. The statement is clear for elements of $K^\times$ (as
  both terms are trivial) and at $\CC^\times$ for the same reason, so
  we are left to verify it for each local component. The proof for odd
  prime ideals $\id{p}$ is the following: if $\id{p}$ is unramified in
  $K/\Q$, then the result is clear (as all characters are trivial),
  while at ramified primes, it follows from (\ref{eq:ramprimedef})
  together with Lemma~\ref{lemma:charnorm}.

  At last, it is easy to verify that
  $\chi_2^2 = \varepsilon_2 \circ \norm$ using the character values
  in Table~\ref{table:characters}, the parity of
  Table~\ref{table:parity} and the norm of the generators given in
  Table~\ref{table:quotientstructure}.
\item From the definition of $\varepsilon$ and $\chi$ it is clear that
  the ramification hypothesis is fulfilled.
  
\item From its definition it is clear that for all odd primes
  $^\tau\chi_{\id{p}} = \chi_{\id{p}}$, so the third condition is also
  locally fulfilled. The reason it holds for primes $\id{p}_2$
  dividing $2$ is that
  $^\tau\chi_{\id{p}_2} \cdot \chi_{\id{p}_2} = \chi_{\id{p}_2} \circ
  \norm$, hence
  $^\tau\chi_{\id{p}_2} = \chi_{\id{p}_2}^{-1} \cdot( \chi_{\id{p}_2}
  \circ \norm)$. An easy case by case computation on the generators
  shows that
  $^\tau\chi_{\id{p}_2} = \chi_{\id{p}_2} \cdot (\delta_{-2}\circ
  \norm)$.
\end{enumerate}

An important property of our character at $2$ is that its restriction
to the $2$-adic integers always satisfies
\begin{equation}
  \label{chi2}
  \chi_2|_{\Z_2^\times}=\delta_{2}^{v_2(d)+1}\delta_{-1}^{\#Q_5+\#Q_7}.
\end{equation}
	
\noindent{\bf Compatibility:} the subgroup of units in $K$ is
generated by roots of unity of order $2$, $6$ and $4$ (for
$\Q(\sqrt{-1})$). Since our characters have order a power of $2$, the
compatibility relation at roots of order $3$ (if $K$ has one) is
trivial. If $K=\Q(\sqrt{-1})$, all sets $Q_i$ for $i=1,3,5,7$ are empty
and the compatibility at $\sqrt{-1}$ follows from the fact that
$\chi_2(\sqrt{-1})=1$ in such case.

Let us make the following abuse of notation: for $\id{p}$ a prime
ideal of $\Om_K$ let us denote by $\id{p} \in Q_i$ the fact that 
$\id{p}\cap \Z$ is in such set. Then compatibility at $-1$ follows from
\begin{equation}
  \label{eq:compat1}
  \chi(-1) = \prod_{\id{p}}\chi_{\id{p}}(-1) = \prod_{\id{p} \in Q_1
    \cup Q_5 \cup Q_7}\chi_{\id{p}}(-1) \chi_2(-1)=(-1)^{\#Q_5+\#Q_7}\delta_{-1}(-1)^{\#Q_5+\#Q_7}=1.
\end{equation}
	
\medskip
	
\noindent {\bf Extension:} As explained before, to extend $\chi$ to
the whole id\`ele group $\II_K$, it is enough to define it on id\`eles
whose image via $\Id$ (in (\ref{eq:ideleSES})) generate the class
group of $K$.  For that purpose, let $\{\id{r}_1,\cdots,\id{r}_h\}$ be
prime ideals of $K$ whose class generate $\Cl(K)$ (we can and do
assume they are not ramified in $K/\Q$). Since $\id{r}_i$ is not
principal, it must split in $K/\Q$, so if $r_i = \norm(\id{r}_i)$, the
element $a_i$ in $\II_K$ with trivial infinite component and finite
components:
\[
  (a_i)_\id{p}=
  \begin{cases}
    r_i & \text{ if }\id{p} = \id{r}_i,\\
    1 & \text{ otherwise}.
  \end{cases}
\]
is a preimage of $\id{r}_i$ under $\Id$.  The value $\chi(a_i)$ cannot
be arbitrary. For example, suppose that $\id{r}_i$ has odd order in
the class group, so there exists an ideal $\id{t}$ such that
$\id{t}^2$ lies in the same class as $\id{r}_i$. In particular, if
$b_i$ denotes an id\`ele in the preimage of $\id{t}$ by $\Id$, there
must exist $\alpha \in K^\times$,
$u \in \prod_{\id{q}}\Om_{\id{q}}^\times \times \CC^\times$ such that
$a_i = \alpha u b_i^2$. Since we want $\chi$ to be a character, it
must hold that
\begin{equation}
  \label{eq:charextodd}
  \chi(a_i) = \chi(u) \chi(b_i^2) = \chi(u) \chi(b_i)^2=\chi(u) \varepsilon(\norm(b_i)),  
\end{equation}
so there is a unique possible value for $\chi(a_i)$.

Since $\Cl(K)$ is a finite abelian group, and $\chi$ is
multiplicative, we only need to understand how to define $\chi$ for
ideal classes $\id{r}_i$ whose order is a power of $2$, so let us suppose that this is the case. Let
\begin{equation}
  \label{eq:charext}
\chi(a_i) = \sqrt{\varepsilon(\norm(a_i))}  
\end{equation}
(it does not really matter which square root one takes). Then we just
extend $\chi$ multiplicatively to the whole id\`ele group. Recall that
we already proved that $\chi^2$ and $\varepsilon\circ \norm$ coincide
on
$K^\times \cdot(\prod_{\id{q}}\Om_{\id{q}}^\times \times \CC^\times)$
so with this definition they coincide on the whole id\`ele group
$\II_K$.
	
There is a caveat here: it is not clear at all why the multiplicative
function that we defined is well defined! Once again, a power of the
id\`ele $a_i$ lies in
$K^\times \cdot(\prod_{\id{q}}\Om_{\id{q}}^\times \times \CC^\times)$
hence we need to prove that our definition really extends the previous
one. To avoid confusions, for the time being let $\tilde{\chi}$ denote
the function whose value at the id\`eles $a_i$ is given by
(\ref{eq:charextodd}) and (\ref{eq:charext}) and $\chi$ the character
on
$K^\times \cdot(\prod_{\id{q}}\Om_{\id{q}}^\times \times \CC^\times)$
defined before; if the id\`ele $a_i$ corresponds to an ideal of order
$s$ in the class group, the consistency relation translates into the
equality $\tilde{\chi}(a_i)^s=\chi(a_i^s)$. It is enough to prove it
in the following two cases:

\vspace{1pt}

\noindent $\bullet$ The ideal $\id{r}_i$ has odd order $s$ in the
class group. Then as explained before, there exists $b_i \in \II_K$,
$\alpha \in K^\times$ and
$u \in \prod_{\id{q}}\Om_{\id{q}}^\times \times \CC^\times$ such that
$a_i = \alpha u b_i^2$. Note that $\Id(b_i^s)$ is also a principal
ideal, hence $b_i^s$ lies in
$K^\times\cdot(\prod_{\id{q}}\Om_{\id{q}}^\times \times \CC^\times)$. Then
by~(\ref{eq:charextodd}) and the fact that
$\chi^2 = \varepsilon \circ \norm$ on
$K^\times \cdot (\prod_{\id{q}}\Om_{\id{q}}^\times \times
\CC^\times)$, we get that
\[
  \tilde{\chi}(a_i)^s =
  \chi(u)^s\varepsilon(\norm(b_i))^s=\chi(u)^s\varepsilon(\norm(b_i^s))=\chi(u)^s\chi^2(b_i^{s})=\chi(u^s
  b_i^{2s})=\chi(a_i^s).
\]

\vspace{1pt}

\noindent $\bullet$ The ideal $\id{r}_i$ has order a power of $2$, say
$2^s$ and is not a square (since it generates the class group).
By definition we want to prove the equality
\[
\chi(a_i^{2^s})=\tilde{\chi}(a_i)^{2^s} = \varepsilon(\norm(a_i))^{2^{s-1}} = \varepsilon(\norm(a_i^{2^{s-1}})).
\]
Let $b_i = a_i^{2^{s-1}}$, an id\`ele whose square satisfies that
$\Id(b_i^2)$ is principal. It is enough to prove that for any such
id\`ele, the following equality holds:
\begin{equation}
  \label{eq:charwelldef}
  \chi(b_i^2) = \varepsilon(\norm(b_i)).
\end{equation}
It is well known that the two torsion subgroup of the class group is
generated by the prime ideals $\id{q}=\langle q,\sqrt{-d}\rangle$,
where $q$ is an odd prime dividing $d$, and also the prime
$\id{p}_2=\langle 2,1+\sqrt{-d}\rangle$ when $d \equiv 1 \pmod 4$. Let $q$ be
an odd prime dividing $d$ and let $b_q$ be the id\`ele of $K$ defined
by
\[
  (b_q)_{\id{p}} =
  \begin{cases}
    1 & \text{ if }\id{p} \neq \id{q},\\
    \sqrt{-d} & \text{ if }\id{p} =\id{q}.
  \end{cases}
\]
Then $\Id(b_q) = \id{q}$. Similarly, if $d \equiv 1 \pmod 4$, let $b_2$ be the id\`ele define by 
\[
  (b_2)_{\id{p}} =
  \begin{cases}
    1 & \text{ if }\id{p} \neq \id{p}_2,\\
    1+\sqrt{-d} & \text{ if }\id{p} =\id{p}_2.
  \end{cases}
\]

\vspace{2pt}
\noindent{\bf Claim:} it is enough to prove (\ref{eq:charwelldef}) for the
elements $b_q$.
\vspace{2pt}

Suppose that equality~(\ref{eq:charwelldef}) holds for them. Let $b$
be an id\`ele satisfying that $\Id(b)^2$ is principal. Then
\[
\Id(b) = \prod_{q \mid 2d} \Id(b_q)^{\epsilon_q},
\]
for some $\epsilon_q \in \{0,1\}$, so there exists
$\alpha \in K^\times$, and
$u \in \prod_{\id{q}}\Om_{\id{q}}^\times \times \CC^\times$ such that
$b = \alpha u \prod_{q \mid 2d} b_q^{\epsilon_q}$. By the multiplicative property of the character $\chi$, 
\[
\chi(b^2) = \chi(u)^2 \prod_{q \mid 2d} \chi(b_q^2)^{\epsilon_q} = \varepsilon(\norm(u)) \prod_{q \mid 2d}\varepsilon(\norm(b_q))^{\epsilon_q} = \varepsilon(\norm(b)).
\]
Let $q$ be an odd prime dividing $d$.  To prove (\ref{eq:charwelldef})
for the elements $b_q$, we compute both sides of the equality and
prove that they coincide.  Note that $b_q^2 = q c_q$, where
$q \in K^\times$ and $c_q = b_q^2/q$ has the property that it is a
unit at all finite places. Then
\begin{equation}
  \label{eq:nada}
  \chi(b_q^2) =\chi_{\id{q}}\left(\frac{-d}{q}\right)\chi_2\left(\frac{1}{q}\right)\prod_{\id{p} \in Q_1 \cup Q_5 \cup Q_7}\chi_{\id{p}}\left(\frac{1}{q}\right),
\end{equation}
where the product runs over primes $\id{p} \neq \id{q}$. On the other hand, the
right hand side of~(\ref{eq:charwelldef}) equals
\begin{equation}
  \label{eq:epsilon1}
\varepsilon(\norm(b_q))= \varepsilon_q(d) = \varepsilon_q(d/q) \varepsilon_2(q)^{-1}\prod_{p \in Q_3 \cup Q_5}\varepsilon_p(q)^{-1},
\end{equation}
where the product runs over primes different from $q$. Recall that for
each ramified prime, under the isomorphism
$(\Om_K/\id{p})^\times \simeq (\Z/p)^\times$, we have the equality
$\chi_{\id{p}}=\varepsilon_p\delta_p$. In particular, both sides
evaluate the same at elements of $\Z_p^\times$. Using such a relation
in (\ref{eq:nada}) for all odd ramified primes and gathering together
the terms involving $\varepsilon$ gives
\begin{multline}
  \label{eq:compa}  
  \chi(b_q^2) = \chi_2^{-1}(q) \varepsilon_q\left(\frac{-d}{q}\right) \prod_{\substack{p \in Q_1 \cup Q_3 \cup Q_5\cup Q_7 \\ p \neq q }}\varepsilon_p^{-1}(q) \cdot \delta_q\left(\frac{-d}{q}\right) \prod_{\substack{p \in Q_1 \cup Q_3 \cup Q_5\cup Q_7 \\ p \neq q }} \delta_p(q) =  \\
=  \varepsilon_q(d) \left(\chi_2^{-1}(q)\varepsilon_q(-1)\varepsilon_2(q)\delta_q(2)^{v_2(d)}\right)\cdot \left(\delta_q(2)^{v_2(d)} \delta_q\left(\frac{-d}{q}\right)\prod_{\substack{p \in Q_1 \cup Q_3 \cup Q_5\cup Q_7 \\ p \neq q }}\delta_p(q)\right).  
\end{multline}
Our goal is to prove that the product of all the factors of the last
expression except the first one is $1$ for the result to hold.
When $q \equiv 1 \pmod 4$, quadratic reciprocity implies that
$\delta_q(p)= \delta_p(q)$ and $\delta_q(-1) = 1$, so the last factor
(between round brackets) in (\ref{eq:compa}) equals $1$. On the other hand,
for $q \equiv 3 \pmod 4$ quadratic reciprocity implies that
$\delta_q(p) = \delta_p(q) \delta_p(-1)$. Note that $q$ is one of the
elements in $Q_3 \cup Q_7$, so the last factor in (\ref{eq:compa})
equals $(-1)^{\#Q_3+\#Q_7}$. In both cases, the last factor equals
$\delta_{-1}(q)^{\#Q_3+\#Q_7}$.

Regarding the second factor, quadratic reciprocity again implies that
$\delta_q(2)=\delta_{2}(q)$ (recall the definition of $\delta_2$ from
Table~\ref{table:characters}). By definition $\varepsilon_2=\delta_{-1}^{\#Q_3 + \#Q_5}$ and on elements of $\Z_2^\times$, $\chi_2 = \delta_2^{v_2(d)+1}\delta_{-1}^{\#Q_5 + \#Q_7}$ (see (\ref{chi2})) then the right hand side of (\ref{eq:compa}) equals
\[
\varepsilon_q(d)\delta_2(q)\varepsilon_q(-1)\delta_{-1}(q)^{2\#Q_3 +
    2\#Q_5 + 2\#Q_7} = \varepsilon_q(d)\delta_2(q)\varepsilon_q(-1).
\]
Then we are led to prove that $\varepsilon_q(-1)\delta_2(q) = 1$,
which follows from the definitions, since:
\begin{itemize}
\item If $q \equiv \pm 1 \pmod 8$,
  $\varepsilon_q(-1) = 1 = \delta_2(q)$.

\item If $q \equiv \pm 3 \pmod 8$,
  $\varepsilon_q(-1) = -1 = \delta_2(q)$.
\end{itemize}

Suppose now that $d \equiv 1 \pmod 4$, when we also need to check the
compatibility for $b_2$. A similar computation as the previous one gives that:
\[
  \varepsilon_2(1+d) = \varepsilon_2\left(\frac{1+d}{2}\right) \cdot
  \prod_{p \in Q_3\cup Q_5}\varepsilon_p(2)^{-1},
\]
while
\[
  \chi(b_i^2) =\chi_2\left(\frac{(1+\sqrt{-d})^2}{2}\right)\cdot
  \prod_{p \in Q_3 \cup Q_5}\varepsilon_p(2)^{-1} \cdot \prod_{p \in
    Q_1\cup Q_3\cup Q_5\cup Q_7}\delta_p(2).
\]
By quadratic reciprocity,
$\prod_{p \in Q_1\cup Q_3\cup Q_5\cup Q_7}\delta_p(2) = (-1)^{\#Q_3 +
  \#Q_5}$, so the statement follows from the following easy to verify
(from its definitions) facts:

\begin{itemize}
\item If $d\equiv 1 \pmod 8$, then
  $\varepsilon_2\left(\frac{1+d}{2}\right) =1$,
  $\chi_2\left(\frac{(1+\sqrt{-d})^2}{2}\right)=1$ and $\#Q_3 + \#Q_5$
  is even.
\item If $d\equiv 5 \pmod{16}$, then
  $\varepsilon_2\left(\frac{1+d}{2}\right) =-1$,
  $\chi_2\left(\frac{(1+\sqrt{-d})^2}{2}\right)=1$ and $\#Q_3 + \#Q_5$
  is odd.
\item If $d\equiv 13 \pmod{16}$, then
  $\varepsilon_2\left(\frac{1+d}{2}\right) =1$,
  $\chi_2\left(\frac{(1+\sqrt{-d})^2}{2}\right)=-1$ and
  $\#Q_3 + \#Q_5$ is odd.
\end{itemize}

\vspace{2pt}

Now that we defined the character $\chi$ on the whole id\`ele group
and proved that it is well defined, we only need to verify
that our extension also satisfies
\[
  ^\tau \chi = \chi \cdot (\psi_{-2} \circ \norm)
\]
for all elements of $\II_K$. Since we already proved this is the case
for elements of $K^\times\cdot(\prod_{\id{q}}\Om_{\id{q}}^\times \times
\CC^\times)$, it is enough to prove it for the id\`eles
$a_i$ (as defined before) with local finite components
\[
  (a_i)_\id{p}=
  \begin{cases}
    r_i & \text{ if }\id{p} = \id{r}_i,\\
    1 & \text{ otherwise}.
  \end{cases}
\]
Note that $\tau(a_i)$ is the id\`ele of $K$ with value
$r_i$ at $\overline{\id{r}_i}$ and $1$ at the other places. Then
\begin{equation}
  \label{eq:tau1}
  ^\tau\chi(a_i) = \chi(\tau(a_i)) = \chi(a_i)^{-1} \chi(a_i \tau(a_i)) = \chi(a_i)^{-1} \chi\left(\frac{a_i \tau(a_i)}{r_i}\right),    
\end{equation}
where $\frac{1}{r_i}$ denotes the image by
$K^\times \hookrightarrow \II_K$.  Note that
$\frac{a_i\tau(a_i)}{r_i}$ is a unit at all places, so
\begin{equation}
  \label{eq:tau2}
  \chi\left(\frac{a_i \tau(a_i)}{r_i}\right) =
  \chi_2(r_i)^{-1}\prod_{\id{p} \in Q_1 \cup Q_5 \cup Q_7} \chi_{\id{p}}(r_i)^{-1}.
\end{equation}
By the product formula,
\begin{equation}
  \label{eq:varepsilon}
  1= \varepsilon(r_i) = \varepsilon_{r_i}(r_i) \varepsilon_2(r_i)\prod_{p \in Q_3 \cup Q_5}\varepsilon_p(r_i).
\end{equation}
Since
$\varepsilon_{r_i}(r_i) = \varepsilon(\norm(a_i)) = \chi^2(a_i)$,
multiplying (\ref{eq:tau2}) and (\ref{eq:varepsilon}) and using that $\chi_{\id{p}}(r_i) = \varepsilon_p(r_i) \delta_p(r_i)$ we get that
\begin{equation}
  \label{eq:11}
  \chi\left(\frac{a_i \tau(a_i)}{r_i}\right) =\chi^2(a_i)\chi_2(r_i)^ {-1}\varepsilon_2(r_i) \prod_{p \in Q_1 \cup Q_3\cup Q_5 \cup Q_7} \delta_p(r_i).
\end{equation}
Recall that $r_i$ splits in $K$, hence $\kro{-d}{r_i} = 1$ and quadratic reciprocity implies that 
\[
  1 = \kro{2}{r_i}^ {v_2(d)}\kro{-1}{r_i}^ {\#Q_3 + \#Q_7+1}\prod_{p
    \in Q_1 \cup Q_3\cup Q_5 \cup Q_7} \delta_p(r_i).
\]
In particular, the right hand side of (\ref{eq:tau1}) equals
\[
\chi(a_i)\chi_2(r_i)^{-1}\varepsilon_2(r_i)\delta_2(r_i)^ {v_2(d)}\delta_{-1}(r_i)^ {\#Q_3 + \#Q_7+1} = \chi(a_i)\cdot\left(\delta_2(r_i)\delta_{-1}(r_i)^{\#Q_5+\#Q_7}\varepsilon_2(r_i)\delta_{-1}(r_i)^ {\#Q_3 + \#Q_7+1}\right) .
  \]
A similar (but easier) computation shows that
$\psi_{-2}(\norm(a_i)) = \delta_{-2}(r_i)$, so the result follows from the equality
\[
  \delta_2(r_i)\varepsilon_2(r_i)\delta_{-1}(r_i)^ {\#Q_3 + \#Q_5+1} = \delta_2(r_i)\delta_{-1}(r_i)=\delta_{-2}(r_i).
\]
\end{proof}
\begin{remark}
	The precise conductor ${\mathfrak f}$ of $\chi_{\id{p}_2}$ has valuation:
	\[
	v({\mathfrak f}) =
	\begin{cases}
		5 & \text{ if } d\equiv 1 \pmod 8,\\
		3 & \text{ if } d \equiv 3, 5, 6 \pmod 8,\\
		3 & \text{ if } d \equiv 2\pmod 8 \text{ and }2 \nmid \#Q_3 + \#Q_5,\\
		0 & \text{ if } d \equiv 2 \pmod 8 \text{ and }2 \mid \#Q_3 + \#Q_5.\\
	\end{cases}
	\]
	When $d \equiv 7 \pmod 8$ it is either $0, 2, 3$ depending on its choice.
	\label{rem:conductorchi2}
\end{remark}
Although we constructed one Hecke character satisfying that
\[
^\tau \chi = \chi \cdot \psi_{-2},
  \]
it is a natural problem to understand all such possible finite order Hecke
characters. Note
that if $\chi$ is such a character, and $\nu$ is a character of
$\Gal_\Q$, then $\chi \cdot \nu$ also satisfies the same condition. 
%
%

\begin{thm}
\label{thm:unicity}
Let $\chi_1$ and $\chi_2$ be finite order Hecke characters such that
$^\tau\chi_i = \chi_i \cdot \psi_{-2}$ for $i=1,2$. Then there exists
a rational character $\nu:\Gal_\Q \to \CC^\times$ such that
$\chi_2 = \chi_1 \cdot \nu$.
\end{thm}
\begin{proof}
  Let $\tilde{\nu}$ denote the quotient $\frac{\chi_1}{\chi_2}$, a
  character on $\II_K$. The hypothesis
  $^\tau\chi_i = \chi_i \cdot \psi_{-2}$ implies that
  $^\tau\tilde{\nu} = \tilde{\nu}$. Let $\II_K^1$ be the kernel of the
  norm map $\norm:\II_K \to \II_\Q$.

  \vspace{2pt}

  \noindent {\bf Claim:} the character $\tilde{\nu}$ is trivial on $\II_K^1$.
  \vspace{2pt}

  To prove the claim, let $v$ be a place of $\QQ$ which does not split
  in $K$, and let $w$ be the place of $K$ dividing it. In particular
  $K_w/\Q_v$ is a Galois quadratic extension. If $k \in K_w$ has norm
  one, Hilbert's theorem 90 implies that there exists $t \in K_w$ such
  that $k = \frac{t}{\sigma(t)}$ for $\sigma \in \Gal(K_w/\Q_v)$
  non-trivial. The hypothesis $^\tau\tilde{\nu} = \tilde{\nu}$ then
  implies that $\tilde{\nu}(k)=1$. If the place $v$ happens to split,
  let $w_1,w_2$ be the two places of $K$ above it and let
  $(k_1,k_2) \in K_{w_1}\times K_{w_2}$ be an element of norm one,
  i.e. $k_1\cdot k_2 = 1$. The hypothesis
  $^\tau\tilde{\nu}=\tilde{\nu}$ implies that
  $\tilde{\nu}_{w_1}=\tilde{\nu}_{w_2}$, so
  $\tilde{\nu}_{w_1}(k_1)\tilde{\nu}_{w_2}(k_2) =
  \tilde{\nu}_{w_1}(k_1\cdot k_2) = 1$ as claimed.

  Then the character $\tilde{\nu}$ gives a well defined character on
  $\norm(\II_K)$, a subgroup of $\II_\Q$ and we can extend it to
  $\Q^\times \norm(\II_K)$ by making it trivial at elements of
  $\Q^\times$. Recall that $\Q^\times \norm(\II_K)$ has finite index
  in $\II_\Q$ so let $\nu$ be any extension of $\tilde{\nu}$ to the
  whole id\`ele group $\II_\Q$. Then by construction $\tilde{\nu}$
  coincides with $\nu \circ \norm$ on $\II_K$, so in particular
  $\chi_1 = \chi_2 \cdot (\nu\circ\norm)$.
\end{proof}
%
%
It is also natural to study whether our construction can be extended
to negative values of $d$, i.e. to real quadratic fields. The problem
now is that we need some control on the fundamental unit. In this
case, we have a partial answer.

\begin{thm}
  Suppose that $K = \QQ(\sqrt{d})$ is a real quadratic field, whose
  fundamental unit has norm $-1$. Then the same statement of Theorem~\ref{thm:charexistence} holds.
\end{thm}
\begin{proof}
  To use the previous results, write $d = -(-d)$ (so $d<0$ in the above
  notations/definitions) and make precisely the same local definitions
  for both $\varepsilon$ and $\chi$ at finite places.	
  There are two important facts to consider: while proving
  (\ref{eq:compat1}), we get an extra $-1$ factor coming from the fact that
  we changed $d \leftrightarrow -d$. This forces us to add ramification
  at one of the archimedean places (we will latter specify which one).
	
  Let $\epsilon$ be a fundamental unit (fixed). The proof works
  mutatis mutandis once we checked the compatibility of $\chi$ at
  $\epsilon$. Our assuming $\epsilon$ of norm $-1$, implies that
  $Q_3 = Q_7 = \emptyset$, the reason being that if
  $\epsilon = a+b\sqrt{d}$, with $2a, 2b \in \Z$, the condition
  $a^2-db^2 = -1$ implies that $-1$ is a square for all odd primes
  dividing $d$. Furthermore, for all such primes, the reduction of
  $\epsilon$ has order $4$, so that $\chi_p(\epsilon) = \pm 1$ if
  $p \in Q_1$ and a primitive fourth root of unity if $p \in Q_5$. We
  claim that
  $\chi_2(\epsilon) \prod_{p \in Q_5}\chi_p(\epsilon) = \pm 1$ (and
  therefore
  $\chi_2(\epsilon) \prod_{p \in Q_1 \cup Q_5}\chi_p(\epsilon) = \pm
  1$).
  \begin{itemize}
  \item If $d \equiv 1 \pmod 8$ then $\#Q_5$ is even and $\chi_2$ is
    quadratic, hence the statement.
		
  \item If $d \equiv 5 \pmod 8$ (the case $d=3$ in
    Table~\ref{table:quotientstructure}) $\#Q_5$ is odd, $2$ is inert,
    $\chi_2$ has order $4$ and evaluated at any element of order $4$
    gives a primitive fourth root of
    unity.
		
  \item If $d \equiv 2 \pmod 8$ and $\#Q_5$ is even, $\chi_2$ has
    order $2$, while if $\#Q_5$ is odd, $\chi_2$ has order $4$ and
    $\epsilon$ has order $8$ (which follows from
    Table~\ref{table:quotientstructure}, as its norm equals $-1$) so
    $\chi_2(\epsilon)$ is a fourth root of unity.
  \end{itemize}
  Then if the product
  $\chi_2(\epsilon) \prod_{p \in Q_1 \cup Q_5}\chi_p(\epsilon) = 1$,
  define $\chi$ to be trivial at the archimedean component where
  $\epsilon$ is negative and the sign character at the other, while if
  the product equals $-1$, take the opposite choice. Since
  $\norm(\epsilon)=-1$, the compatibility is satisfied and the same
  proof of Theorem~\ref{thm:charexistence} applies.
\end{proof}

For general real quadratic fields we run some numerical experiments (a
couple of hundreds) and in all cases, a character of the expected
conductor is found. The result will be explained (and proved) in a
sequel (see \cite{2103.06965}).

\subsection{The case $t \equiv 3 \pmod 4$}
\label{section:t=3}

Let $K=\Q(\sqrt{-d})$ with $d>0$ (square-free). The method is very
similar to the previous case. Define the following sets:
\begin{itemize}
\item
  $Q_{++}=\{p \mid d, \; p \nmid 2t, \; p \equiv \square \pmod 4, \;
  p\equiv \square \pmod t\}$.
\item
  $Q_{+-}=\{p \mid d, \; p \nmid 2t, \; p \equiv \square \pmod 4, \;
  p\not \equiv \square \pmod t\}$.
\item
  $Q_{-+}=\{p \mid d, \; p \nmid 2t, \; p \not \equiv \square \pmod 4,
  \; p \equiv \square \pmod t\}$.
\item
  $Q_{--}=\{p \mid d, \; p \nmid 2t, \; p \not \equiv \square \pmod 4,
  \; p\not \equiv \square \pmod t\}$.
\end{itemize}
We have the following elementary result (that will clarify later computations).
\begin{lemma}
  \label{lemma:tbehaviour}
  Suppose that $t$ is unramified in $K$. Then the prime $t$ splits in
  $K$ precisely when the following equality holds:
  \[
    (-1)^{\#Q_{+-}+\#Q_{--}}\delta_t(2)^{v_2(d)}=-1.
  \]
  Similarly, it is
  inert when $(-1)^{\#Q_{+-}+\#Q_{--}}\delta_t(2)^{v_2(d)}=1$.
\end{lemma}
\begin{proof}
  Follows easily from the well known fact that $t$ splits in $\Q(\sqrt{-d})$ if and only if $-d$ is a square modulo $t$.
\end{proof}

\noindent{\bf The character $\varepsilon$:} Define an even character
$\varepsilon:\II_\Q \to \CC^\times$ ramified at the primes in
$Q_{+ +}$, $Q_{-+}$, $Q_{+-}$ and eventually at $2$ and $t$. Its local components $\varepsilon_p$
are defined as follows:
\begin{itemize}
\item For primes $p \in Q_{++}\cup Q_{-+}$, the character
  $\varepsilon_p:\ZZ_p^\times \to \overline{\Q}^\times$ is quadratic,
  i.e, $\varepsilon_p=\delta_{p}$.
	
\item For primes $p \in Q_{+-}$, the character
  $\varepsilon_p:\ZZ_p^\times \to \overline{\Q}^\times$ is any
  character of order $2^{v_2(p-1)}$.

\item For $p=t$ define $\varepsilon_t=\begin{cases}
    \delta_t^{\#Q_{+-}+\#Q_{--}+v_t(d)+v_2(d)+1} & \text{ if }t\equiv 3\pmod 8,\\
    \delta_t^{\#Q_{+-}+\#Q_{--}+v_t(d)+1} & \text{ if }t\equiv 7\pmod 8.
		\end{cases}$

\item For $p=2$ define $\varepsilon_2=\begin{cases}
    \delta_{-1}^{\#Q_{-+}+\#Q_{--}+v_t(d)+v_2(d)+1} & \text{ if }t\equiv 3\pmod 8,\\
    \delta_{-1}^{\#Q_{-+}+\#Q_{--}+v_t(d)+1} & \text{ if }t\equiv 7\pmod 8.
		\end{cases}$
\item At all other primes, $\varepsilon_p$ is trivial.
\item The character $\varepsilon_\infty$ (the archimidean component) is trivial.
\end{itemize}
By Lemma~\ref{lemma:tbehaviour}, $\varepsilon_t$ is trivial
if $t$ splits in $K$ and equals $\delta_t$ if $t$ is inert in $K$. A similar result for $\varepsilon_2$ is the following.
\begin{lemma}
  \label{lemma:char2epsilon2}
  The previous defined character $\varepsilon_2$ satisfies that
  \[
    \varepsilon_2 = \begin{cases}
      1 & \text{ if } d \equiv 3 \pmod 4,\\
      \delta_{-1} & \text{ if } d \equiv 1 \pmod 4.
    \end{cases}
  \]
\end{lemma}
\begin{proof} Since $2 \nmid d$, the prime $2$ is ramified in $K$ if
  and only if $-d \equiv 3 \pmod 4$. Recall that the prime $t$ is not
  an element of $Q_{\pm \pm}$, so the prime divisors of $d$ congruent
  to $3$ modulo $4$ are precisely the ones in $Q_{-+} \cup Q_{--}$ and
  possibly $t$. Then the prime $2$ is ramified in $K$ precisely when
  $\#Q_{-+} + \#Q_{--} + v_t(d)$ is even. Since $2 \nmid d$,
  $\varepsilon_2$ equals $\delta_{-1}$ when $2$ is ramified and $1$
  when it is unramified, as claimed.
\end{proof}
With the above definitions, and using that
$\varepsilon_2(-1)\varepsilon_t(-1) = (-1)^{\#Q_{+-}+\#Q_{-+}}$, it
is not hard to verify that
\[
\prod_p\varepsilon_p(-1) \varepsilon_{\infty}(-1) = (-1)^{\#Q_{-+}+\#Q_{+-}}\varepsilon_2(-1)\varepsilon_{t}(-1)=1.
\]
The main result of the present section is the following.

\begin{thm}
  There exists a Hecke character $\chi:\Gal_K \to \overline{\QQ}$ such
  that:
  \begin{enumerate}
	\item $\chi^2(\sigma) = \varepsilon(\sigma)$ for all $\sigma \in \Gal_K$,
		
  \item $\chi$ is unramified at primes not dividing
    $2t\prod_{p \in Q_{+-} \cup Q_{--}}p$,
		
	\item If $\tau \in \Gal_\Q$ is not the identity on $K$, 
$^\tau\chi = \chi \cdot \psi_{-t}$ as characters of $\Gal_K$.
  \end{enumerate}
  \label{thm:charexistence2}
\end{thm}

\begin{proof}
	
  Following the strategy described at the beginning of the section,
  define $\chi_p:\Om_{\id{p}}^\times \to \CC^\times$ by:
		
  \begin{itemize}
  \item If $\id{p}$ is an odd (i.e. $\id{p}\nmid 2$) unramified prime,
    $\chi_{\id{p}}$ is the trivial character.
  \item If $p\in Q_{\pm \pm}$ and $\id{p}\mid p$,
 \begin{equation}
	\label{eq:charcomp2}
	\chi_{\id{p}} = \varepsilon_p \delta_p.    
\end{equation}

  \item At primes $\id{t}$ dividing $t$, define the character $\chi_\id{t}$ by
    \begin{itemize}
    \item If $t$ ramifies in $K$, $\chi_\id{t} = \varepsilon_{t}$.
			
    \item If $t$ splits in $K$, say $t = \id{t} \overline{\id{t}}$,
      let $\chi_{\id{t}} = \delta_t$ and
      $\chi_{\overline{\id{t}}}=1$.
			
    \item If $t$ is inert in $K$, $\chi_\id{t}$ is an order $4$ character (hence its restriction to $\F_t^\times$ is trivial).
    
    \end{itemize}
 We will denote in all cases $\chi_t=\prod_{\id{t}\mid t}\chi_{\id{t}}$.
  \item At a prime $\id{p}_2$ dividing $2$, define the character
    $\chi_{\id{p}_2}$ as follows:
    \begin{itemize}
    \item If $2$ is not ramified in $K/\Q$, it is trivial.
			
    \item If $2$ ramifies in $K/\Q$ but $2 \nmid d$, then define
      $\chi_{\id{p}_2}$ as:
      \begin{itemize}
      \item if $t\equiv3\pmod8$, the character of conductor $2$
        sending $\sqrt{-d}$ to $-1$.
					
      \item if $t\equiv7\pmod8$, the trivial character.
      \end{itemize}
    \item If $2\mid d$ then $\chi_{\id{p}_2}$ is the character
      of conductor $\id{p}_2^5$ whose value at the generators
      $5$, $-1$ and $1+\sqrt{-d}$ equals: $\chi_{\id{p}_2}(5)=-1$,
      $\chi_{\id{p}_2}(-1)=\delta_{2}(t)$
      and
      \begin{itemize}
      \item If $d\equiv2\pmod8$,
        $\chi_{\id{p}_2}(1+\sqrt{-d})=\begin{cases}
          1 & \text{ if } t\equiv 3\pmod8,\\
          \sqrt{-1} & \text{ if } t\equiv 7\pmod8.
	\end{cases}$
				
      \item If $d\equiv6\pmod8$,
        $\chi_{\id{p}_2}(1+\sqrt{-d})=\begin{cases}
          \sqrt{-1} & \text{if } t\equiv 3\pmod8,\\
          -1 & \text{if } t\equiv 7\pmod8.
        \end{cases}$
					
      \end{itemize}
\end{itemize}
  Abusing notation, we will write $\chi_2=\chi_{\id{p}_2}$.
\item The archimidean component of $\chi$ is trivial.
\end{itemize}
Let us recall some properties of the local characters just defined (which
motivate their definition) that will play a crucial role.
\begin{enumerate}
\item[(P1)] The product $\varepsilon_t\chi_t$ on elements of $\Z_t^\times$ equals 
\[
\varepsilon_t \chi_t = \begin{cases}
\delta_t & \text{ if } t \nmid d,\\
1 & \text{ if } t \mid d.
\end{cases}
\]
\item[(P2)] If $d \equiv 1 \pmod 4$ then $\chi_{2}(\sqrt{-d})=\delta_t(2)$.
\item[(P3)] If $2\mid d$, then $\chi_{2}^2(1+\sqrt{-d})=\chi_{2}(1+d)$.
\item[(P4)] If $2 \mid d$ then
  $\chi_{2}|_{\ZZ_2^\times}=\delta_{-2}$ if $t\equiv3\pmod8$
  and $\chi_{2}|_{\ZZ_2^\times}=\delta_{2}$ if
  $t\equiv7\pmod8$.
  
\item[(P5)] In all cases $\chi_{2}(-1)=\delta_2(t)^{v_2(d)}$ and $\chi_2(t)=1$.
\end{enumerate}

\vspace{1pt}

As in the previous case, define $\chi$ on
$K^\times \cdot (\prod_{\id{q}}\Om_{\id{q}}^\times \times \CC^\times)$
to be trivial at elements of $K^\times$ and as the product of the
local component at elements of the second factor. We claim that $\chi$ satisfies the expected three properties.
  \begin{enumerate}
\item We need to verify that the equality
\[
\chi^2 = \varepsilon \circ \norm
  \]
  holds for all elements of
  $K^\times \cdot (\prod_{\id{q}}\Om_{\id{q}}^\times \times
  \CC^\times)$. As before, it is enough to prove it componentwise, and
  the result for all primes not dividing $2t$ follows from
  (\ref{eq:charcomp2}) together with Lemma~\ref{lemma:charnorm}.

  For primes dividing $t$, the case $t$ split and $t$ ramified follows
  from the fact that both $\chi_t^2$ and $\varepsilon_t\circ \norm$ are
  trivial. In the inert case, it is enough to check the condition at a
  generator $g$ of $\F_{t^2}^\times$: $\chi_t^2(g)=-1$ (as $\chi_t$
  has order $4$), and $\varepsilon_t(\norm(g))=-1$ because $\norm(g)$
  generates $\F_t^\times$.
    
  If $2 \nmid d$, the
  statement is also clear as $\chi_{2}^2$ is trivial and
  $\varepsilon_2$ is trivial if $2$ is unramified (by
  Lemma~\ref{lemma:char2epsilon2}) and $\delta_{-1}$ when $2$ is
  ramified. But the norm map in the ramified case only takes the
  values $\{0,1,2\}$ modulo $4$, hence $\delta_{-1} \circ \norm$ is
  trivial as well. At last, if $2 \mid d$, both $\chi_2^2$ and
  $\varepsilon_2 \circ \norm$ are trivial at elements of $\Z_2^\times$
  and agree at $1+\sqrt{-d}$ by (P3).
\item The ramification statement is clear from the definition.
\item For odd primes $\id{p}$ not dividing $t$ the local character
  $(\psi_{-t})_p\circ \norm$ is trivial while
  $^\tau\chi_{\id{p}}=\chi_{\id{p}}$, hence the statement. For primes dividing $2$, since $(\psi_{-t})_2\circ\norm$ is
  also trivial, we need to verify that
  $^\tau\chi_{2} = \chi_{2}$ (remember the notation $\chi_{2}=\chi_{\id{p}_2}$, where $\id{p}_2$ is any prime dividing $2$). This follows easily from its
  definition when $2 \nmid d$. When $2 \mid d$, we only need to verify
  the property at the element $1+\sqrt{-d}$, but
  $^\tau\chi_{2}(1+\sqrt{-d}) =\chi_{2}(1-\sqrt{-d})=
  \chi_{2}(1+\sqrt{-d})^{-1}\chi_{2}(1+d) =
  \chi_{2}(1+\sqrt{d})$ by (P3).

  \noindent Let $\id{t}$ be a prime dividing $t$. If $t$ ramifies in
  $K$, $(\psi_{-t})_t\circ \norm$ is trivial while
  $^\tau\chi_{\id{t}} = \chi_{\id{t}}$, hence the statement.  If $t$
  splits, without loss of generality we can assume that
  $\chi_{\id{t}}$ matches $(\psi_{-t})_t\circ \norm$ and
  $\chi_{\overline{\id{t}}}$ is trivial, so the result holds. At last,
  suppose that $t$ is inert in $K$. Let $g$ be a generator of
  $\F_{t^2}^ \times$; then
  $^\tau\chi_t(g)\chi_t(g) = \chi_t(\norm(g)) =1$ (since $\chi_t$ is
  trivial on elements of $\F_t^\times$) so
  $^\tau\chi_t = \chi_t^{-1}$. On the other hand,
  $\delta_t(\norm(g))=-1$ because $\norm(g)$ is a generator of
  $\F_t^\times$, hence (since $\chi_t(g)$ is a fourth root of unity)
  $^\tau\chi_t(g) = \chi_t(g)^{-1} = -\chi_t(g)=\chi_t(g) \cdot
  \psi_{-t}(\norm(g))$ as claimed.
\end{enumerate}
	
\smallskip
	
\noindent{\bf Compatibility:} since all characters have order a power
of $2$, the compatibility relation at roots of unity of order $3$ (if
$K$ happens to have one) is trivial. The only case when $K$ contains
roots of unity of order $4$ is for $K=\Q(\sqrt{-1})$. In such a case,
all sets $Q_{\pm,\pm}$ are empty. By definition:
\begin{enumerate}
\item $\chi_{2}(\sqrt{-1}) = -1$ if $t \equiv 3 \pmod 8$ and $1$ if $t \equiv 7 \pmod 8$.
  
\item $\chi_t$ is an order $4$ character (since $t$ is inert in $\Q(\sqrt{-1})$) whose restriction to $\F_t^\times$ is trivial.
\end{enumerate}
Since $t \equiv 3 \pmod 4$, $-1$ is not a square in $\F_t$, hence
$\sqrt{-1}$ is an element of $\F_{t^2}^\times$ so
$\chi_t(\sqrt{-1}) \in \{\pm 1\}$. Let $g$ be a generator of
$\F_{t^2}$, and let $\sqrt{-1} = g^e$. Recall that the minimum power
of $g$ in $\F_t^\times$ is $t+1$, hence if $t \equiv 3 \pmod 8$, $e$
has valuation one at $2$ so $\chi_t(\sqrt{-1})=-1$, while if
$t \equiv 7 \pmod 8$, $4 \mid e$ and $\chi_t(\sqrt{-1})=1$ as needed.

\vspace{1pt}

For all other fields (abusing a little notation) we have
	
\[
  \chi(-1)=\prod_{\id{p}\in Q_{+-}\cup
    Q_{--}}\chi_{\id{p}}(-1)\chi_{2}(-1)\chi_{t}(-1)=(-1)^{\#Q_{+-}+\#Q_{--}}\chi_{2}(-1)\chi_{t}(-1).
\]
Recall from property (P5) that
$\chi_{2}(-1)=\delta_t(2)^{v_2(d)}$.  Consider the different
cases:
\begin{itemize}
\item If $t$ is unramified in $K$, quadratic reciprocity
  (Lemma~\ref{lemma:tbehaviour}) implies that $t$ splits in $K$
  (respectively is inert in $K$) if and only if
  $(-1)^{\#Q_{+-}+\#Q_{--}}\delta_t(2)^{v_2(d)}=-1$ (respectively
  $1$). In the first case $\chi_t(-1)=-1$ while in the second it
  equals $1$. In both cases,
  $(-1)^{\#Q_{+-}+\#Q_{--}}\chi_{2}(-1)\chi_t(-1)=1$ as
  expected.
\item If $t$ ramifies in $K$, by definition
  $\chi_t = \varepsilon_t$, its value at $-1$ equals
  \[
\varepsilon_t(-1)=(-1)^{\#Q_{+-}+\#Q_{--}}\delta_t(2)^{v_2(d)}.
  \]
\end{itemize}
	
\medskip
	
\noindent{\bf Extension:} we extend our character $\chi$ to id\`eles
whose image generates the class group $\Cl(K)$ exactly as in the
previous case, namely via (\ref{eq:charextodd}) for ideals of odd
order in the class group and via (\ref{eq:charext}) for those whose
order is a power of $2$. Then we are led to prove that if for each odd
prime number $q$ dividing $d$ (and also for $q=2$ if
$d \equiv 1 \pmod 4$), $b_q$ denotes the id\`ele
\[
  (b_q)_{\id{p}} =
  \begin{cases}
    1 & \text{ if }\id{p} \neq \id{q},\\
    \sqrt{-d} & \text{ if }\id{p} =\id{q}.
  \end{cases}
\]
then
\begin{equation}
  \label{eq:chraweldeff2}
\chi(b_q^2) = \varepsilon(\norm(b_i)).  
\end{equation}
Once again, we compute both sides of the equation to verify that they
do match. Start supposing that $q \neq t$, then the left hand side
equals
\begin{equation}
  \label{eq:nada2}
  \chi(b_q^2) =\chi_{\id{q}}\left(\frac{-d}{q}\right)\chi_2\left(\frac{1}{q}\right)\chi_t\left(\frac{1}{q}\right)\prod_{\id{p} \in Q_{\pm \pm}}\chi_{\id{p}}\left(\frac{1}{q}\right),
\end{equation}
where the product is over primes which do not divide $q$.  On the
other hand, the right hand side equals
\begin{equation}
  \label{eq:epsilon2}
\varepsilon(\norm(b_q))= \varepsilon(d) = \varepsilon_q(d/q) \varepsilon_2(q)^{-1}\varepsilon_t(q)\prod_{p \in Q_{\pm \pm}}\varepsilon_p(q)^{-1},
\end{equation}
where the product runs over primes different from $q$. Recall that for
all ramified primes different from $t$,
$\chi_{\id{p}}=\varepsilon_p\delta_p$. In particular, both sides
evaluate the same at elements of $\Z_p^\times$. Using such a relation
in (\ref{eq:nada2}) for all odd ramified primes different from $t$ we
get
\begin{multline}
  \label{eq:compa2}  
  \chi(b_q^2) = \chi_2(q)^{-1} \chi_t(q)^{-1}\varepsilon_q\left(\frac{-d}{q}\right) \prod_{\substack{p \in Q_{\pm \pm} \\ p \neq q }}\varepsilon_p(q)^{-1} \cdot \delta_q\left(\frac{-d}{q}\right) \prod_{\substack{p \in Q_{\pm \pm} \\ p \neq q }} \delta_p(q) =  \\
=  \varepsilon_q(d) \left(\chi_2(q)^{-1}\chi_t(q)^{-1}\varepsilon_q(-1)\varepsilon_2(q)\varepsilon_t(q)^{-1}\delta_q(2)^{v_2(d)}\delta_{q}(t)^{v_t(d)}\right)\cdot \left(\delta_q(2)^{v_2(d)}\delta_{q}(t)^{v_t(d)} \delta_q\left(\frac{-d}{q}\right)\prod_{\substack{p \in Q_{\pm \pm} \\ p \neq q }}\delta_p(q)\right).  
\end{multline}
Our goal is to prove that the product of all the factors except the
first one is $1$ for the result to hold.
By quadratic reciprocity, if $p,q$ are odd primes, then
\[
  \delta_p(q)\delta_q(p) =
  \begin{cases}
    1 & \text{ if }p \in Q_{+ \pm},\\
    \delta_{-1}(q) & \text{ if }p \in Q_{-\pm}.
    \end{cases}
  \]
  Then the last term of~(\ref{eq:compa2}) equals
  $\delta_{-1}(q)^{\#Q_{-+}+\#Q_{--}}$.  Regarding the middle factor,
  we claim that the following equality holds (note that it does not
  involve the factor $\varepsilon_q(-1)$):
\begin{equation}
  \label{eq:larga}
  \left(\chi_2(q)^{-1}\chi_t(q)^{-1}\varepsilon_2(q)\varepsilon_{t}(q)\delta_q(2)^{v_2(d)}\delta_{q}(t)^{v_t(d)}\right)\delta_{-1}(q)^{\#Q_{-+}+\#Q_{--}}=\delta_q(t).
\end{equation}

Since all the above values belong to $\{\pm 1\}$ we can remove the
inverses. By property (P1),
$\chi_t(q)\varepsilon_t(q)=\delta_t(q)^{1+v_t(d)}$, so quadratic
reciprocity implies that
$\chi_t(q)\varepsilon_t(q)\delta_q(t)^{v_t(d)}=\delta_q(t)\delta_{-1}(q)^{1+v_t(d)}$,
and the claim is equivalent to the equality
\[
  \left(\chi_2(q)\varepsilon_2(q)\delta_q(2)^{v_2(d)}\right)\delta_{-1}(q)^{\#Q_{-+}+\#Q_{--}+1+v_t(d)}=1.
\]
The last term equals $\varepsilon_2(q)$ when $t \equiv 7 \pmod 8$ and
it equals $\varepsilon_2(q)\delta_{-1}(q)^{v_2(d)}$ when
$t \equiv 3 \pmod 8$. But recall that $\chi_2$ on elements of
$\Z_2^\times$ is trivial when $2\nmid d$ and when $2 \mid d$, it
equals $\delta_{-2}$ if $t \equiv 3 \pmod 8$ and $\delta_2$ if
$t \equiv 7 \pmod 8$ by property (P4). Then the claim follows from the
observation that $\delta_q(2)=\delta_{2}(q)$.
	
To finish the compatibility proof when $q \neq t$, we need to verify
that $\delta_q(t)\varepsilon_q(-1) = 1$, an equality that follows from
the definitions (that are collected in
Table~\ref{table:compatibility2}).
\begin{table}[h]
  \begin{tabular}{|c|c|r|r||c|c|r|r|}
    \hline
    $q \pmod 4$ & $q \pmod t$  & $\varepsilon_q(-1)$ & $\delta_q(t)$& $q \pmod 4$ & $q \pmod t$  & $\varepsilon_q(-1)$ & $\delta_q(t)$ \\
    \hline \hline
    $1$ &  $\square$ & $1$ & $1$ &   $3$ &  $\square$  & $-1$ & $-1$  \\ 
    \hline		
    $1$ &  $\mathord{\not\mathrel{\square}}$  & $-1$ & $-1$ &   $3$ &  $\mathord{\not\mathrel{\square}}$  & $1$ & $1$  \\ 
    \hline
  \end{tabular}
  \caption{\label{table:compatibility2}}
\end{table}

If $q = t$ (so in particular $t \mid d$) the computation is similar,
replacing $q$ by $t$ in (\ref{eq:nada2}) and in (\ref{eq:epsilon2})
but omitting the factor with subscript $t$. Note that in this case
property (P1) states that $\chi_t \varepsilon_t = 1$ so the analogue
of (\ref{eq:compa2}) becomes
\begin{equation}
  \label{eq:q=t}  
  \chi(b_t^2) = \chi_2(t)\varepsilon_t\left(\frac{-d}{t}\right) \prod_{p \in Q_{\pm \pm}}\varepsilon_p(q)^{-1} \cdot \prod_{p \in Q_{\pm \pm}} \delta_p(q) =  
  \varepsilon_t(d) \left(\chi_2(t)\varepsilon_t(-1)\varepsilon_2(t)\right)\cdot \left(\prod_{p \in Q_{\pm \pm}}\delta_p(t)\right).  
\end{equation}
Quadratic reciprocity implies that $\delta_p(t)= 1$ if
$p \in Q_{++}\cup Q_{--}$ and $-1$ if $p \in Q_{+-}\cup Q_{-+}$. Then
the last factor equals
$(-1)^{\#Q_{+-}+\#Q_{-+}} = \varepsilon_t(-1)\varepsilon_2(t)$ (since
$\delta_{t}(-1)=-1$ and $\delta_{-1}(t)=-1$) and the validity of
(\ref{eq:q=t}) follows from the fact that $\chi_2(t)=1$ (by property
(P5)).

At last, when $d\equiv 1\pmod 4$ we also need to prove the same result for the
id\`ele $b_2$ whose local component equals $1+\sqrt{-d}$ at the place
$\id{p}_2$ and $1$ at all other places. Then
\begin{equation}
  \label{eq:eq2compat2}
 \varepsilon_2(1+d) = \varepsilon_2\left(\frac{1+d}{2}\right) \cdot \prod_{p \in Q_{\pm \pm}}\varepsilon_p(2)^{-1} \cdot \varepsilon_t(2)^{-1},  
\end{equation}
while
\begin{equation}
  \label{eq:eq2compat22}
  \chi(b_i^2) = \chi_{2}\left(\frac{(1+\sqrt{-d})^2}{2}\right) \prod_{p \in Q_{\pm \pm}}\varepsilon_p(2)^{-1} \delta_p(2) \cdot \chi_{t}(2)^{-1}.
\end{equation}
Then we are led to prove that
\[
\varepsilon_2\left(\frac{1+d}{2}\right) \varepsilon_t(2)^{-1}=\chi_{2}\left(\frac{(1+\sqrt{-d})^2}{2}\right) \prod_{p \in Q_{\pm \pm}}\delta_p(2) \cdot \chi_{t}(2)^{-1}.
\]
Recall by Property (P1) that
$\varepsilon_t(2)\chi_t(2) = \delta_t(2)^{1+v_t(d)}$. Also, the equality
$\frac{(1+\sqrt{-d})^2}{2} = \frac{1-d}{2}+\sqrt{-d}$ implies that
$\chi_{2}\left(\frac{(1+\sqrt{-d})^2}{2}\right)=\chi_{2}(\sqrt{-d})=\delta_t(2)$
(by definition), so
\[
\varepsilon_t(2)\chi_t(2)^{-1}\chi_{2}\left(\frac{(1+\sqrt{-d})^2}{2}\right) = \delta_t(2)^{v_t(d)}.
  \]
Regarding the other terms:
\begin{itemize}
\item If $d \equiv 1 \pmod 8$ then $\varepsilon_2\left(\frac{1+d}{2}\right)=1$ and $\prod_{p \in Q_{\pm \pm}}\delta_p(2) \delta_t(2)^{v_t(d)} = 1$, hence the statement.
  
\item If $d \equiv 5 \pmod 8$ then $\varepsilon_2\left(\frac{1+d}{2}\right)=-1$ and $\prod_{p \in Q_{\pm \pm}}\delta_p(2) \delta_t(2)^{v_t(d)} = -1$, hence the statement.
\end{itemize}

\vspace{2pt}

Now that we have a well defined character on the whole ad\`ele group $\II_K$, we need to verify that the condition
\[
  ^\tau \chi = \chi \cdot (\psi_{-t} \circ \norm)
\]
holds for the extension. Once again, it is enough to check the
property at the id\`eles $a_i$ in $\II_K$ with trivial infinite
component and finite components:
\[
  (a_i)_\id{p}=
  \begin{cases}
    r_i & \text{ if }\id{p} = \id{r}_i,\\
    1 & \text{ otherwise}.
  \end{cases}
\]
where the unramified prime ideals $\{\id{r}_i\}$ generate the class
group and where $r_i$ is the norm of $\id{r}_i$.  The same computation of (\ref{eq:tau1}),
(\ref{eq:tau2}) and (\ref{eq:varepsilon}) makes equation~(\ref{eq:11})
becomes
\[
^\tau\chi(a_i)=\chi(a_i)^{-1}  \chi\left(\frac{a_i \tau(a_i)}{r_i}\right) =\chi(a_i)\chi_2(r_i)^ {-1}\chi_t(r_i)^{-1}\varepsilon_2(r_i) \varepsilon_t(r_i)\prod_{p \in Q_{\pm \pm}} \delta_p(r_i).
\]
The fact that $r_i$ splits in $K$ implies that $\kro{-d}{r_i} = 1$ so
quadratic reciprocity gives
\[
  1 = \kro{2}{r_i}^ {v_2(d)}\kro{t}{r_i}^{v_t(d)}\kro{-1}{r_i}^ {\#Q_{-+}+
    \#Q_{--}+1}\prod_{p \in Q_{\pm \pm}} \delta_p(r_i).
\]
Since $\psi_{-t}(\norm(a_i))=\delta_t(r_i)$, we are left to verify that
\[
  \chi_2(r_i)^{-1}\chi_t(r_i)^{-1}\varepsilon_2(r_i)\varepsilon_{t}(r_i)\delta_{r_i}(2)^{v_2(d)}\delta_{r_i}(t)^{v_t(d)}\delta_{-1}(r_i)^{\#Q_{-+}+\#Q_{--}+1}=\delta_t(r_i),
\]
which follows directly from ~(\ref{eq:larga}) .
\end{proof}

Theorem~\ref{thm:unicity} and its proof translates mutatis mutandis to
the case $t \equiv 3 \pmod 4$ studied in this section, so in particular $\chi$ is unique up to a character of $\Gal_\Q$. 

\begin{remark}
  The precise conductor ${\mathfrak f}$ of $\chi_{\id{p}_2}$ has valuation:
  \[
    v({\mathfrak f}) =
    \begin{cases}
      0 & \text{ if } d\equiv 3 \pmod 4,\\
      2 & \text{ if } d \equiv 1 \pmod 4 \text{ and }t \equiv 3 \pmod8,\\
      0 & \text{ if } d \equiv 1 \pmod 4 \text{ and }t \equiv 7 \pmod8,\\
      5 & \text{ if } 2 \mid d.
    \end{cases}
    \]
    \label{rem:conductorchi2Et}
\end{remark}

\section{Extension and lowering the level}
\label{section:extension}

Recall that by Proposition~\ref{prop:qcurveE} (respectively
Proposition~\ref{prop:qcurveE2}) the Galois conjugate of the elliptic
curve $\E$ (respectively $\Et$) is isogenous to its quadratic twist by $\psi_{-2}$
(respectively to its quadratic twist by $\psi_{-3}$). In particular if
$\chi$ denotes the character constructed in
Theorem~\ref{thm:charexistence} (respectively
Theorem~\ref{thm:charexistence2}) then the twisted representation
$\rho_{\E,p}\otimes \chi$ (respectively $\rho_{\Et,p}\otimes \chi$) is
invariant under the action of the Galois group $\Gal(K/\Q)$ so it
should extend to a $2$-dimensional Galois representation of
$\Gal_\Q$. In this section we will prove that this is indeed the case
and furthermore, compute the determinant and conductor of the extension. Let us
state an important result on induced representations.

\begin{thm}
  \label{thm:Henniart}
  Let $E/F$ be a finite extension of local fields, and let $\rho$ be
  an $n$-dimensional representation of $W_E$. Then the conductor of the
  induced representation $\Ind_{W_E}^{W_F}$ equals
  \begin{equation}
    v(\cond(\Ind_{W_E}^{W_F} \rho))=n\delta(E/F) +f(E/F) v(\cond(\rho)),
    \label{eq:Henniart}
  \end{equation}
  where $\delta(E/F)$ denotes the valuation of the different of the extension and $f(E/F)$ the inertial degree.
\end{thm}
\begin{proof}
  See for example \cite{MR0354618}, page 105 (after Proposition 4).
\end{proof}

\subsection{The case $\rho_{\E,p}$}
Let $\varepsilon$ be the real character constructed in
Section~\ref{section:t=2} and let $S(\E)$ be the set of odd primes of
bad reduction of $\E$ (where the curve has multiplicative reduction by
Lemma~\ref{lemma:multred}). Abusing notation, we will say that a
rational prime $q \in S(\E)$ if there exists a prime element of
$S(\E)$ dividing $q$.

\begin{thm}
Suppose that $K/\Q$ is imaginary quadratic. Then the twisted
representation $\rho_{\E,p}\otimes \chi$ extends to a $2$-dimensional
representation of $\Gal_\Q$ attached to a newform of weight $2$,
Nebentypus $\varepsilon$ and level $N$ given by
\[
  N=2^e \cdot \prod_{q \in S(\E)}q \cdot \prod_{q\in Q_3} q \cdot \prod_{q \in
    Q_1 \cup Q_5 \cup Q_7}q^2.
\]
The value of $e$ is one of:
\[
  e = \begin{cases}
    1,8 & \text{ if } 2 \text{ splits},\\
    8 & \text{ if } 2 \text{ is inert},\\
    7,8 & \text{ if } d \equiv 5 \pmod 8,\\
    5, 8 & \text{ if } d\equiv 1 \pmod 8,\\
    8, 9 & \text{ if }2 \mid d.
  \end{cases}
\]
Furthermore, the coefficient field is a quadratic extension of
$\Q(\chi)$.
\label{thm:levelandnebentypus}
\end{thm}

\begin{proof}
  As mentioned before, the existence of the extension was proved by
  Ribet in (\cite{MR2058653}). We give an alternative proof based on
  Galois representations (which is well known to experts) to get
  control on the level and conductor. If the elliptic curve $\E$ has
  complex multiplication, then its $j$-invariant is a real number in
  $K$, hence rational. This implies that the curve is a quadratic
  twist of a rational elliptic curve, hence the existence of the
  extension is automatic. Assume then that $\E$ does not have complex
  multiplication.
	
  To easy notation, let $\rho$ denote $\rho_{\E,p}\otimes \chi$. Its
  conductor divides $\lcm\{N(E_{(a,b,c)}),\cond(\chi)^2\}$ and its
  Nebentypus matches $\varepsilon$ restricted to $\Gal_K$ (by the
  first claim of Theorem~\ref{thm:charexistence}).  Let $\tau$ be as
  in Theorem~\ref{thm:charexistence} (i.e. an element of $\Gal_\Q$
  whose restriction to $\Gal(K/\Q)$ is non-trivial) and suppose
  furthermore that it corresponds to complex conjugation (although
  this is not really necessary). It is enough to define the extension
  of $\rho$ at $\tau$ and check the Nebentypus statement on it.
	
  Recall that $^{\tau}\rho$ denotes the Galois representation defined
  on $\sigma$ by $^{\tau}\rho(\sigma) = \rho(\tau \sigma
  \tau^{-1})$. By the third property of $\chi$, $\rho$ and
  $^{\tau}\rho$ are isomorphic (as they have the same trace at
  Frobenius elements). In particular, there exists
  $A \in \GL_2(\overline{\Q_p})$ such that
  $\rho = A ^{\tau}\rho A^{-1}$. Furthermore, since $\rho$ is
  irreducible (because $\E$ does not have complex multiplication),
  Schur's lemma implies that the matrix $A$ is unique up to a scalar.
  Since $\tau$ has order $2$, the equality
  $\rho(\sigma) = {^{\tau^2}\rho} {(\sigma)} = A^2 \rho(\sigma)
  A^{-2}$ implies that $A^2 = \lambda$ (a scalar matrix).
	
  Suppose that there exists an extension
  $\tilde{\rho}:\Gal_\Q \to \GL_2(\overline{\Q_p})$. Then for
  $\sigma \in \Gal_K$,
  $\tilde{\rho}(\tau \sigma \tau^{-1}) =
  \tilde{\rho}(\tau)\tilde{\rho}(\sigma)\tilde{\rho}(\tau)^{-1} =
  \tilde{\rho}(\tau)\rho(\sigma)\tilde{\rho}(\tau)^{-1}$ and the
  uniqueness of the matrix $A$ implies that there exists a constant
  $\mu$ such that $\tilde{\rho}(\tau) = \mu A$. Since
  $\tilde{\rho}^2(\tau) = \rho(\tau^2) = 1$,
  $\mu^2 = \frac{1}{\lambda}$.

  This suggests that we should define
  $\tilde{\rho}(\tau) = \frac{1}{\sqrt{\lambda}}A$ and it is easy to
  verify that this definition indeed gives an extension (the other
  choice of a square root gives the second possible extension; they
  differ by the twist by the character attached to the quadratic extension
  $K/\Q$). To determine the determinant of the extended representation
  $\tilde{\rho}$, note that $\det(\rho) = \varepsilon \cyclo$
  (where $\cyclo$ denotes the cyclotomic $p$-adic character) as a character
  of $\Gal_K$, hence it is enough to check that they coincide in an
  element of $\Gal_\Q$ which is not in $\Gal_K$ (like $\tau$).  By
  Ribet's result, we know that any extension is odd, i.e.
  $\det(\tilde{\rho})(\tau) = -1$. But also
  $\varepsilon(\tau) \cyclo(\tau) = -1$ (since $\varepsilon$ is
  even) hence $\det(\tilde{\rho})=\varepsilon \cyclo$.

  Modularity of the representation $\tilde{\rho}$ follows from Serre's
  conjectures (\cite{MR2551764} and \cite{MR2827796}), so we know it
  is attached to a modular form of weight $2$ and Nebentypus
  $\varepsilon$; we need to compute the conductor of $\tilde{\rho}$ to
  finish the proof. Let $\id{q}$ be an odd prime unramified in
  $K/\Q$. Then $\chi$ is unramified at $\id{q}$, so the conductor
  valuation of $\rho_{\E,p}\otimes \chi$ equals one for primes in
  $S(\E)$ and zero for the other ones. The validity of the second
  factor in the formula for $N$ then follows from the fact that the
  conductor of a representation does not change while restricting it
  to the absolute Galois group of an unramified extension.

  Let $\id{q}$ be an odd prime ramifying in $K/\Q$ and $q$ be the
  rational prime it divides (which does not equal $p$). The local
  induced representation
  $\Ind_{\Gal_{K_\id{q}}}^{\Gal_{\Q_q}}(\rho_{\E,p}|_{\Gal_{K_\id{q}}}\otimes
  \chi_{\id{q}}) = \tilde{\rho}|_{\Gal_{K_{\id{q}}}} \oplus
  (\tilde{\rho}|_{\Gal_{K_{\id{q}}}}\otimes \mu_{K_{\id{q}}})$, where
  $\mu_{K_{\id{q}}}$ is the quadratic character of the local the extension
  $K_{\id{q}}/\Q_q$. Now we apply Theorem~\ref{thm:Henniart}. By
  construction $\chi_{\id{q}}$ is unramified for $q \in Q_3$ and of
  conductor $\id{q}$ for $q \in Q_1 \cup Q_5 \cup Q_7$, hence the
  right hand side of~(\ref{eq:Henniart}) equals $2$ for primes in
  $Q_3$ and $4$ for primes in $Q_1 \cup Q_5 \cup Q_7$. Note that both
  $\tilde{\rho}$ and $\tilde{\rho} \otimes \mu$ have the same
  Nebentypus, hence they both must ramify at primes in $Q_3$ with
  conductor exponent $1$. Since $\mu$ has conductor exponent $1$ at
  $q$, twisting $\tilde{\rho}$ by $\mu$ cannot vary the conductor
  exponent from $1$ to $3$ (or from $0$ to $4$), hence at primes in
  $Q_1 \cup Q_5 \cup Q_7$ both $\tilde{\rho}$ and
  $\tilde{\rho}\otimes \mu$ have conductor exponent $2$ as claimed.

  The value of $e$ equals the value of the conductor exponent of $\rho_{\E,p}$
  when $2$ is unramified; this gives the inert case
  result. Furthermore, in the split case (say
  $2 = \id{p}_2\overline{\id{p}}_2$) we can chose the local character
  $\chi_{\id{p}_2}$ so that the twist of $\E$ by $\chi_{\id{p}_2}$ has
  split multiplicative reduction of conductor $\id{p}_2$ to get the
  statement.

  If $2$ ramifies in $K/\Q$ we use again
  formula~(\ref{eq:Henniart}). Note that if $d \not \equiv 1 \pmod 8$
  then the value of the conductor of $\rho_{\E,p}$ at $\id{p}_2$ is
  larger than twice the value of the conductor of $\chi$ at $\id{p}_2$
  (see Remark~\ref{rem:conductorchi2}) so the formula follows easily
  from Lemma~\ref{lemma:2valuation} noting that once again the
  conductor at $2$ of $\tilde{\rho}$ matches that of
  $\tilde{\rho}\otimes \mu$. When $d \equiv 1 \pmod 8$, the local
  character $\chi_{\id{p}_2}$ corresponds to a tame extension
  (generated by the square root of $1+\sqrt{-d}$ times a unit), so
  Lemma~\ref{lemma:d=1reduction} implies that the twisted
  representation has conductor valuation $8$ (when
  $b \equiv 1 \pmod 4$) or $6$ (when $b \equiv 3 \pmod 4$) when $b$
  is odd and $12$ when $b$ is even. Since we can assume (changing $b$
  by $-b$) that $b \equiv 3 \pmod 4$ the result follows.
\end{proof}
\begin{remark}
  The coefficient field can be computed as follows: if $p$ is a prime
  inert in $K/\Q$ then
  $\trace(\tilde{\rho}(\Frob_p))^2=a_p(\E)\chi(\Frob_p)+2\varepsilon(\Frob_p)p$, so
  it is enough to perform such computation for one inert prime of
  $K/\Q$.
	\label{rem:coefffield}
\end{remark}

\begin{remark}
  If $K$ is real quadratic the same proof gives an extension, but we
  cannot distinguish whether the Nebentypus equals $\varepsilon$ or
  $\varepsilon \mu_K$. Our proof deeply used the fact that by Ribet's
  result, the extension is odd, but we do not get any information for
  real quadratic fields $K$, as complex conjugation is an element of
  $\Gal_K$. In the forthcoming article (\cite{2103.06965}) a
  completely different approach will be presented to solve this issue.
\end{remark}	

Let $\tilde{\rho}_p$ denote the extension of
$\rho_{\E,p}\otimes \chi$, a modular representation of
$\Gal_\Q$.

\begin{coro}
  Suppose that $p \nmid 2d$ and suppose that the residual Galois
  representation $\overline{\tilde{\rho}_p}$ is absolutely irreducible. Then there exists a newform $g \in S_2(\Gamma_0(n),\varepsilon)$, where
  \[
  n=2^e \cdot \prod_{q\in Q_3} q\cdot \prod_{q \in Q_1 \cup Q_5 \cup Q_7}q^2,  
\]
for $e$ as in Theorem~\ref{thm:levelandnebentypus}, such that
$\rho_{\E,p} \equiv \rho_{g,K,p}\otimes \chi^{-1} \pmod{\id{p}}$, where
$\rho_{g,K,p}$ is the restriction of the representation $\rho_{g,p}$
to the Galois group $\Gal_K$ and $\id{p}$ is a prime ideal of
$\overline{\Q}$ dividing $p$.
\label{coro:loweringlevelE}
\end{coro}
\begin{proof}
  Let $\id{q}$ be a prime ideal of $\Om_K$ dividing $N(E_{(a,b,c)})$ but
  not dividing $p$. By Lemma~\ref{lemma:loweringthelevel} and by the well known
  Hellegouarch's result the residual representation
  $\overline{\rho_{\E,p}}$ is unramified at $\id{q}$. Since $\chi$ is
  unramified at $\id{q}$, the same holds for
  $\overline{\rho_{\E,p}\otimes \chi}$, and since $K/\Q$ is unramified
  at $\id{q}$, the image of inertia of $\overline{\tilde{\rho}_p}$
  matches that of $\overline{\rho_{\E,p}}$. In particular,
  $\overline{\tilde{\rho}_p}$ is unramified at all primes not dividing
  $2dp$. The \emph{finite} hypothesis (to remove $p$ also from the
  level) at $\id{p}$ for primes $\id{p}$ dividing  $p$ follows from the same
  argument given in \cite{MR2075481} (page 783) under our assumption
  that $p$ does not ramify in $K/\Q$. Our absolutely irreducible
  assumption implies that we are in the hypothesis of Ribet's lowering
  the level result (see \cite{MR1047143}) so there exists an eigenform
  $g \in S_2(n,\varepsilon)$ with $n$ only divisible by ramified
  primes and probably by the prime two which is congruent to our representation $\tilde{\rho}_p$ modulo
  $\id{p}$ for some prime ideal $\id{p}$ dividing $p$.
\end{proof}
It is important to remark that the level, weight and Nebentypus of the
newform $g$ does not depend neither on the particular solution
$(a,b,c)$ nor on the prime $p$.

\begin{prop}
  \label{prop:conductorCM}
  The extension of the trivial solutions $(\pm 1,0,1)$ corresponds to
  forms $g_{\pm}$ with complex multiplication in the space
  $S_2(\Gamma_0(n),\varepsilon)$ where
  $n=2^8 \cdot \prod_{q\in Q_3} q\cdot \prod_{q \in Q_1 \cup Q_5 \cup
    Q_7}q^2$.
\end{prop}
\begin{proof}
  By Lemma~\ref{lemma:CM1} and Remark~\ref{rem:conductorCM} the
  conductor exponent of the elliptic curve with complex multiplication
  attached to the trivial solution $(\pm 1,0,1)$ equals $8$ when $2$
  does not ramify in $K/\Q$, it equals $12$ if $d \equiv 1\pmod 4$ and
  $10$ if $2\mid 10$. In all cases, formula~(\ref{eq:Henniart})
  implies that the extension has exponent valuation $8$ at $2$.
\end{proof}
	
\subsection{The case $\rho_{\Et,p}$}
As in the previous section, let $\varepsilon$ be the real character
constructed in Section~\ref{section:t=3} and let $S(\Et)$ be the set
of odd primes different from $3$ of bad reduction of $\Et$. Abusing
notation, if $q$ is a rational prime, by $q \in S(\Et)$ we will mean
that there exists a prime ideal of $\Om_K$ dividing $q$ in the set $S(\Et)$.

Before stating the result, let us make some remarks regarding the
conductor of the twisted representation $\rho_{\Et,p}\otimes
\chi$. Let $\id{q}$ be an odd prime ramifying in $K/\Q$ not dividing
$3$. The curve $\tilde{E}_{(a,b,c)}$ has additive reduction at all
such primes and its local type (by Remark~\ref{rem:twistlevel}) is
that of a principal series (given by a character whose inertial part
has order $3$) or a supercuspidal representation. Since the inertial
part of $\chi_{\id{q}}$ has order a power of two, it cannot cancel the
inertial contribution of $\rho_{\Et,p}$, hence the conductor of the
twisted representation at $\id{q}$ has still valuation $2$.

At primes dividing $2$, the conductor valuation of $N(\Et)$ never
matches the square of the conductor of $\chi_{\id{p}_2}$ (see
Lemma~\ref{lemma:case2conductor2} and
Remark~\ref{rem:conductorchi2Et}) hence the twisted representation has
conductor the least common multiple of both quantities. At primes
dividing $3$ there is a situation where the twisted representation has
smaller conductor than the elliptic curve. It happens precisely when
$3$ is inert in $K/\Q$ and $\tilde{E}_{(a,b,c)}$ has conductor
valuation $2$. In such case, the local type of the Weil-Deligne
representation is that of a principal series whose inertia is given by
an order $4$ character (by Remark~\ref{rem:case3inert}). Then twisting
by $\chi_3$ (also a character of order $4$ while restricted to the
inertia subgroup) cancels one of the characters and the twisted
representation has conductor valuation $1$.

\begin{thm}
  Suppose that $K/\Q$ is imaginary quadratic. Then the twisted
  representation $\rho_{\tilde{E},p}\otimes \chi$ descends to a
  $2$-dimensional representation of $\Gal_\Q$ attached to a newform
  $g$ of weight $2$, Nebentypus $\varepsilon$ and level $N$ given by
  \[
    N=2^{a}\cdot3^b \cdot \prod_{q \in S(\Et)}q\cdot \prod_{q \in Q_{\pm
        \pm}}q^{2}.
  \]
  The value of $a$ is one of:
  \[
    a=\begin{cases}
      2 & \text{ if } 2 \text{ is inert},\\
      1,2 & \text{ if }2 \text{ splits},\\
      4 & \text{ if }2 \text{ ramifies but }2\nmid d,\\
      8 & \text{ if }2 \mid d.
    \end{cases}
  \]
  and the value of $b$ is one of
  \[
    b = \begin{cases}
      2,3 & \text{ if } 3 \text{ is split},\\
      1,3 & \text{ if } 3 \text{ is inert},\\
      5 & \text{ if } 3 \text{ ramifies}.
    \end{cases}
  \]
  Furthermore, the coefficient field is some quadratic extension of
  $\Q(\chi)$.
  \label{thm:levelandnebentypus2}
\end{thm}
\begin{proof}
  The extension existence result and its Nebentypus description is
  proven exactly in the same way as in
  Theorem~\ref{thm:levelandnebentypus}. The conductor result is clear
  for all odd unramified primes, since the curve has semistable
  reduction (by Lemma~\ref{lemma:multredEt}) and $\chi$
  is unramified at such primes.
  
  If $\id{q}\mid q$ is an odd ramified prime ideal not dividing $3$,
  Lemma~\ref{lemma:conductorramifiedprimes} implies that
  $v_{\id{q}}(\rho_{\tilde{E},p}) = 2$. The same holds for
  $\rho_{\tilde{E},p}\otimes \chi$ (as explained before). Then the
  conductor formula of the induced representation (\ref{eq:Henniart})
  implies that the four dimensional representation has conductor 
  valuation $4$ at such primes, so the extended representation has
  conductor valuation $2$.

  At the prime $2$ the character $\chi_2$ is unramified when $2$ is
  unramified in $K/\Q$, hence the twisted representation (and its
  extension) has the same conductor as $\Et$. If $d \equiv 1 \pmod 4$
  then $\chi_{2}$ has conductor $2$, so the twisted
  representation has conductor valuation $4$ and its induction
  conductor valuation $8$, so $\tilde{\rho}_p$ has conductor valuation
  $4$ at the prime $2$. When $2 \mid d$, $\chi_{2}$ has conductor $5$, so the
  twisted representation valuation $10$, the induced one $16$ and
  $\tilde{\rho}_p$ has conductor valuation $8$ as claimed.

  At last we need to prove the exponent for the prime $3$. If $3$
  ramifies in $K/\Q$, $v_{\id{p}_3}(N(\Et))=8$ (by
  Lemma~\ref{lemma:case2conductor3}) and the character has conductor
  valuation at most $1$, so the induced representation has valuation $10$ and
  $\tilde{\rho}_p$ has valuation $5$. If $3$ splits, the twisted
  representation has conductor exponent $2$ or $3$ (so does
  $\tilde{\rho}_p)$, while in the inert case, as explained before, the
  twisted representation has conductor valuation $1$ or $3$.
\end{proof}
Let $\tilde{\rho}_p$ denote the extension of
$\rho_{\Et,p}\otimes \chi$.

\begin{coro}
  Suppose that $p \nmid 6d$ and suppose that the residual Galois
  representation $\overline{\tilde{\rho}_p}$ is absolutely
  irreducible. Then there exists a newform
  $g \in S_2(\Gamma_0(n),\varepsilon)$, where
  \[
  n=2^a\cdot3^b \cdot \prod_{q\in Q_{\pm \pm}} q^2,  
\]
for $a$ and $b$ as in Theorem~\ref{thm:levelandnebentypus2} such that
$\rho_{\Et,p} \equiv \rho_{g,K,p}\otimes \chi^{-1} \pmod{\id{p}}$, where
$\rho_{g,K,p}$ is the restriction of the representation $\rho_{g,p}$
to the Galois group $\Gal_K$ and $\id{p}$ is a prime ideal of
$\overline{\Q}$ dividing $p$.
\label{coro:loweringlevelEt}
\end{coro}
\begin{proof} Mimics the one of Corollary~\ref{coro:loweringlevelE}.
\end{proof}
We can give a precise formula for the level of the form attached to
the trivial solution $(\pm 1,0,1)$.
\begin{prop}
  \label{prop:conductorCM2}
  The extension of the trivial solutions $(\pm 1,0,1)$
  correspond to forms $g_{\pm}$ with complex multiplication in the space
  $S_2(\Gamma_0(n),\varepsilon)$ where
  $n=2^a\cdot3^b \cdot \prod_{q\in Q_{\pm \pm}} q^2$, with
  \[
    a = \begin{cases}
      2 & \text{ if }d\equiv 3 \pmod 4,\\
      4 & \text{ if }d \equiv 1\pmod 4,\\
      8 & \text{ if } 2\mid d.
    \end{cases}
  \]
  and with
  \[
    b = \begin{cases}
      3 & \text{ if }3 \nmid d,\\
      5 & \text{ if }3 \mid d.
    \end{cases}
    \]
\end{prop}
\begin{proof}
  Follows the same proof of Theorem~\ref{thm:levelandnebentypus2} with
  the precise formula for the conductor exponent of
  $\tilde{E}_{(\pm 1,0,1)}$ given in Remark~\ref{rem:conductorCM2}.
\end{proof}
\section{Irreducibility of the residual representations of $\rho_{\E,p}$
  and  of $\rho_{\Et,p}$}
\label{section:levellowering}

To use Ribet's level lowering result we need to assure that the
residual representation $\overline{\tilde{\rho}_p}$ is absolutely
irreducible. Ellenberg's result on the image of Galois representations
attached to $\Q$-curves (\cite{MR2075481}) is very useful for this
purpose as it not only provides irreducibility but also implies large image of
the residual representation (i.e.
$\SL_2(\F_p)$ is contained in the image of the residual projective
representation). This plays a crucial role while trying to discard
forms with complex multiplication (a major problem that will be
addressed later). Ellenberg's theorem holds under the hypothesis that
there exists of a prime $\id{q}$, not dividing $6$, where the
$\Q$-curve has multiplicative reduction.  Recall from
Lemma~\ref{lemma:multred} (respectively Lemma~\ref{lemma:multredEt})
that $\E$ (respectively $\Et$) has multiplicative reduction at all
primes $q$ larger than $2$ (respectively $3$) dividing $c$.

Then
we are led to distinguish between two different cases: namely when $c$
is supported only at the primes $\{2,3\}$ (i.e. $2$ and $3$ are the unique
primes in its factorization) where Ellenberg's result does not apply, and
when $c$ is divisible by a prime larger than $3$.

When $c$ is only supported at $\{2,3\}$, the value of the level $N$
belongs to a finite list, hence we can compute all such spaces and try
to discard their forms for not being related to putative
solutions. This is precisely the strategy we follow while studying
equation~(\ref{eq:ben-chen}). However, the prime $3$ does not play any
special role while studying solutions of (\ref{eq:24p}), hence if
$(a,b,c)$ is a non-trivial primitive solution of it, and $3\mid c$,
then we are in the hypothesis of Ribet's lowering the level result. A
crucial hypothesis in Ribet's theorem is that the residual image is
absolutely irreducible, hence we need a different approach to achieve
such a goal.

\subsection{The case $c$ supported in $\{2,3\}$.} The results of the
present subsection will only be applied while studying solutions of
equation~(\ref{eq:24p}). Suppose that the unique possible
primes dividing $c$ are $2$ and $3$. Recall that $c$ being divisible
by $2$ (respectively by $3$) implies that $d \equiv 7 \pmod 8$
(respectively $d \equiv 2 \pmod 3$) by Lemma~\ref{lemma:odd}.




\begin{thm}
  There exists an explicit bound $N_K$ such that if $p>N_K$ and
  $(a,b,c)$ is a primitive solution of (\ref{eq:24p}) where $c$ is only
  supported in $\{2,3\}$, then the representation $\rho_{\E,p}$ has
  absolutely irreducible image.
  \label{thm:multiplicativeprime}
\end{thm}
\begin{proof}
  The proof mimics the one presented in \cite{MR2561200} (case
  $(ii)$). 
  Suppose that $c=2^\alpha3^\beta$ so the curve $\E$ has conductor
  $2^a\cdot 3^b$, where $a \le 12$, $b \le 1$ (by
  Lemma~\ref{lemma:2valuation}). Let $\epsilon_3 = 1$ if $3 \mid c$
  and $0$ otherwise. Then
  $N(\tilde{\rho})= 2^s 3^{\epsilon_3}\cdot \prod_{q\in Q_3} q\cdot
  \prod_{q \in Q_1 \cup Q_5 \cup Q_7}q^2$, where $s \le 9$ by
  Theorem~\ref{thm:levelandnebentypus}.  Suppose that the residual
  image of $\tilde{\rho}_p$ is reducible, i.e.
  \begin{equation}
    \label{eq:ss}
    \overline{\tilde{\rho}_p}^{s.s} \simeq \nu \oplus \varepsilon \nu^{-1}\overline{\cyclo},
  \end{equation}
  where as before $\cyclo$ denotes the $p$-adic cyclotomic character.
  In particular, since the conductor of the reduced representation
  divides the conductor of $\tilde{\rho}$,
  $\cond(\nu) \mid N(\tilde{\rho})$. Let $\ell$ be a prime number such
  that $\ell \equiv 1 \pmod{\lcm(4d,\cond(\nu))}$.  In particular,
  $\ell$ splits in $K/\Q$, say $\ell = \id{l}\overline{\id{l}}$.  Then
  on the one hand
  $\trace(\tilde{\rho}_p(\Frob_\ell))=a_{\id{l}}(\E)\chi(\id{l})$ is
  an integer (as the form has an inner twist) and by Hasse's bound it
  satisfies that $|a_{\id{l}}(\E)| \le 2\sqrt{\ell}$. On the other
  hand, (\ref{eq:ss}) and our assumption on $\ell$ implies that
  $\trace(\overline{\tilde{\rho}_p}^{s.s}(\Frob_\ell)) = \ell + 1$. In
  particular, $p \mid \ell+1-a_{\id{l}}(\E)\chi(\id{l})$ (which is
  non-zero), which cannot occur if $p>2\sqrt{\ell}+\ell+1$.
\end{proof}
\begin{remark}
  The constant $N_K$ depends on the first prime congruent to $1$
  modulo $\lcm(4d,N(\tilde{\rho}))$. We can improve the bound for the
  conductor of $\nu$ in the previous theorem. From (\ref{eq:ss}) it
  follows that actually
  $\cond(\nu)\cdot \cond(\nu^{-1}\varepsilon) \mid
  N(\tilde{\rho})$. If $3 \mid c$ then $\varepsilon$ is unramified at
  $3$, and $\tilde{\rho}$ has conductor exponent one at $3$, then
  $3 \nmid \cond(\nu)$. For odd ramified primes $q$, since the
  conductor of $\varepsilon$ is square-free,
  $v_q(\cond(\nu)) \le 1$. Then if
  $\tilde{d} = \frac{d}{2^{v_2(d)}}$ (i.e. the prime to $2$ part of
  $d$) actually $\cond(\nu) \mid 2^4\cdot \tilde{d}$, so we can
  replace the least common multiple by $16\tilde{d}$.

  Regarding the least minimum prime number $\ell$ congruent to $1$
  modulo $16\tilde{d}$, according to Dirichlet's theorem,
  $1/\varphi(16\tilde{d})$-th of the primes are congruent to $1$
  modulo $16\tilde{d}$, but giving a precise bound on the first such
  prime is very ineffective for computational purposes.
\label{rem:largeimagebound}
\end{remark}

\subsection{The case $c$ not supported in $\{2,3\}$}

Assume on the contrary that there exists an odd prime $q$ dividing $c$
and not dividing $3$ (so in particular $q \nmid d$). Then we are in the hypothesis of Ellenberg's big
image result (\cite[Theorem 3.14]{MR2075481}).
\begin{thm}[Ellenberg] Given $d$ a positive integer, suppose that
  $E/\Q(\sqrt{-d})$ is a $\Q$-curve which has multiplicative reduction
  at an odd prime $\id{q}$ not dividing $3$. Then, there exists an
  integer $N_d$ such that the projective image of the residual
  representation of $\rho_{E,p}$ is surjective for all primes
  $p >N_d$. Furthermore, some explicit values of $N_d$ are the
  following: $N_1 = 7$, $N_2 = 5$, $N_3 = 7$, $N_5 = 499$, $N_6=563$
  and $N_7=349$.
  \label{thm:ellenberg}
\end{thm}
\begin{proof}
  This result was proved in \cite{MR2075481}, where a method to bound $N_d$ was given. The explicit value of
  $N_d$ was given in: \cite{MR2646760} for $d=1$ and $d=2$ and in
  \cite{MR2561200} for all values of $d$ whose field conductor is at
  most $8$ (with an improvement via a finite computation given in
  \cite{Angelos} for $d=3$). We only need to prove the cases $d=5$ and
  $d=6$.

  Let $N$ be any positive integer, and $\chi$ the character
  corresponding to $K/\Q$ (of conductor ${\mathfrak f}$). Let $\FF$ be
  a Petersson-orthogonal basis for $S_2(\Gamma_0(N))$. Define
\[
(a_m,L_\chi)_N=\sum_{f\in \FF}a_m(f)L(f\otimes\chi,1).
\]
If $M \mid N$, define $(a_m,L_\chi)_N^M$ as the contribution from the
old forms of level $M$. Then one of the main results of Ellenberg in
\cite{MR2075481} is that if for some prime $p$ the value
\begin{equation}
  \label{eq:Ellenbergexplicit}
  (a_1,L_\chi)_{p^2}^{p-\text{new}}=(a_1,L_\chi)_{p^2}-p(p^2-1)^{-1}(a_1-p^{-1}\chi(p)a_p,L_\chi)_p
\end{equation}
is non-zero, the residual image for the prime $p$ is large. Also in the aforementioned
article, lower bounds for $(a_1,L_\chi)_{p^2}^{p-\text{new}}$ are
presented, where the main contribution comes from the first term. In
\cite{MR2176151} (Theorem 1) the following formula is proven
\[
(a_m,L_\chi)_{p^2} = 4\pi \chi(m)e^{-2\pi m/\sigma N log(N)} -E^{(3)}+E_3-E_2-E_1+(a_m,B(\sigma N\log(N))),
\]
where $\sigma$ is taken to be $\frac{q^2}{2\pi}$ and $N=p^2$. Theorem
1 of \cite{MR2176151} provides the following bounds:
\begin{itemize}
\item
  $|(a_m,B(\sigma N \log(N)))| \le 30
  (400/399)^3exp(2\pi)q^2m^{3/2}N^{-1/2}d(N)N^{-2\pi \sigma/q^2}$,
    where $d(N)$ denote the number of divisors of $N$.
    
  \item $|E_1|\le (16/3)\pi^3m^{3/2}\sigma \log(N)\exp(-N/2 \pi m\sigma \log(N))$.
    
  \item $|E_3| \le (8/3) \zeta(3/2)^2 \pi^3\sigma m^{3/2} N^{-1/2}\log(N) d(N)\exp(-N/2 \pi m \sigma \log(N))$.
    
  \item $|E^{(3)}|\le 16\pi^3m \sum_{c>0, N\mid c}\min\{\frac{2}{\pi}\phi(q)c^{-1}\log(c),\frac{1}{6}\sigma N\log(N)m^{1/2}c^{-3/2}d(c)\}$, where $\phi$ is Euler's function.
\end{itemize}
Also the following bounds hold:
\begin{itemize}
\item If the field discriminant is even, then by Proposition 10 of
  \cite{MR2646760}
  \begin{multline*}
  |E_2| \le 64q\phi(q)\pi^5 m^2\left(\frac{\zeta(2)}{6}/N^2+\frac{1}{\pi}\left(\zeta(3)log\left(\frac{eN}{2}\right)\right)-\zeta'(3)N^{-3}\right)\\
+32\pi^5\zeta(7/2)^2m^{5/2}d(N)N^{-7/2}\left((\frac{N^2}{4\pi^2m}+1)(1-\theta)^{-1}+(1-\theta)^{-2}\right)\exp(-\frac{N}{2\pi \sigma m\log(N)})\\
+\frac{512}{3}\zeta(11/2)^2\pi^7m^{7/2}d(N)N^{-11/2}(1-\theta^2)^{-3},
  \end{multline*}
where $x = \sigma N \log(N)$ and $\theta = \exp(-2\pi/x)$.
\item By Lemma 3.13 of \cite{MR2075481}, for $t=1$ or $t=p$,
  \[
    (a_{mt},L_\chi)_p \le 2\sqrt{3}m^{1/2}d(m)(1-\exp(-2 \pi/q\sqrt{p}))^{-1}(4\pi + 16\zeta^2(3/2)\pi^2p^{-3/2}).
    \]
\end{itemize}
We wrote a script in \verb*|PARI/GP| to compute the bound for
$ (a_1,L_\chi)_{p^2}^{p-\text{new}}$ using the previous bounds. The
function depends on a parameter $M$ which is used to compute the bound
for $E^{(3)}$. Its role is to take the first function of the minimum
for all values of $c$ up to $MN$ and the second one for the rest of
them. Here is an example:
\begin{verbatim}
EllenbergBound(503,20,503^2*800)
%1 = 0.022830368857725150995580770803441539591
\end{verbatim}
Since the output is positive, we can take $N_5$ as the prime prior to $503$. Similarly, 
\begin{verbatim}
? EllenbergBound(571,24,571^3*10)
%2 = 0.22729177780123987468499355952058729725
\end{verbatim}
proves the bound $N_{6} = 563$.
\end{proof}

\section{Strategies to discard newforms}
\label{section:methostodiscard}
We need to prove that the newforms in the space
$S_2(\Gamma_0(n),\varepsilon)$ obtained from
Corollary~\ref{coro:loweringlevelE}, while studying
equation~(\ref{eq:24p}), and Corollary~\ref{coro:loweringlevelEt}, while
studying equation~(\ref{eq:ben-chen}), are not related to non-trivial primitive solutions. 
  To discard newforms we will mostly use a strategy due to Mazur
  (which in practice works better for forms without complex
  multiplication). 

\begin{prop}[Mazur's trick]
  Let $(a,b,c)$ be a non-trivial primitive solution, and
  $g\in S_2(n,\varepsilon)$ be such that
  $\overline{\rho_{\E,p}\otimes \chi} \simeq \overline{\rho_{g,K,p}}$. Let $q$
  be a rational prime with $q\nmid pn$. Let $\id{q}$ be a prime of $\Om_K$
  dividing $q$ and define
  \[
    B(q,g;a,b)=\begin{cases}
      N(a_{\id{q}}(\E)\chi(\id{q})-a_q(g)) & \text{ if } q \nmid c \text{ and } q \text{ splits in }K,\\
      N(a_q(g)^2-a_q(\E)\chi(q)-2q\varepsilon(q)) & \text{ if } q \nmid c \text{ and } q \text{ is inert in }K,\\
      N(\varepsilon^{-1}(q)(q+1)^2-a_q(g)^2) & \text{ if }q \mid c.
    \end{cases}
  \]
  Then $p\mid B(q,f;a,b)$.
  \label{prop:Mazurtrick}
\end{prop}
\begin{proof} If $h$ is a rational newform in
  $S_2(\Gamma_0(n),\varepsilon)$ and $K$ is a quadratic field, recall
  the well known formula for the Fourier coefficients of the modular
  form $H$ obtained as the automorphic base change of $h$ to $K$: if
  $q = \id{q}\overline{\id{q}}$ (is split) then $a_{\id{q}}(H) = a_q(h)$,
  while if $q$ is inert in $K$ then
  $a_q(H)= a_q(g)^2-2q\varepsilon(q)$.

  Taking $h=g$ gives the first two statements. The last one
  corresponds to ``Ribet's lowering the level'' condition. The proof
  of this fact (well known to experts) mimics the one given in
  \cite[Lemma 24]{MR2966716}.
\end{proof}
The way we apply the last proposition is as follows: take an odd prime $q$
not dividing $d$, and define
\[
C(q,g)=\prod_{(a,b)\in\F_q^2} B(q,g;a,b),
\]
where the product is over non-zero elements satisfying
equation~(\ref{eq:24p}) modulo $q$. Then the prime $p$ must divide all
the values $C(q,g)$, so we compute a couple of them and compute their
gcd obtaining in some cases a bound for $p$. However, it happens
sometimes that the value $C(q,g)$ is always zero, in which case we cannot discard the newform $g$ for any prime $p$.
A similar strategy holds while studying solutions of
(\ref{eq:ben-chen}), using the curve $\Et$.

\vspace{1pt}

Recall that the trivial solutions $(\pm 1,0,1)$ correspond to  newforms
$g_{\pm}$ with complex multiplication. Since they are solutions for
all values of $p$, Proposition~\ref{prop:Mazurtrick} implies that
$p \mid C(q,g_{\pm})$ for all primes $p$, so $C(q,g_{\pm}) = 0$. In
particular, a safe check of any implementation of Mazur's trick should
fail to discard such forms with complex multiplication.

Modular forms with complex multiplication however have the property
that the image of their Galois representations are not as large as
expected (their image lies in the normalizer of a Cartan group, since
they are the induction of a one dimensional representation from an index
two subgroup of $\Gal_\Q$). In particular, if we can prove that given a non-trivial
primitive solution $(a,b,c)$ there exists a prime $q>3$ dividing $c$
then Ellenberg's large image result (Theorem~\ref{thm:ellenberg})
implies that our representation $\tilde{\rho}_p$ has surjective
projective image for $p$ large enough, hence cannot be congruent to a newform with complex
multiplication and we have good chances to prove non-existence of
non-trivial primitive solutions.

\vspace{2pt}

\noindent
{\bf Problem 2:} how can we discard newforms with complex multiplication when $c$ is supported at $\{2,3\}$?

\vspace{2pt}

This is a very delicate problem, and at the time, we do not know of a
general answer to it. However, for the examples
of the present article some naive ideas are enough to get results.

By Lemma~\ref{lemma:odd}, if $d \not \equiv 7 \pmod 8$ then $c$ cannot
be even. If $d \equiv 7 \pmod 8$ then either $2 \mid ab$, in which case
$c$ is odd, or $2\nmid ab$, in which case $c$ is even. Recall that when
$d \equiv 7 \pmod 8$, there are two possible conductor exponent values
at the prime $2$ (depending on the parity of $ab$). Then in the space
corresponding to solutions where $ab$ is even (where our trivial
solution lies), we know that $c$ is odd, so we have good chances that
it is divisible by a prime larger than $3$ (as $c$ cannot be $1$) and
we can discard the newforms with complex multiplication as they are
not related to non-trivial primitive solutions by Ellenberg's
result. In the other space, there is a priori no reason for Mazur's
trick to fail. This is precisely the case when we study
equation~(\ref{eq:24p}) for $d=7$. The same phenomena occurs while
studying non-trivial primitive solutions $(a,b,c)$ of
equation~(\ref{eq:ben-chen}) with $c$ divisible by $3$.

For this strategy to work completely while studying
equation~(\ref{eq:24p}), we need to rule out the possibility of $c$
being a power of $3$ (so in particular $d \equiv 2 \pmod 3$). Suppose then that $(a,b,c)$ is a non-trivial primitive solution of ~(\ref{eq:24p}) with $c$ divisible by
$3$. Then the modular form $f$ (attached to the extension of
$\rho_{\E,p}\otimes \chi$ by Theorem~\ref{thm:levelandnebentypus}) has
level divisible by $3$ and is congruent to a form $g$ whose level is
not. In particular, we are in the ``lowering the level'' hypothesis at
$3$, hence
  \begin{equation}
    \label{eq:levelloweringat3}
N(\varepsilon^{-1}(3)(3+1)^2-a_3(g)^2) \equiv 0 \pmod p.
  \end{equation}
  In practice, this gives a bound for the possible exponents $p$ where
  a solution supported at $\{3\}$ can exist.

%
%

\section{Solutions of equation $x^4+dy^2=z^p$ for small values of $d$}
\label{section:solvingequation1}
In the present section we study solutions of (\ref{eq:24p}) for
square-free values of $d$ up to ten. For that purpose we follow the
general strategy described in Section \ref{section:generalstrategy}. The
algorithms used for the following examples are available at the web
page \url{http://sweet.ua.pt/apacetti/research.html} as well as the outputs of the computations (in a file labeled \textit{``OutputsEq1.txt''}).

\subsubsection{The equation $x^4+dy^2=z^p$, for $d=1,2,3$} For $d=1,2$
the equation was completely solved in \cite{MR2646760}, where the
authors proved that there are no non-trivial primitive solutions for
$p>2$. For $d=3$, in \cite[Theorem 3]{MR2561200} it was
proved the non existence of non-trivial primitive solutions for
$p>131$. The bound $131$ comes from Ellemberg's bound (see Lemma 8 of
\textit{loc.cit}), but applying \cite[Proposition 5.4]{Angelos} it can be
lowered to $7$, proving non existence of non-trivial primitive
solutions for $p>17$.

\subsubsection{The equation $x^4+5y^2=z^p$} This equation was not
considered before, and our method gives the following result.

\begin{thm}
  Let $p > 499$ be a prime number. Then there are no non-trivial primitive
  solutions of the equation
  \[
    x^4+5y^2=z^p.
  \]
  \label{thm:eq1d5}
\end{thm}
\begin{proof}

  Theorem~\ref{thm:levelandnebentypus} implies that $\varepsilon$ is a
  character of order $4$ and conductor $4\cdot 5$ while $\chi$ has
  order $8$. Let $(a,b,c)$ be a non-trivial primitive solution. Since
  $d \not \equiv 7 \pmod 8$, $c$ cannot be even, so either $c$ is a
  power of $3$, or it is divisible by a prime larger than $3$. In
  the later case, Theorem~\ref{thm:ellenberg}
  implies that $N_d = 499$ so if $p > 499$ the image is large (in
  particular absolutely irreducible). If $c$ is a power of three, we
  can use Theorem~\ref{thm:multiplicativeprime} (and
  Remark~\ref{rem:largeimagebound}) with the prime $\ell = 241$ giving
  the bound $N_K=273$. In particular, if $p > 499$ we are in the
  lowering the level hypothesis, hence by
  Corollary~\ref{coro:loweringlevelE} there exists a newform $g$ in
  $S_2(\Gamma_0(2^7\cdot5^2),\varepsilon)$ or in
  $S_2(\Gamma_0(2^8\cdot5^2),\varepsilon)$.  whose Galois
  representation is congruent modulo $p$ to $\E \otimes \chi$.

\vspace{3pt}

\noindent $\bullet$ The space $S_2(\Gamma_0(2^7\cdot5^2),\varepsilon)$
has $12$ Galois conjugacy classes (of newforms), none of them with complex multiplication. Using Mazur's trick
for primes $3\le q\le 30$ different from $5$ and every $g$ in the space we conclude that
$p\in\{2,7,13\}$.

\vspace{3pt}

\noindent $\bullet$ The space $S_2(\Gamma_0(2^8\cdot5^2),\varepsilon)$
has $55$ Galois conjugacy classes, $24$ of them with complex
multiplication (one of them corresponding to the trivial solution by
Proposition~\ref{prop:conductorCM}). Applying Mazur's trick for primes
$3\le q \le 20$ different from $5$ to forms without complex
multiplication we conclude that $p\in\{2,3,5,7,11\}$. If $c$ is
divisible by a prime larger than $3$ then Ellenberg's result implies
that our form cannot be congruent to a newform with complex
multiplication. If $c$ is a power of $3$, then the forms with complex
multiplication should satisfy the raising the level hypothesis,
i.e. that $p \mid N(16\varepsilon^{-1}(3)-a_3(g)^2)$. This implies
that $p\in\{2,3,5,29,101,139\}$.
\end{proof}

Here is how our script works: in \verb*|Magma| just load the file ``\textit{Eq1d5.mg}'' to get the previous statements.
%
\verbatimfont{\footnotesize}
\begin{verbatim}
> load "Eq1d5.mg"; 
Loading "Eq1d5.mg"
Loading "Mazur42p.mg"
Forms in Space 2^7*5^2:
Forms with CM:
[]
Primes obtained via Mazur's trick for non-CM forms:
{@ 2 @}
{@ 2 @}
{@ 2 @}
{@ 2 @}
{@ 2, 7 @}
{@ 2, 7 @}
{@ 2, 7 @}
{@ 2, 7 @}
{@ 13 @}
{@ 13 @}
{@ 13 @}
{@ 13 @}
Forms in Space 2^8*5^2:
Forms with CM:
[ 1, 2, 3, 4, 5, 6, 7, 8, 13, 14, 17, 18, 21, 22, 23, 24, 29, 30, 31, 32, 33,
34, 44, 45 ]
Primes obtained via Mazur's trick for non-CM forms:
{@ 11, 2, 7 @}
{@ 11, 2, 7 @}
{@ 11, 2, 7 @}
{@ 11, 2, 7 @}
{@ 2, 7 @}
{@ 2, 7 @}
{@ 2, 7 @}
{@ 2, 7 @}
{@ 7 @}
{@ 7 @}
{@ 7 @}
{@ 7 @}
{@ 3 @}
{@ 3 @}
{@ 2, 3 @}
{@ 2, 3 @}
{@ 3, 7 @}
{@ 3, 7 @}
{@ @}
{@ @}
{@ 2, 3, 5, 7 @}
{@ 2 @}
{@ 2 @}
{@ 2 @}
{@ 2 @}
{@ 3 @}
{@ 3 @}
{@ 3 @}
{@ 3 @}
{@ 2, 5 @}
{@ 2, 5 @}
Primes obtained via Mazur's trick for CM forms:
{@ 2 @}
{@ 2 @}
{@ 2 @}
{@ 2 @}
{@ 2 @}
{@ 2 @}
{@ 2 @}
{@ 2 @}
{@ 2, 3 @}
{@ 2, 3 @}
{@ 2, 3 @}
{@ 2, 3 @}
{@ 2, 5 @}
{@ 2, 5 @}
{@ 2 @}
{@ 2 @}
{@ 3, 139 @}
{@ 3, 139 @}
{@ 3, 139 @}
{@ 3, 139 @}
{@ 2, 29 @}
{@ 2, 29 @}
{@ 5, 101 @}
{@ 5, 101 @}
\end{verbatim}
\verbatimfont{\normalsize}
\subsubsection{The equation $x^4+6y^2=z^p$} In this case we can prove the following result.
\begin{thm}
  Let $p > 563$ be a prime number. Then there are no non-trivial primitive
  solutions of the equation
  \[
    x^4+6y^2=z^p.
  \]
  \label{thm:eq1d6}
\end{thm}
\begin{proof}
  Note that since $6 \mid d$, the $c$-value of a non-trivial primitive
  solution $(a,b,c)$ cannot be supported in $\{2,3\}$, so in this case we always
  get large image for $p> 563$ (by
  Theorem~\ref{thm:ellenberg}). Theorem~\ref{thm:levelandnebentypus}
  implies that $\varepsilon$ is a character of conductor $4\cdot3$ and
  order $2$ whereas $\chi$ has order $4$.
  Corollary~\ref{coro:loweringlevelE} implies the existence of a
  newform $g$ attached to a non-trivial primitive solution $(a,b,c)$ lying in
  one of $S_2(\Gamma_0(2^8\cdot 3),\varepsilon)$ or
  $S_2(\Gamma_0(2^9\cdot 3),\varepsilon)$ congruent modulo $p$ to
  $\rho_{\E,p}\otimes \chi$.

\vspace{3pt}

\noindent $\bullet$ The space $S_2(\Gamma_0(2^8\cdot3),\varepsilon)$
has $10$ Galois conjugacy classes. Six of them have complex
multiplication (one of them corresponding to the trivial solution by
Proposition~\ref{prop:conductorCM}). Running Mazur's trick for $q=5$
and $q=7$ for each non-CM form we conclude that $p\in\{2,7\}$.

\vspace{3pt}

\noindent $\bullet$ The space $S_2(\Gamma_0(2^9\cdot3),\varepsilon)$
has $13$ Galois conjugacy classes, three of them with complex
multiplication. Applying Mazur's trick for primes $5\le q\le 20$ to each
non-CM form we conclude that $p\in\{2,5,7\}$.
\end{proof}
\begin{remark}
  If Ellenberg's constant in Theorem~\ref{thm:ellenberg} could be improved to
  $11$ (as we expect), the previous result would hold for $p>11$ as well.
\end{remark}

\subsubsection{The equation $x^4+7y^2=z^p$} Our method allows to prove the following result.
\begin{thm}
  Let $p>349$ be a prime number. Then there are no non-trivial primitive
  solutions of the equation
  \[
    x^4+7y^2=z^p.
  \]
  \label{thm:eq1d7}
\end{thm}
\begin{proof}
  Let $(a,b,c)$ be a non-trivial primitive solution. Since
  $7 \equiv 1 \pmod 3$ we deduce that $c$ cannot be divisible by $3$,
  but it could be a power of $2$, so to get absolutely irreducible
  image we need to use Theorem
  ~\ref{thm:multiplicativeprime}. Remark~\ref{rem:largeimagebound}
  implies we can use Theorem~\ref{thm:multiplicativeprime} with
  $\ell=113$ getting absolutely irreducible image for all primes
  $p \ge 127$. If $c$ is divisible by a prime larger than $3$ then the
  image is large for all primes $p>349$ (by Theorem~\ref{thm:ellenberg}).

  Theorem~\ref{thm:levelandnebentypus} implies that $\varepsilon$ is
  trivial while the character $\chi$ is the quadratic even character
  of conductor $7\cdot8$. Corollary~\ref{coro:loweringlevelE} implies
  the existence of a newform $g$ in $S_2(\Gamma_0(2\cdot7^2))$ or
  $S_2(\Gamma_0(2^8\cdot7^2))$ congruent to $\rho_{\E,p}\otimes \chi$
  modulo $p$.

Let us give two different ways to discard the forms in the first space.
If $g \in S_2(\Gamma_0(2\cdot 7^2))$ is a newform (candidate for
a primitive solution) its base change to $K$ gives a Bianchi modular form whose
twist by $\chi^{-1}$ must correspond to a Bianchi modular form of level
$(\frac{1+\sqrt{-7}}{2})^6 \cdot (\frac{1-\sqrt{-7}}{2}).$ Such space
can easily be computed (using Cremona's algorithm
\cite{MR743014}, available at
\url{https://github.com/JohnCremona/bianchi-progs/releases/tag/v20200713})
the result being also available at the lmfdb (\cite{lmfdb}). There are two forms
whose level has norm $128$, given by \lmfdbbmf{2.0.7.1}{128.4} and
\lmfdbbmf{2.0.7.1}{128.5}, whose level equals
$(\frac{1+\sqrt{-7}}{2})^3(\frac{1-\sqrt{-7}}{2})^4$ and its Galois
conjugate, so none comes from a primitive solution of our equation.

\vspace{3pt}

\noindent $\bullet$ The space $S_2(\Gamma_0(2\cdot 7^2))$ has $2$
Galois conjugacy classes, one of them has rational coefficients
(corresponding to an elliptic curve) and the other with coefficients
in the quadratic extension corresponding to the polynomial
$x^2-2x-7$. None of them have complex multiplication.  Mazur's trick
for primes $3\le q\le 50$ different from $7$ discards the rational form if $p$ is not in
$\{2,7,17\}$. The second form cannot be discarded using Mazur's
trick. Since it does not appear in the space of Bianchi modular forms,
its local type at $7$ must not be the correct one (so the twist by
$\chi^{-1}$ of its base change to $K$ is ramified at
$\sqrt{-7}$). Actually, \verb*|Magma| can compute such local type,
giving that the local component is indeed supercuspidal, induced from
an order $8$ character of the unramified quadratic extension of
$\Q_7$. Such local type does not match the one of our elliptic curve
(induced from an order $4$ character of the same extension), hence
they cannot be congruent.

\vspace{3pt}

\noindent $\bullet$ The space $S_2(\Gamma_0(2^8\cdot 7^2))$ has $98$
Galois conjugacy classes, $17$ of them with rational coefficients and
$30$ forms with complex multiplication (one of them corresponding to
the trivial solution by Proposition~\ref{prop:conductorCM}). Mazur's
trick (for primes $3\le q\le 20$ different from $7$) allows us to
eliminate the for forms without complex multiplication when $p$ is not
in the set $\{2,3,5,7,11,17,23,27,31\}$.
\end{proof}

\section{Solutions of equation $x^2+dy^6=z^p$ for small values of $d$}
\label{section:solvingequation2}
As in the last section, we apply the general strategy to study
solutions of the equation $x^4+dy^2=z^p$ for square-free values of $d$
up to ten.  To avoid false expectations, we want to point out that our
approach does not provide any result for $d=5,7$, and we will explain
what goes wrong in these particular cases (see \cite{2109.07291} for
partial solutions). The outputs can be founded in \textit{``OutputsEq2.txt''}.
\subsubsection{The equation $x^2+y^6=z^p$} This case was considered in \cite{MR2966716}.
\subsubsection{The equation $x^2+2y^6=z^p$} This case is very
interesting, as working with Bianchi modular forms is enough to get
the following result.
\begin{thm}
  Let $p>5$ be a prime number. Then there are no non-trivial primitive
  solutions of the equation
  \[
    x^2+2y^6=z^p .
  \]
\label{thm:eq2d2}
\end{thm}
\begin{proof}
  Following the notation of Section~\ref{section:t=3}, the sets
  $Q_{\pm, \pm}$ are all empty; $\varepsilon$ is the trivial character
  (i.e. the form $g$ does not have Nebentypus) while the character
  $\chi$ corresponds to the quadratic character $\delta_3$ at one of
  the primes dividing $3$ in $\Q(\sqrt{-2})$ and does not ramify at
  any other prime. This is a very interesting example, as the curve
  $\tilde{E}_{(a,b,c)}$ has always good reduction at $2$ and the
  $3$-part of the conductor equals $3(1+\sqrt{-2})$, $3(1-\sqrt{-2})$,
  $9$ or $27$. In particular, it is more efficient to work with
  Bianchi modular forms than with rational ones (to avoid high powers
  of $2$ in the level). The newform $g$ attached to a primitive
  solution satisfies that its base extension to $K$ and its twist by
  $\chi^{-1}$ (which equals $\chi$ as it is quadratic) gets only bad
  reduction at primes dividing $3$.  Computing the respective spaces
  (using Cremona's algorithm, although such spaces are also available
  at \cite{lmfdb}) it turns out that there are no Bianchi modular
  forms at any level but $3^3$.

  Let $(a,b,c)$ be a non-trivial primitive solution. If $c$ is
  divisible by $3$, then the proof of
  Lemma~\ref{lemma:case2conductor3} implies that our curve has
  conductor exponent $2$ at one prime dividing $3$ and $1$ at the
  other. However the space of Bianchi modular forms of such a level is
  the zero space, hence $3$ cannot divide $c$ and $c$ must be
  divisible by a prime larger than $3$, so we are in the hypothesis of
  Ellenberg's large image result.

  The space of Bianchi modular forms of weight $2$ and level $3^3$
  contains three newforms corresponding to the elliptic curves
  \lmfdbecnf{2.0.8.1}{729.4}{a}{1}, \lmfdbecnf{2.0.8.1}{729.4}{b}{1}
  and \lmfdbecnf{2.0.8.1}{729.4}{c}{1}. The second curve has complex
  multiplication and is the base change of a rational elliptic curve
  (it corresponds to the trivial solution, and cannot be congruent to
  $\tilde{E}_{(a,b,c)}$ if $p>5$ by
  Theorem~\ref{thm:ellenberg}). The other two ones (complex
  conjugate of each other) satisfy that $a_5 = -1$. It is easy to
  compute for each possible value of $(a,b)$ modulo $5$ the value of
  $a_5(\Et)$ and verify it belongs to the set $\{2,-7,-10\}$ hence
  both elliptic curves cannot be congruent to an elliptic curve coming
  from a non-trivial primitive solution if $p>5$.
\end{proof}

\subsubsection{The equation $x^2+3y^6=z^p$} This case was considered
in \cite{Angelos}.

\subsubsection{The equation $x^2+5y^6=z^p$} Following our general
strategy, the only non-empty set is $Q_{+-}=\{5\}$ so the character
$\varepsilon$ has conductor $4\cdot 5$, and its local component at $5$
has order $4$. In particular $\chi$ has order
$8$. Corollary~\ref{coro:loweringlevelEt} implies that we need to
compute the spaces $S_2(\Gamma_0(2^4\cdot 3^2 \cdot 5^2),\varepsilon)$
and $S_2(\Gamma_0(2^4\cdot 3^3 \cdot 5^2),\varepsilon)$. By
Remark~\ref{rem:coefffield} the coefficient field of $\tilde{\rho}$ is
a degree eight extension of $\Q$.

\vspace{2pt}

\noindent $\bullet$ The space
$S_2(\Gamma_0(2^4\cdot 3^2 \cdot 5^2),\varepsilon)$ has $15$ conjugacy
classes, three of them with coefficient field $\Q(\sqrt{-1})$, another
three of them with coefficient field a quadratic extension of it, and
the last nine newforms with coefficient field a degree $4$ extension
of $\Q(\sqrt{-1})$ (so of degree eight over $\Q$). There are seven
forms with complex multiplication. Mazur's trick allows to discard all
newforms whose coefficient field does not have the right degree and
some extra ones, but there are four newforms without complex multiplication which pass
systematically Mazur's test, so we cannot reach a contradiction in
this particular case.

\vspace{2pt}

\noindent$\bullet$ The space
$S_2(\Gamma_0(2^4\cdot 3^3 \cdot 5^2),\varepsilon)$ has $24$ conjugacy
classes, six of them with complex multiplication (one of them corresponding to the trivial
solution by Proposition~\ref{prop:conductorCM2}). Once again, some
newforms without complex multiplication pass systematically Mazur's test.

Some partial results can still be obtained using more advanced
elimination techniques (see \cite{2109.07291}).

\subsubsection{The equation $x^2+6y^6=z^p$} In this case we can prove the following result.

\begin{thm}
  Let $p> 563$ be a prime number. Then there are no non-trivial primitive
  solutions of the equation
  \[
    x^2+6y^6=z^p.
  \]
  \label{thm:eq2d6}
  \end{thm}
\begin{proof}
  Let $(a,b,c)$ be a non-trivial primitive solution. Since $d$ is
  divisible by $6$, $c$ is prime to $6$ so in particular $c$ is
  divisible by a prime larger than $3$ and we are in the hypothesis of
  Ellenberg's large image result. In particular,
  Theorem~\ref{thm:ellenberg} implies that the
  residual image is absolutely irreducible for all primes
  $p > N_6 = 563$.

  Since $d$ is only divisible by the primes $2$ and $3$, all sets
  $Q_{\pm,\pm}$ are empty, so the character $\varepsilon$ equals the
  quadratic character of conductor
  $12$. Theorem~\ref{thm:levelandnebentypus2} implies the existence of
  a quadratic character $\chi$ of conductor
  $3 \cdot\langle 2,\sqrt{-6}\rangle^5$ and
  Corollary~\ref{coro:loweringlevelEt} implies that we need to compute
  the space $S_2(\Gamma_0(2^8\cdot 3^5),\varepsilon)$. Such space has
  dimension $1152$ and splits into 58 Galois conjugacy classes of newforms, eight of them having complex multiplication so Ellenberg's result implies that we can  discard them. We discard the remaining newforms as follows:
  
  \noindent $\bullet$ We run Mazur's trick for primes $5\le q\le 13$
  for the first (in \verb*|Magma|'s order) $43$ newforms without
  complex multiplication and it follows that $p\in\{2,3,5,11,7,37\}$.

\noindent $\bullet$ The last $7$ newforms have a large coefficient
field (of degrees $48$, $72$ and $144$) and for some unclear reason
\verb*|Magma| is unable to compute norms of elements in such large
fields. To overcome this problem, we used \verb*|Magma| to compute
the coefficients $a_5$ and $a_{11}$ of each of these forms and apply
Mazur's trick in \verb*|PARI/GP| for $q=5,11$ by hand (where the norms
are computed within a few seconds). It follows that $p\in\{2,11,13\}$.

\end{proof}
\begin{remark}
  If Ellenberg's constant in Theorem~\ref{thm:ellenberg} could be
  improved to $11$ (as we expect), the previous result would hold for
  $p>37$. It is probably the case that the result can even be improved
  to $p>17$ via adding extra primes $q$ into Mazur's test, but we did
  not pursue this objective because all involved computations are very
  time consuming (due to the fact that the coefficient fields have
  huge dimension).
\end{remark}

\subsubsection{The equation $x^2+7y^6=z^p$} The set $Q_{-+}=\{7\}$
while all other ones are empty. The character $\varepsilon$ has order
$2$ and conductor $21$, while $\chi$ has order $4$. Let $(a,b,c)$ be a
non-trivial primitive solution. Corollary~\ref{coro:loweringlevelEt}
implies the existence of a newform in one of the spaces
$S_2(\Gamma_0(2^t\cdot 3^s \cdot 7^2),\varepsilon)$, for $t=1,2$ and
$s=1,3$.

Note that the level valuation at $2$ equals $1$ when $c$ is even and
$2$ when $c$ is odd, by Lemma~\ref{lemma:case2conductor2}. In
particular, Ellenberg's large image result applies to all newforms in
$S_2(\Gamma_0(2^2\cdot 3^s \cdot 7^2),\varepsilon)$ (so we can discard
the newforms with complex multiplication in such spaces).

\vspace{2pt}

\noindent $\bullet$ The space
$S_2(\Gamma_0(2\cdot 3\cdot 7^2),\varepsilon)$ has $2$ Galois conjugacy
classes of newforms, but Mazur's trick only allows us to discard the
second one (in \verb*|Magma|'s order).

\vspace{3pt}

\noindent $\bullet$ The space
$S_2(\Gamma_0(2^2\cdot 3\cdot 7^2),\varepsilon)$ has $4$ Galois conjugacy
classes, one of them with complex multiplication (which can be discarded). Applying
Mazur's trick with primes $1\le q\le 20$ to the forms without complex multiplication we
can discard the three newforms if $p$ does not belong to the set $\{2,3,7\}$.

\vspace{3pt}

\noindent $\bullet$ The space
$S_2(\Gamma_0(2\cdot 3^3\cdot 7^2),\varepsilon)$ has $6$ Galois conjugacy
classes but only three of them can be proved to be not related to primitive
solutions using Mazur's trick. We cannot discard the remaining
three ones.

\vspace{3pt}

\noindent $\bullet$ The space
$S_2(\Gamma_0(2^2\cdot 3^3\cdot 7^2),\varepsilon)$ has $7$ Galois conjugacy
classes, three of them having complex multiplication (one coming from the trivial solution
by Proposition~\ref{prop:conductorCM2}). All forms having complex multiplication can be discarded
using Ellenberg's large image result. Applying Mazur's trick for primes
$1\le q\le 20$ we can discard the remaining ones if $p>7$.

Still, some partial results can also be obtained for this particular
equation using more advances elimination techniques (see
\cite{2109.07291}).
  \bibliographystyle{alpha}
  \bibliography{biblio}
\end{document}